\documentclass[12pt,oneside,reqno]{amsart}

\usepackage{imakeidx}
\usepackage{imakeidx}
\makeindex[columns=3]
\usepackage[utf8]{inputenc}
\usepackage{tgtermes}
\usepackage[english]{babel}
\usepackage[normalem]{ulem}
\usepackage{amsmath,
	mathtools,
	amssymb,
	graphicx,
	bbm,
	xcolor,
	amsthm,
	mathrsfs,
	imakeidx,
	stmaryrd,
	fullpage,
	mathdots}
\usepackage[all]{xy}
\usepackage[backref=page]{hyperref}
\hypersetup{}
\usepackage{hyperref}
\usepackage{tikz-cd}
\usetikzlibrary{decorations.pathmorphing}
\usepackage{dynkin-diagrams}
\usepackage{listings}
\usepackage{bm}
\theoremstyle{plain}

\newtheorem{theorem}{Theorem}[section]
\newtheorem{corollary}[theorem]{Corollary}
\newtheorem{definition}[theorem]{Definition}
\newtheorem{definition-proposition}[theorem]{Definition-Proposition}
\newtheorem{lemma}[theorem]{Lemma}
\newtheorem{proposition}[theorem]{Proposition}

\newtheorem{remark}[theorem]{Remark}

\newtheorem*{theorem-no-num}{Theorem}
\newtheorem*{proposition-no-num}{Proposition}
\newtheorem*{corollary-no-num}{Corollary}

\newcommand{\kbf}{\mathbf{k}}
\newcommand{\et}{\mathrm{\acute{e}t}}
\newcommand{\ket}{\mathrm{k\acute{e}t}}
\newcommand{\fl}{\mathrm{fl}}
\newcommand{\kfl}{\mathrm{kfl}}
\newcommand{\fs}{\mathrm{fs}}
\newcommand{\gp}{\mathrm{gp}}
\newcommand{\id}{\mathrm{Id}}
\newcommand{\Hom}{\mathrm{Hom}}
\newcommand{\Ext}{\mathrm{Ext}}
\newcommand{\Mod}{\mathrm{Mod}}
\newcommand{\Gal}{\mathrm{Gal}}
\newcommand{\cHom}{\mathcal{H}om}
\newcommand{\cExt}{\mathcal{E}xt}
\newcommand{\Spec}{\mathop{\mathrm{Spec}}}
\newcommand{\Gm}{\mathbb{G}_{\mathrm{m}}}
\newcommand{\Gml}{\mathbb{G}_{\mathrm{m,log}}}
\newcommand{\GL}{\mathrm{GL}}
\newcommand{\N}{\mathbb{N}}
\newcommand{\Z}{\mathbb{Z}}
\newcommand{\Q}{\mathbb{Q}}
\newcommand{\R}{\mathbb{R}}
\newcommand{\F}{\mathbb{F}}
\newcommand{\cC}{\mathcal{C}}
\newcommand{\cD}{\mathcal{D}}
\newcommand{\Ocal}{\mathcal{O}}
\newcommand{\Mcal}{\mathcal{M}}
\newcommand{\Qcal}{\mathcal{Q}}
\newcommand{\Mbf}{\mathbf{M}}
\newcommand{\LP}{\mathbf{LP}}
\newcommand{\qW}{q\text{-Weil}}
\newcommand{\coker}{\mathrm{coker}}
\newcommand{\Mor}{\mathrm{Mor}}
\newcommand{\MorGammaP}{\mathrm{Mor}_{\Gamma}^{P}}
\newcommand{\MorGammaPpPol}{\mathrm{Mor}_{\Gamma}^{P,\mathrm{pPol}}}
\newcommand{\MorGammaNpPol}{\mathrm{Mor}_{\Gamma}^{\mathbb{N},\mathrm{pPol}}}
\newcommand{\MorGammaN}{\mathrm{Mor}_{\Gamma}^{\mathbb{N}}}

\DeclareFontFamily{U}{wncy}{}
\DeclareFontShape{U}{wncy}{m}{n}{<->wncyr10}{}
\DeclareSymbolFont{mcy}{U}{wncy}{m}{n}
\DeclareMathSymbol{\Sha}{\mathord}{mcy}{"58}

\makeatletter
\@namedef{subjclassname@2020}{%
	\textup{2020} Mathematics Subject Classification}
\makeatother

\makeindex

\begin{document}
	\title{Honda-Tate theory for log abelian varieties over finite fields}
	\author{Xiaoyu Zhang, Heer Zhao}
        
	\address{Universität Duisburg-Essen, Fakultät für Mathematik, Mathematikcarrée, Thea-Leymann-Straße 9, 45127 Essen, Germany}
        \email{xiaoyu.zhang@uni-due.de}

        \address{Harbin Institute of Technology, Institute for Advanced Study in Mathematics, 150001 Harbin, China}
        \email{heer.zhao@gmail.com}

        \begin{abstract}
		In this article we study the Honda-Tate theory for log abelian varieties over an fs log point $S=(\Spec\kbf,M_S)$ for $\kbf=\F_q$ a finite field, generalizing the classical Honda-Tate theory for abelian varieties over $\kbf$. For the standard log point $S$, we give a complete description of the isogeny classes of such log abelian varieties using Weil $q$-numbers of weight 0,1, and 2. In the general case where $M_S$ admits a global chart $P\to\kbf$ with $P=\N^k$, we also give a complete description of simple isogeny classes of log abelian varieties over $S$ in terms of rational points in generalized simplices.
	\end{abstract}

	\maketitle

	\tableofcontents

	\section{Introduction}
	Honda-Tate theory is an important result in the theory of abelian varieties over finite fields (\cite{Honda1968,Tate1966}). It gives a complete description of the isogeny classes of simple abelian varieties over a finite field $\kbf=\mathbb{F}_q$ with $q$ elements in terms of Galois conjugacy classes of Weil $q$-numbers (of weight 1).
	It has many deep applications, for example, in the proof of the Manin conjecture on the Newton polygons of abelian varieties by Honda and Serre(\cite{Tate1968}); in the point counting problems on Shimura varieties (\cite{Kottwitz1992}). There are also many generalizations and refinements of this theory, see for example \cite{KisinMadapusiShin2022}.

	Recall that for an integer $w$, a \emph{Weil $q$-number of weight $w$} is an algebraic integer whose images in $\mathbb{C}$ all have absolute value equal to $q^{w/2}$.	
	In particular, a Weil $q$-number of weight $0$ is a root of unity by a classical result of Kronecker on algebraic integers (\cite{Kronecker1857}, see also \cite[Theorem 4.5.4]{Prasolov2004} for a convenient reference). We write
	\[
	\qW(w)
	\]
	for the set of Galois conjugacy classes $[\alpha]$ of Weil $q$-numbers $\alpha$ of weight $w$.
	For an abelian variety $A$ over $\kbf$, we write $[A]$ for its isogeny class and 
        \[P_{A,\pi_A}\]
    for the characteristic polynomial of the $q$-Frobenius endomorphism $\pi_A$ on $A$ (or equivalently on any member of $[A]$), see \cite[\S10, p.48]{Milne2008}.
    Then the theory of Honda-Tate states as follows:
	\begin{theorem}\label{Honda-Tate for AVs}
		Let $\kbf=\mathbb{F}_q$. There is a bijection between the set of isogeny classes of simple abelian varieties over $\kbf$ and the set of Galois conjugacy classes of Weil $q$-numbers of weight $1$
		\begin{align*}
			\begin{split}
				\left\{
				\text{simple abelian varieties over }
				\kbf
				\right\}/\text{isogeny}
				\rightarrow
				&
				\qW(1)
				\\
				[A]
				\mapsto
				&
				[\alpha],
			\end{split}
		\end{align*} 
                where $\alpha$ is a root of $P_{A,\pi_A}$.
	\end{theorem}

    Throughout this introduction we write 
    \[\Gamma:=\mathrm{Gal}(\overline{\kbf}/\kbf)=\overline{\langle\gamma\rangle}\]
    for the absolute Galois group of $\kbf=\F_q$ with $\gamma:\kbf\to\kbf,\, x\mapsto x^q$
    the canonical topological generator.    
    The characteristic polynomial $P_{A,\pi_A}$ of $\pi_A$ can also be obtained as the characteristic polynomial of $\gamma$ on the $\ell$-adic Tate module $T_{\ell}(A)$ of $A$, where $\ell\neq\mathrm{char}(\kbf)$ is a prime number (see \cite[Chap. II, \S2, the paragraph after Theorem 2.9]{Milne2008}).
    This fact will be important for dealing with log abelian varieties (see below), for which we are only able to use $\ell$-adic Tate module to define a polynomial which plays the role of $P_{A,\pi_A}$ for an abelian variety $A$, see Definition \ref{char poly of log av} and \eqref{equation char poly of Frobenius on log 1-motive}.

	Log abelian varieties, developed by Kajiwara, Kato and Nakayama in their series of papers starting with \cite{KajiwaraKatoNakayama2008}, are proper (log) smooth group objects which degenerate abelian varieties. They are expected to behave like abelian varieties in many aspects. In particular, it would be desirable to classify log abelian varieties up to isogeny over an fs log point $S=(\Spec\kbf,M_S)$ with $\kbf$ a finite field, just as in the case of abelian varieties over $\kbf$.

	Our first main result reduces the problem to the classification of log abelian varieties without abelian part, and the latter is essentially a problem of Galois modules over $\kbf$.

        \begin{theorem}[See Theorem \ref{isogeneous decomposition of log AV}]\label{main theorem-0}
            Let $S=(\Spec \kbf,M_S)$ be an fs log point with $\kbf=\F_q$ finite. Then for any log abelian variety $A$ over $S$, we have an isogeny
            \[A\sim A_1\times B,\]
            where $B$ is an abelian variety over $\kbf$ and $A_1$ is a log abelian variety over $S$ whose corresponding log 1-motive $\Mbf_1$  (see \cite[Theorem 3.4]{KajiwaraKatoNakayama2008}) has no abelian part, i.e. $\Mbf_1=[Y_1\to T_{1\log}]$ with $T_1$ a torus over $\kbf$.
        \end{theorem}

        By Theorem \ref{main theorem-0}, we have
        \begin{align*}
            &\left\{
		\text{simple LAVs over }S
		\right\}/\text{isogeny}  \\
           =&\left\{
		\text{simple AVs over }\kbf
		\right\}/\text{isogeny}\bigsqcup \left\{
		\text{simple LAVs without abelian part over $S$}
		\right\}/\text{isogeny},
        \end{align*}
        where AV (resp. LAV) stands for abelian variety (resp. log abelian variety).
        Therefore in order to classify log abelian varieties over $S$ up to isogeny, it suffices to classify log abelian varieties without abelian part over $S$ up to isogeny.
        
    We focus on the case that $S$ is a finite charted fs log point, i.e.
    \begin{equation}\label{charted log point}
    \text{\parbox{.85\textwidth}{$S=(\Spec \kbf,M_S)$ is an fs log point such that $\kbf=\F_q$ and $M_S$ admits a global chart $P\to \kbf,a\mapsto \begin{cases}1,&\text{ if $a=0$}\\0,&\text{ otherwise}\end{cases}$  with $P$ a sharp fs monoid.}}
    \end{equation}
    Here we use additive notation for $P$. 
    
	In the special case of $S$ being a \textbf{standard log point}, i.e. $P=\N$ in (\ref{charted log point}),  we have a classification theorem of log abelian varieties over $S$ up to isogeny parallel to Theorem \ref{Honda-Tate for AVs}. To state the theorem, we need the following definition: let $Y$ be a group scheme over $\kbf$ which is \'etale locally isomorphic to the constant group scheme $\Z^r$ for some $r>0$. Then $\gamma\in\Gamma$ gives rise to a $\Z$-linear automorphism $\pi_Y$ of $\Z^r$, and we call the characteristic polynomial of $\pi_Y$ the \emph{characteristic polynomial of the geometric Frobenius on $Y$}, and denote it by 
    \[P_{Y,\pi_Y}\]
    (see §\ref{1-motives} for more details).
 
	\begin{theorem}[See Theorem \ref{classificaton of isogenous classes of log Gamma-module of rank 1}]\label{main theorem-1}
	Let $S$ be as in \eqref{charted log point} with $P=\N$, i.e. a standard log point. Let
    \begin{equation}\label{q-Weil(0,2)}
        \qW(0,2):=\{([\alpha],[q\alpha^{-1}])\mid [\alpha]\in \qW(0)\},
    \end{equation}
    which is a subset of $\qW(0)\times\qW(2)$.
    Then the bijection of sets from Theorem \ref{Honda-Tate for AVs} extends to a bijection of sets (as in the second row of the commutative digram below)
            \begin{equation}\label{main theorem-1-diagram}
                \xymatrix{
            \left\{
		\text{simple AVs over $\kbf$}
		\right\}/\text{isogeny} \ar[r]^-{\simeq}\ar@{^(->}[d] &\qW(1)\ar@{^(->}[d] \\
            \left\{
		\text{simple LAVs over }S
		\right\}/\text{isogeny} \ar[r]^-{\simeq} &\qW(0,2)\bigsqcup\qW(1)
            }
            \end{equation}
		such that any simple log abelian varieties without abelian part, whose corresponding log 1-motive (see \cite[Theorem 3.4]{KajiwaraKatoNakayama2008}) is $[Y\to T_{\log}]$, is mapped to $([\alpha],[q\alpha^{-1}])\in \qW(0,2)$ with $\alpha$ a root of the characteristic polynomial $P_{Y,\pi_Y}$. 
	\end{theorem}

    By Theorem \ref{main theorem-0} and Theorem \ref{main theorem-1}, the Weil $q$-numbers associated to a general log abelian variety could have weights 0, 1 and 2 simultaneously, a feature typically exclusive to mixed motives in the classical picture. On the other hand, a log abelian variety, as an analogue of an abelian variety, should give rise to a pure log motive. Such a phenomena gives an example of the idea from \cite{KatoNakayamaUsui2022} that mixed motives should be embedded into log pure motives.
 
    The upper bijection of \eqref{main theorem-1-diagram} is defined via the characteristic polynomial $P_{B,\pi_B}$ of the $q$-Frobenius endomorphism $\pi_B$ of an abelian variety $B$. Note that the linear map $T_{\ell}(\pi_B)$ on $T_{\ell}(B)$ induced by $\pi_B$ agrees with the canonical linear map given by $\gamma\in\Gamma$ on $T_{\ell}(B)$. For a log abelian variety $A$ over $S$, we can also define a certain ``characteristic polynomial'' $P_{A,\gamma}$ of  $\gamma\in\Gamma$, in analogue to $P_{B,\pi_{B}}$ (see Definition \ref{char poly of log av}). However if $A$ is not an abelian variety, $P_{A,\gamma}$ has always more that one irreducible factor, even if $A$ is simple. Nevertheless under the setting of Theorem \ref{main theorem-1}, for a simple log abelian variety $A$ corresponding to $\Mbf=[Y\to T_{\log}]$, we have
    \[P_{A,\gamma}(\theta)
    =
    P_{Y,\pi_Y}(\theta)\cdot P_{T,\pi_T}(\theta)
    =
    P_{Y,\pi_Y}(\theta)\cdot
    \left(
    \frac{(-\theta)^d}{\det(\pi_Y)}\cdot P_{Y,\pi_Y}(\frac{q}{\theta})
    \right),\]
    where $P_{Y,\pi_Y}$ is irreducible and $d$ is the degree of $P_{Y,\pi_Y}(\theta)$, and the identification $P_{T,\pi_T}(\theta)=\frac{(-\theta)^d}{\det(\pi_Y)}\cdot P_{Y,\pi_Y}(\frac{q}{\theta})$ follows from the fact that $Y$ is isogenous to $X:=\cHom_{\kbf}(T,\Gm)$ which follows from that $\Mbf$ is pointwise polarized.   Then the lower bijection beyond the upper bijection in (\ref{main theorem-1-diagram}) is just the map sending a simple log abelian variety without abelian part $A$ over $S$ to $([\alpha],[\beta])\in \qW(0,2)$ with $\alpha$ (resp. $\beta$) a root of $P_{Y,\pi_Y}(\theta)$ (resp. $\theta^d\cdot P_{Y,\pi_Y}(\frac{q}{\theta})$) (see §\ref{1-motives}, \S\ref{Log 1-motives} and \S\ref{Log abelian varieties with constant degeneration} for more details). Of course, here one can take $\beta=\frac{q}{\alpha}$.

    Now consider the case $S=(\Spec \kbf,M_S)$ with $P$ a free monoid $P\cong\mathbb{N}^k$ as in (\ref{charted log point}). For each positive integer $r$, we define a subset $T_{r,k}$ of $(\Q(\zeta_{r})^+)^k$ (\textit{cf}. Definition \ref{T_{r,k}}), which is in some sense a generalization of (rational points in) topological $(k-1)$-simplex. Then we have
	\begin{theorem}[Theorem \ref{classificaton of isogenous classes of log Gamma-module of higher ranks}]
		\label{main theorem-2}
		Let $S$ be as in \eqref{charted log point} with $P=\N^k$. Then the bijection of sets from Theorem \ref{Honda-Tate for AVs} extends to a bijection of sets (as in the second row of the commutative digram below)
            \[
            \begin{tikzcd}
                \left\{
		\text{simple AVs over $S$}
		\right\}/\text{isogeny} 
            \arrow[r,"\simeq"]
            \arrow[d,hookrightarrow]
            &
            \qW(1)
            \arrow[d,hookrightarrow]
            \\
            \left\{
            \begin{tabular}{c}
                simple LAVs $A$ over S with $Y_1=0$
                \\
                or a simple free $\Gamma$-module
            \end{tabular}
            \right\}\big/\text{isogeny}
            \arrow[r,"\simeq"]
            &
            \left(
		\bigsqcup_{r>0}T_{r,k}
		\right)
		\bigsqcup
		\qW(1)
            \end{tikzcd}
            \]		
            under which the elements of the set $\bigsqcup_{r>0}T_{r,k}$ correspond to the log abelian varieties $A$ without abelian part. Here $Y_1$ is as in Theorem \ref{main theorem-0}, and we refer to Definition \ref{simple free Gamma-modules} for the terminology ``simple free $\Gamma$-module''.
	\end{theorem}
        We also have a classification result for log abelian varieties $A$ over $S$ with $Y_1$ of the form $Y_1=\widetilde{Y}_1^a$ where $\widetilde{Y}_1$ is a simple free $\Gamma$-module and $a$ is a positive integer, see Theorem \ref{classificaton of isogenous classes of log Gamma-module of higher ranks-2}. We classify these log abelian varieties in terms of $T_{r,k}^{(a)}$, which is in some sense a higher dimensional generalization of $T_{r,k}$ (see Definition \ref{T_{r,k}}). Moreover, we also give a very explicit criterion for the simpleness of these $A$. We refer to §\ref{Isogeny classes of Z^a} for more details. The \emph{upshot} is that any simple log abelian variety without abelian part is isogenous to such a log abelian variety, see Proposition \ref{isogeny classes in Mor_Gamma^{pPol}} (3), therefore we actually get a description of all the isogeny classes of \emph{simple} log abelian varieties for the case $P=\N^k$.

	For $k=1$, $T_{r,k}$ has only one element. If we identify $T_{r,1}$ with $([\zeta_r],[q\zeta_r^{-1}])\in \qW(0,2)$ for all $r>0$, then Theorem \ref{main theorem-2} gives rise to Theorem \ref{main theorem-1}.  We remark that for $k>1$, the description of a simple log abelian variety $A$ over $S$ with trivial abelian part is much more complicated than the case $k=1$, and has no such simple characterization as using Weil $q$-numbers (of certain weights).
	
    \subsection*{Outline}	
    In §\ref{Review of log abelian varieties with CD} we present a brief review of the notion of log abelian varieties (with constant degeneration), which is the main object of study of this article. In §\ref{1-motives}, the notion of 1-motives is recalled and when the base is a field $\kbf$, we define the characteristic polynomial of an endomorphism of a 1-motive and in particular study various properties of the Froenius morphism for $k$ a finite field. In §§\ref{Log 1-motives} and \ref{Log abelian varieties with constant degeneration}, similar studies are carried out for log 1-motives and log abelian varieties. In §\ref{section lattice pairings}, we introduce the notion of lattice pairings, which is a two-term complex of certain $\Gamma$-modules. We then relate the category of lattice pairings to the category of log 1-motives without abelian part, and further to the category of log abelian varieties without abelian part. In the remaning part of the article, we concentrate on a finite log point $S$ (here ``finite'' means that the underlying field $\kbf$ of $S$ is finite). In §\ref{log av over finite log points}, we study the isogeny category of log abelian varieties over $S$. In §\ref{Decomposition of log abelian varieties up to isogeny over finite log points}, we show that a log abelian variety $A$ over $S$ is isogenous to $A_1\times B$ with $A_1$ a log abelian variety without abelian part over $S$ and $B$ an abelian variety over $\kbf$. In passing we also establish the Honda-Tate theorem for 1-motives (Theorem \ref{Honda-Tate theorem for 1-motives}). In §\ref{subsection LAV without abelian part over finite log points are classified by LP up to isogeny}, we establish an equivalence between the isogeny category of pointwise polarizable lattice pairings and the isogeny category of log abelian varieties without abelian part over $S$. So the classification of isogeny classes of log abelian varieties over $S$ is reduced to the classification of isogeny classes of pointwise polarizable lattice pairings, the latter of which is the goal for the remaining two technical sections of this article. In §\ref{Isogeny classes in MorGammaPpPol over the standard finite log point}, we consider the case $S$ a standard finite log point and we give a complete description of simple lattice pairings up to isogeny (Theorem \ref{classificaton of isogenous classes of log Gamma-module of rank 1}). This gives Theorem \ref{main theorem-1}. We also classify all possible polarizations on lattice pairings as a by-product (Proposition \ref{polarization on (Z[zeta_r],Z[zeta_r]^vee,Id)}). In §\ref{Isogeny classes in MorGammaPpPol over a general finite log point}, we consider the case of a general finite log point $S$ with $P=\N^k$ in \eqref{charted log point}. Based on the preceding section, we describe such lattice pairings using simplices inside cyclotomic fields and their generalizations (Theorems \ref{classificaton of isogenous classes of log Gamma-module of higher ranks} and \ref{classificaton of isogenous classes of log Gamma-module of higher ranks-2}). This gives Theorem \ref{main theorem-2}.

    \subsection*{Notations}
    For a positive integer $n$ and a set $T$, we write $\mathrm{Mat}_n(T)$ for the set of $n\times n$-matrices with entries in $T$. For a matrix $A$, we write $A^\mathrm{t}$ for its transpose.

    For a commutative group scheme $G$ over a base scheme $S$ and an integer $n$, we write $n_G$ for the homomorphism $G\to G$ of multiplication-by-$n$ and $G[n]$ for the kernel of $n_G$.

	\section{Review of log abelian varieties with constant degeneration}\label{Review of log abelian varieties with CD}
	For the readers' convenience, we present a short introduction to log abelian varieties with constant degeneration in this section. For details, we refer to \cite{KajiwaraKatoNakayama2008} and \cite{Zhao2017}.

    \subsection{1-motives}\label{1-motives}
    We review in this subsection the theory of 1-motives.
    Let $S$ be a scheme. Let $\cC_{\fl}$ be the category of sheaves of abelian groups on the fppf site (see \cite[\href{https://stacks.math.columbia.edu/tag/021S}{Tag 021S}]{Stacks-project}) of $S$, and let $\cD_{\fl}$ be the derived category of $\cC_\fl$. 
    \begin{definition}\label{category of 1-motives}
        Let $Y$ be a group scheme over $S$, which is \'etale locally a free abelian group of finite rank, and let $G$ be an extension of an abelian scheme $B$ by a torus $T$ over $S$. A \emph{1-motive with lattice part $Y$ and semi-abelian part $G$} over $S$, is a two term complex 
        \[\Mbf=[Y\xrightarrow{u} G]\] 
        in $\cC_\fl$, where $Y$ sits in degree -1. A \emph{morphism} of 1-motives over $S$ is just a morphism of complexes. We denote the category of 1-motives over $S$ by $\Mcal_1$.

        A morphism of 1-motives $(f_{-1},f_0):[Y\xrightarrow{u} G]\to [Y'\xrightarrow{u'} G']$ is called an \emph{isogeny}, if $f_{-1}:Y\to Y'$ is injective with finite cokernel and $f_0:G\to G'$ is surjective with finite kernel.
    \end{definition}

    \begin{definition}
        Let $S$ be a scheme, and let $\Mbf=[Y\xrightarrow{u} G]$ be a 1-motive over $S$. For any positive integer $n$, let 
        \[T_{\Z/n\Z}(\Mbf):=H^{-1}(\Mbf\otimes^L\Z/n\Z).\]
        For any prime number $\ell$, let 
        \[T_{\ell}(\Mbf):=\varprojlim_rT_{\Z/\ell^r\Z}(\Mbf).\]
    \end{definition}

    The canonical distinguished triangle $Y\xrightarrow{u}G\to\Mbf\to Y[1]$ in $\cD_\fl$ gives rise to a long exact sequence
    \begin{align*}
        H^{-1}(Y\otimes^L\Z/n\Z)\to H^{-1}(G\otimes^L\Z/n\Z)\to H^{-1}(\Mbf\otimes^L\Z/n\Z)\to  \\
        H^{0}(Y\otimes^L\Z/n\Z)\to H^{0}(G\otimes^L\Z/n\Z)\to\cdots .
    \end{align*}
    Here $Y[1]$ denotes the shift by 1 of $Y$ as a complex concentrated in degree 0. This should not lead to any confusion with the $n$-torsion subgroup scheme $G[n]$ of a semi-abelian scheme $G$, since we never consider the torsion subgroups of the torsion-free group $Y$. Since $Y$ is torsion-free and $n_G$ is surjective, we get a short exact sequence
    \[0\to G[n]\to T_{\Z/n\Z}(\Mbf)\to Y/nY\to 0\]
    in $\cC_\fl$. Both $G[n]$ and $Y/nY$ are finite locally free group schemes over $S$. By descent, $T_{\Z/n\Z}(\Mbf)$ as a $G[n]$-torsor over $Y/nY$ is also a finite locally free group scheme over $S$. If $n$ is invertible on $S$, $T_{\Z/n\Z}(\Mbf)$ is \'etale over $S$. Since $(G[\ell^i])_i$ is a Mittag-Leffler system, we also have a short exact sequence
    \begin{equation}\label{l-adic s.e.s. of 1-motive}
        0\to T_{\ell}(G)\to T_{\ell}(\Mbf)\to Y\otimes_{\Z}\Z_{\ell}\to 0
    \end{equation}
    in $\cC_\fl$. If $\ell$ is invertible on $S$ and $S$ is connected, then the sequence (\ref{l-adic s.e.s. of 1-motive}) is a short exact sequence of finite rank free $\Z_\ell$-modules endowed with a continuous $\pi_1^{\et}(S)$-action, where $\pi_1^{\et}(S)$ denotes the \'etale fundamental group of $S$.

    In the rest of this subsection, the base $S$ will be $\Spec\kbf$ with $\kbf$ a field. We will use the notation $\Gamma:=\Gal(\overline{\kbf}/\kbf)$ for $\pi_1^{\et}(S)$.

    \begin{definition}\label{def simple 1-motives}
            A non-zero 1-motive $[Y\to G]$ over $\kbf$ is called \emph{simple}, if any non-zero morphism
            \[
            (f_{-1},f_0)\colon [Y'\to G']\to[Y\to G],
            \]
            which is injective as a morphism of complexes, has to be an isogeny.
    \end{definition}

    \begin{remark}\label{simple 1-motives}
        Clearly a simple 1-motive over $\kbf$ must be of the form $Y[1]$, or $T$, or $B$, where
    \begin{enumerate}
        \item $Y$ is a \emph{simple lattice}, i.e. it contains no proper sub-lattice,
        \item $T$ is a \emph{simple torus}, i.e. it contains no proper sub-torus,
        \item and $B$ is a simple abelian variety.
    \end{enumerate}
    \end{remark}

    \begin{proposition}\label{char poly has integral coefficients and is indep of the choice of the prime}
        Let $f=(f_{-1},f_0)$ be an endomorphism of a 1-motive $\Mbf=[Y\to G]$ over $\kbf$. Let $\ell$ be a prime number which is invertible in $\kbf$, and let $P_{\Mbf,f,\ell}(\theta)\in\Z_\ell[\theta]$ be the characteristic polynomial of the endomorphism $T_{\ell}(f)$ on $T_\ell(\Mbf)$.     
        In case that $\Mbf=Y[1]$, we have $f=(f_{-1},0)$ and also write $P_{\Mbf,f}$ as $P_{Y,f_{-1}}$. Similarly, in case that $\Mbf=G$, we also write $P_{\Mbf,f}$ as $P_{T,f_0}$.
        Then we have the following.
        \begin{enumerate}
            \item $P_{\Mbf,f,\ell}[\theta]=P_{Y,f_{-1},\ell}[\theta]\cdot P_{G,f_0,\ell}[\theta]=P_{Y,f_{-1},\ell}[\theta]\cdot P_{B,f_{\mathrm{ab}},\ell}[\theta]\cdot P_{T,f_{\mathrm{t}},\ell}[\theta]$, where $f_{\mathrm{ab}}:B\to B$ (resp. $f_{\mathrm{t}}:T\to T$) denotes the abelian (resp. torus) part of $f_0$ (see Lemma \ref{hom of semiabelian schemes} (1)).
            \item $P_{\Mbf,f,\ell}[\theta]\in\Z[\theta]$.
            \item $P_{\Mbf,f,\ell}[\theta]$ is independent of the choice of $\ell$.
        \end{enumerate}
    \end{proposition}
    \begin{proof}
        \begin{enumerate}
            \item
            After choosing $\mathbb{Z}_\ell$-basis for $T_\ell(Y)$, $T_\ell(T)$ and $T_\ell(B)$, it is easy to see that the endomorphism $T_\ell(f)$ is represented by a matrix of the following form
            \[
            \begin{pmatrix}
                a_{11} & a_{12} & a_{13} \\
                0 & a_{22} & a_{23} \\
                0 & 0 & a_{33}
            \end{pmatrix}
            \]
            with $a_{11}\in\GL_{\mathrm{dim}(T)}(\Z_\ell)$, $a_{22}\in\GL_{2\mathrm{dim}(B)}(\Z_\ell)$ and
            $a_{33}\in\GL_{\mathrm{rk}(Y)}(\Z_\ell)$. Thus the characteristic polynomial of $T_\ell(f)$ is the product of the characteristic polynomials of $a_{11}$, $a_{22}$ and $a_{33}$, which are respectively the characteristic polynomials of $T_\ell(f_{\mathrm{t}})$, $T_\ell(f_{\mathrm{ab}})$ and $T_\ell(f_{-1})$.

            \item 
            It is well known that $P_{B,f_{\mathrm{ab}},\ell}\in\mathbb{Z}[\theta]$ (see \cite[Chap. IV, \S19, Theorem 4]{Mumford1970}). Write $X=\cHom(T,\mathbb{G}_m)$ for the character group of $T$, which is \'{e}tale locally a free abelian group of finite rank. The endomorphism $f_{\mathrm{t}}$ on $T$ induces an endomorphism $f_{\mathrm{t}}^\ast$ on $X$. Moreover the canonical pairing $\langle-,-\rangle\colon X\times T\to\mathbb{G}_m$ satisfies
            \[
            \langle f_{\mathrm{t}}^\ast(x),t\rangle
            =
            \langle x,f_{\mathrm{t}}(t)\rangle,
            \quad
            \forall
            x\in X(\overline{\kbf}),\,
            t\in T(\overline{\kbf}).
            \]
            We have an induced pairing $\langle-,-\rangle\colon T_{\ell}(X[1])\times T_\ell(T)=(X\otimes\Z_{\ell})\times T_\ell(T)\to T_\ell(\mathbb{G}_m)=\mathbb{Z}_\ell(1)$ such that
            \[
            \langle (f_{\mathrm{t}}^\ast\otimes\id_{\Z_\ell})(x),t\rangle
            =
            \langle x,T_\ell(f_{\mathrm{t}})(t)\rangle,
            \quad
            \forall
            x\in X\otimes\Z_{\ell},\,
            t\in T_\ell(T).
            \]
            After fixing $\mathbb{Z}_\ell$-basis for $X\otimes\Z_{\ell}$ and $T_\ell(T)$, we see immediately that the matrices representing $f_{\mathrm{t}}^\ast\otimes\id_{\Z_\ell}$ and $T_\ell(f_{\mathrm{t}})$ are transpose of each other. In particular,
            \[
            P_{T,f_{\mathrm{t}},\ell}
            =
            P_{X,f_{\mathrm{t}}^\ast,\ell}.
            \]
            So it is enough to show that $P_{Y,f_{-1},\ell}\in\mathbb{Z}[\theta]$ for any $Y$ and any endomorphism $f_{-1}$ on $Y$. However, we have
            \[
            T_\ell(Y[1])=Y\otimes_{\mathbb{Z}}\mathbb{Z}_\ell,
            \quad
            T_\ell(f_{-1})=f_{-1}\otimes\id_{\Z_\ell}.
            \]
            So $P_{Y,f_{-1},\ell}$ is exactly the characteristic polynomial of $f_{-1}$ on $Y$, which clearly lies in $\mathbb{Z}[\theta]$.

            \item 
            Again it is well known that $P_{B,f_{\mathrm{ab}},\ell}$ is independent of $\ell$ (see \emph{loc.cit.}). From the above argument, it is enough to show that $P_{Y,f_{-1},\ell}$ is independent of $\ell$ for any $Y$ and any $f_{-1}$, which again follows from the fact that $T_\ell(f_{-1})=f_{-1}\otimes\mathrm{Id}_{\Z_\ell}$.
        \end{enumerate}
    \end{proof}

    \begin{definition}\label{definition of degree and char poly for 1-motives}
        Let $f=(f_{-1},f_0)$ be an endomorphism of a 1-motive $\Mbf=[Y\xrightarrow{u} G]$ over $\kbf$. By Proposition \ref{char poly has integral coefficients and is indep of the choice of the prime}, 
        \[P_{\Mbf,f}(\theta):=P_{\Mbf,f,\ell}(\theta)\in \Z[\theta]\]
        is independent of the choice of $\ell$, and we call it the \emph{characteristic polynomial of $f$}.
    \end{definition}

    Now assume that $\kbf=\F_q$ is finite, then $\Gamma$ is topologically generated by the $q$-Frobenius map:
    \begin{equation}\label{description of Galois gp of finite field}        \Gamma=\Gal(\overline{\kbf}/\kbf)=\overline{\langle\gamma\rangle}
        \text{ with } \gamma:\overline{\kbf}\to\overline{\kbf},x\mapsto x^q.
    \end{equation}
    The lattice $Y$ over $\kbf$ amounts to a $\Z$-representation $\rho_Y:\Gamma\to\mathrm{GL}(\Z^r)$ for some $r>0$, where we identify $Y\times_{\Spec\kbf}\Spec\overline{\kbf}$ with $\Z^r$. Since $\Gamma$ is commutative, the linear map $\rho_Y(\gamma):\Z^r\to\Z^r$ gives rise to an endomorphism of $Y$, and we denote it by $\pi_Y:Y\to Y$. Let $\pi_G:G\to G$ (resp. $\pi_B:B\to B$, resp. $\pi_T:T\to T$) be the geometric Frobenius morphism of $G$ (resp. $B$, resp. $T$) over $\kbf$, see \cite[Chap. II, the beginning of \S1]{Milne2008}. Actually $\pi_Y$ can be regarded as the geometric Frobenius of the scheme $Y$ over $\kbf$. One can check easily that $\pi_G\circ u=u\circ\pi_Y$, and thus $\pi_{\Mbf}=(\pi_Y,\pi_G)$ defines an endomorphism of $\Mbf$.

    \begin{lemma}\label{charpoly of lattices, AVs and tori}
        Let $\Mbf=[Y\xrightarrow{u}G]$ be a 1-motive over the finite field $\kbf=\F_q$, let $B$ (resp. $T$) be the abelian (resp. torus) part of $G$, and let $X$ be the character group of $T$.
        \begin{enumerate}
            \item The roots of the characteristic polynomial $P_{Y,\pi_Y}(\theta)$ (resp. $P_{X,\pi_X}(\theta)$) of the endomorphism $\pi_Y$ (resp. $\pi_X$) of $Y$ (resp. $X$) are roots of unity. In particular they are Weil $q$-numbers of weight 0.
            
            \item For the characteristic polynomial $P_{T,\pi_T}(\theta)$ of the endomorphism $\pi_T$ of $T$, we have
            \[P_{T,\pi_T}(\theta)=\frac{(-\theta)^n}{\det(\pi_X)}P_{X,\pi_X}(\frac{q}{\theta}),\] 
            where $n$ denotes the rank of $X$ and $\det(\pi_X)$ denotes the determinant of the linear map $\pi_X$ on $X$ regarded as a finite rank free abelian group.
            In particular, the roots of $P_{T,\pi_T}(\theta)$ are of the form $q\alpha^{-1}$ with $\alpha$ a root of $P_{X,\pi_X}(\theta)$, and thus they are Weil $q$-numbers of weight 2.
            
            \item
            The roots of the characteristic polynomial $P_{B,\pi_B}(\theta)$ of the endomorphism $\pi_B$ of $B$ are Weil $q$-numbers of weight 1.
            
            \item
            $P_{\Mbf,\pi_{\Mbf}}(\theta)=P_{Y,\pi_Y}(\theta)\cdot P_{B,\pi_B}(\theta)\cdot P_{T,\pi_T}(\theta)$.
            
            \item
            Let $\rho_{\Mbf,\ell}$ (resp. $\rho_{G,\ell}$, resp. $\rho_{Y,\ell}$) denote the canonical action of $\Gamma$ on $T_{\ell}(\Mbf)$ (resp. $T_{\ell}(G)$, resp. $T_{\ell}(Y[1])=Y\otimes_{\Z}\Z_{\ell}$). Then 
            \[
            T_{\ell}(\pi_{G})=\rho_{G,\ell}(\gamma),
            \quad
            T_{\ell}(\pi_{Y})=\pi_Y\otimes\Z_{\ell}=\rho_{Y,\ell}(\gamma),
            \quad
            T_{\ell}(\pi_{\Mbf})=\rho_{\Mbf,\ell}(\gamma).
            \]
            In particular, $P_{\Mbf,\pi_{\Mbf}}(\theta)$ is the characteristic polynomial of $\rho_{\Mbf,\ell}(\gamma)$. 
        \end{enumerate}
    \end{lemma}
    \begin{proof}
        \begin{enumerate}
            \item
            Since the action of $\Gamma\simeq\widehat{\mathbb{Z}}$ on $Y[1]$ factors through the finite quotient $\mathbb{Z}/m\mathbb{Z}$ for some $m>0$, $\pi_Y^m=\mathrm{Id}$, thus the roots of $P_{Y,\pi_Y}$ are roots of unity and are Weil $q$-numbers of weight $0$. The same holds for $X$.

            \item 
            Write $\pi_T^\ast$ for the endomorphism on $X$ induced by $\pi_T$. We claim that
            $\pi_T^\ast=q\cdot\pi_X^{-1}$.
            The canonical pairing $\langle-,-\rangle\colon X\times T\to\mathbb{G}_m$ is Galois equivariant. In particular, for any $x\in X(\overline{\kbf})$ and $t\in T(\overline{\kbf})$, we have
            \begin{align*}
           \pi_T^*(x)(t)=&\langle x,\pi_T(t)\rangle=\langle \pi_X(\pi_X^{-1}(x)),\pi_T(t)\rangle=\langle \pi_X^{-1}(x),t\rangle^q \\
           =&\langle \pi_X^{-1}(x)^q,t\rangle=(q\cdot\pi_X^{-1}(x))(t), 
           \end{align*}
           and thus $\pi_T^*(x)=q\cdot\pi_X^{-1}(x)$. 
           Then we have
           \begin{align*}
           P_{T,\pi_T}(\theta)
           &=
           \det(\theta-\pi_T)
           =
           \det(\theta-\pi_T^\ast)
           =
           \det(\theta-q\pi_X^{-1})
           \\
           &=
           \det(-\theta\pi_X^{-1})\det(\frac{q}{\theta}-\pi_X)
           =
           \frac{(-\theta)^n}{\det(\pi_X)}
           P_{X,\pi_X}(\frac{q}{\theta}).
           \end{align*}

           \item 
           This follows from the classical Honda-Tate theory.

           \item 
           This follows from Proposition \ref{char poly has integral coefficients and is indep of the choice of the prime} (1).

           \item
           The two equalities $T_{\ell}(\pi_{G})=\rho_{G,\ell}(\gamma)$ and $T_{\ell}(\pi_{Y})=\rho_{Y,\ell}(\gamma)$ are clear. We show the last one.
           Since $T_\ell(\pi_{\Mbf})|_{T_\ell(G)}=T_\ell(\pi_{G})$ and similarly $\rho_{\Mbf,\ell}(\gamma)|_{T_\ell(G)}=\rho_{G,\ell}(\gamma)$, the map $T_{\ell}(\pi_{\Mbf})-\rho_{\Mbf,\ell}(\gamma)$ factors as
           \[T_{\ell}(\Mbf)\twoheadrightarrow Y\otimes_{\Z}\Z_{\ell}\to T_{\ell}(G)\hookrightarrow T_{\ell}(\Mbf).\]
           Due to weight reason, there is no non-trivial Galois equivariant map between $Y\otimes_{\Z}\Z_{\ell}$ and $T_{\ell}(G)$. It follows that $T_{\ell}(\pi_{\Mbf})-\rho_{\Mbf,\ell}(\gamma)=0$. So we are done.
        \end{enumerate}
    \end{proof}

    \subsection{Log 1-motives}\label{Log 1-motives} 
	Let $S$ be an fs log scheme, and let $(\fs/S)$ be the category of fs log schemes over $S$. We denote by $(\fs/S)_{\kfl}$ (resp. $(\fs/S)_{\et}$) the Kummer log flat site (resp.  the classical \'etale site) on $(\fs/S)$, see \cite[\S 2.3]{Kato2021} and \cite[\S 2.5]{Illusie2002}. Sometimes we abbreviate $(\fs/S)_{\kfl}$ (resp. $(\fs/S)_{\et}$) as $S_{\kfl}$ (resp. $S_{\et}$) in order to shorten notations.	

    Let $\cC_{\kfl}$ be the category of sheaves of abelian groups on the  site $(\fs/S)_{\kfl}$, and let $\cD_{\kfl}$ be the derived category of $\cC_\kfl$. 

    We denote by $\mathring{S}$ the underlying scheme of $S$. For any scheme $V$ over $\mathring{S}$, we will regard $V$ as an fs log scheme over $S$ by endowing it with the induced log structure from $S$.

	Let $G$ be an extension of an abelian scheme $B$ by a torus $T$ over $\mathring{S}$. We denote by $G_{\log}$ the pushout of $G$ in the category $\cC_\kfl$ along the canonical inclusion $T\hookrightarrow T_{\log}:=\cHom_{S}(X, \Gml)$, where $X=\cHom_S(T,\Gm)$ is the character group of $T$ and $\Gml$ is Kato's logarithmic multiplicative group (see \cite[Theorem 3.2]{Kato2021}).
 
	\begin{definition}(\cite[Definition 2.2]{KajiwaraKatoNakayama2008})
    Let $Y$ be a group scheme over $\mathring{S}$, which is \'etale locally a free abelian group of finite rank, and let $G$ and $G_{\log}$ be as above. A \emph{log 1-motive with lattice part $Y$ and semi-abelian part $G$} over $S$, is a two term complex 
    \[\Mbf=[Y\xrightarrow{u} G_{\log}]\]
    in $\cC_\kfl$, where $Y$ sits in degree -1. We call $\Mbf$ a \emph{log 1-motive without abelian part}, if $B=0$.  A \emph{morphism of log 1-motives} is just a morphism of complexes.  We denote the category of log 1-motives (resp. log 1-motives without abelian part) over $S$ by $\Mcal^{\log}_1$ (resp. $\Mcal^{\log,\mathrm{ab}=0}_1$). 
    \end{definition}

    For $F_1\subset F_2$ two sheaves of abelian groups on $(\fs/S)_{\kfl}$, we denote by $F_2/F_1$ (resp. $(F_2/F_1)_{S_{\et}}$) the quotient on $(\fs/S)_{\kfl}$ (resp. $(\fs/S)_{\et}$). The composition
	\begin{center}$Y\xrightarrow{u} G_{\log}\to (G_{\log}/G)_{S_{\et}}=(T_{\log}/T)_{S_{\et}}=\cHom_{S}(X, (\Gml/\Gm)_{S_{\et}})$
	\end{center}
	corresponds to a pairing 
	\[\langle-,-\rangle: Y\times X \to (\Gml/\Gm)_{S_\et}.\]
	We call the pairing $\langle-,-\rangle$ the \emph{monodromy pairing} of the log 1-motive $\mathbf{M}$. Warning: here we switch the order of $X$ and $Y$ from \cite[\S2.3]{KajiwaraKatoNakayama2008}

    Let $X=\cHom_S(T,\Gm)$ as before, and let $T^*:=\cHom_S(Y,\Gm)$, $B^*:=\cExt^1_{S_{\et}}(B,\Gm)$ which is the dual abelian scheme of $B$, and $v$ the composition $Y\xrightarrow{u}G_{\log}\to B$. Then $G^*:=\cExt^1_{S_{\et}}([Y\xrightarrow{v}B],\Gm)$ fits into a short exact sequence $0\to T^*\to G^*\to B^*\to 0$ of \'etale sheaves of abelian groups, in particular $G^*$ is representable by a group scheme. For any $x\in X$, the pushout of $0\to T\to G\to B\to 0$ along $x:T\to\Gm$ gives rise to a homomorphism $v^*:X\to\cExt^1_{S_{\et}}(B,\Gm)=B^*$. The homomorphism $v^*$ can be lifted to a homomorphism $u^*:X\to G^*_{\log}$, see \cite[\S2.7.3]{KajiwaraKatoNakayama2008}, and the resulting log 1-motive $\Mbf^*=[X\xrightarrow{u^*}G^*_{\log}]$ is called the \emph{dual log 1-motive of $\Mbf$}, see \cite[Definition 2.7.4]{KajiwaraKatoNakayama2008}.

    \begin{definition}
        Let $\Mbf=[Y\xrightarrow{u} G_{\log}]$ be a log 1-motive over $S$, and let $\Mbf^*=[X\xrightarrow{u^*}G^*_{\log}]$ be the dual log 1-motive of $\Mbf$ (see \cite[Definition 2.7.4]{KajiwaraKatoNakayama2008}). A \emph{polarization} of $\Mbf$ is a homomorphism $h=(h_{-1},h_0):\Mbf\to\Mbf^*$ of log 1-motives satisfying
        \begin{enumerate}
            \item the homomorphism $f_{\mathrm{ab}}:B\to B^*$ induced by $h_0$ is a polarization on $B$;
            \item the homomorphism $h_{-1}$ is injective and of finite cokernel;
            \item $\langle y,h_{-1}(y)\rangle_{\bar{s}}\in (M_{S,\bar{s}}/\mathcal{O}^\times_{S,\bar{s}})\backslash\{1\}$ for any $s\in S$ and $y\in Y_{\bar{s}}$;
            \item the homomorphism $f_{\mathrm{t}}:T\to T^*$ induced by $f_0$ coincide with the one induced by $f_{-1}$.
        \end{enumerate}
        We say that $\Mbf$ is \emph{pointwise polarizable}, if the pullback of $\Mbf$ to $(\fs/\bar{s})$ admits a polarization for any $s\in S$.  We denote the category of pointwise polarizable log 1-motives over $S$ by 
        \[\Mcal_1^{\log,\mathrm{pPol}},\]
        and the full subcategory of $\Mcal_1^{\log,\mathrm{pPol}}$ consisting of the ones without abelian part by 
        \[\Mcal_1^{\log,\mathrm{ab}=0,\mathrm{pPol}}.\]
    \end{definition}

    \begin{definition}
        A morphism $(f_{-1},f_0):[Y\xrightarrow{u} G_{\log}]\to [Y'\xrightarrow{u'} G'_{\log}]$	 of log 1-motives over $S$ is called an \emph{isogeny}, if $f_{-1}$ is injective with finite cokernel and the homomorphism $G\to G'$ induced by $f_0$ (see \cite[Proposition 2.5]{KajiwaraKatoNakayama2008}) is surjective with kernel a finite flat group scheme.
	\end{definition}
	Note that $G_{\log}$ in \cite[\S 2.1]{KajiwaraKatoNakayama2008} is defined on the site $(\fs/S)_{\et}$, and agrees with ours by \cite[Proposition 2.1]{Zhao2017}.

     \begin{definition}
        Let $S$ be an fs log scheme, and let $\Mbf=[Y\xrightarrow{u} G_{\log}]$ be a log 1-motive over $S$. For any positive integer $n$, let 
        \[T_{\Z/n\Z}(\Mbf):=H^{-1}(\Mbf\otimes^L\Z/n\Z).\]
        For any prime number $\ell$, let 
        \[T_{\ell}(\Mbf):=\varprojlim_r T_{\Z/\ell^r\Z}(\Mbf).\]
    \end{definition}

    The canonical distinguished triangle $Y\xrightarrow{u}G_{\log}\to\Mbf\to Y[1]$ in $\cD_\kfl$ gives rise to a long exact sequence
    \begin{align*}
        H^{-1}(Y\otimes^L\Z/n\Z)\to H^{-1}(G_{\log}\otimes^L\Z/n\Z)\to H^{-1}(\Mbf\otimes^L\Z/n\Z)\to  \\
        H^{0}(Y\otimes^L\Z/n\Z)\to H^{0}(G_{\log}\otimes^L\Z/n\Z)\to\cdots .
    \end{align*}
    We have seen that $H^{-1}(Y\otimes^L\Z/n\Z)=0$ from last subsection. We have 
    \[H^{0}(G_{\log}\otimes^L\Z/n\Z)=G_{\log}\otimes\Z/n\Z=0\]
    by \cite[Proposition 2.3 (5)]{Zhao2017}. Recall that $X$ is the character group of $T$. By \cite[Proposition 2.2]{Zhao2017} we have $G_{\log}/G=\cHom_S(X,\Gml/\Gm)$  and thus $n:G_{\log}/G\hookrightarrow G_{\log}/G$. It follows that 
    \[H^{-1}(G_{\log}\otimes^L\Z/n\Z)=G_{\log}[n]=G[n].\]
    Thus we get a short exact sequence
    \[0\to G[n]\to T_{\Z/n\Z}(\Mbf)\to Y/nY\to 0\]
    in $\cC_\kfl$. Both $G[n]$ and $Y/nY$ are finite locally free group schemes over $S$. Assume that $\mathring{S}$ is locally noetherian. By \cite[Theorem 9.1]{Kato2021}, $T_{\Z/n\Z}(\Mbf)$ as a $G[n]$-torsor over $Y/nY$ is an object of $(\mathrm{fin}/S)_{\mathrm{r}}$, where $(\mathrm{fin}/S)_{\mathrm{r}}$ denotes the category of finite Kummer log flat group log schemes over $S$ which are logarithmic analogues of finite flat group schemes, see \cite[Definition A.1]{Zhao2017} or \cite[Definition 1.3]{Kato2023}. If $n$ is invertible on $S$, $T_{\Z/n\Z}(\Mbf)$ is Kummer log \'etale over $S$. Since $(G[\ell^i])_i$ is a Mittag-Leffler system, we also have a short exact sequence
    \begin{equation}\label{l-adic s.e.s. of log 1-motive}
        0\to T_{\ell}(G)\to T_{\ell}(\Mbf)\to Y\otimes_{\Z}\Z_{\ell}\to 0
    \end{equation}
    in $\cC_\kfl$. If $\ell$ is invertible on $S$ and $S$ is connected, then the sequence \eqref{l-adic s.e.s. of log 1-motive} is a short exact sequence of free $\Z_\ell$-modules of finite rank endowed with a continuous $\pi_1^{\ket}(S)$-action, where $\pi_1^{\ket}(S)$ denotes the Kummer log \'etale fundamental group of $S$, see \cite[\S10]{Kato2021} or \cite[\S4]{Illusie2002}.

    \begin{definition}
        Assume that $S$ is a log point, and let $\Mbf=[Y\xrightarrow{u}G_{\log}]$ be an object of $\Mcal_1^{\log,\mathrm{pPol}}$. We say that $\Mbf$ is \emph{simple} (in $\Mcal_1^{\log,\mathrm{pPol}}$), if
        any non-zero morphism 
        \[f=(f_{-1},f_0):[Y'\to G'_{\log}]\to [Y\to G_{\log}]\] in $\Mcal_1^{\log,\mathrm{pPol}}$ with both $f_{-1}$ and $f_0$ injective 
        has to be an isogeny.
    \end{definition}

    We warn the reader that a simple object in $\Mcal_1^{\log,\mathrm{pPol}}$ is not simple in $\Mcal_1^{\log}$ in general if we define the notion of being simple in $\Mcal_1^{\log}$ as in Definition \ref{def simple 1-motives}.

    Being simple is preserved under isogeny.

    \begin{proposition}\label{simpleness of ppol log 1-motives is preserved by isogeny}
        Let $(f_{-1},f_0):\Mbf'=[Y'\to G'_{\log}]\to [Y\to G_{\log}]=\Mbf$ be an isogeny of pointwise polarizable log 1-motives. Then $\Mbf'$ is simple if and only if $\Mbf$ is simple.
    \end{proposition}
    \begin{proof}
        Let $f_{\mathrm{c}}:G'\to G$ be the homomorphism induced by $f_0$.
        
        Assume that $\Mbf$ is simple. We show that $\Mbf'$ is also simple. Let 
        \[(i_{-1},i_0):\Mbf'_1=[Y_1'\xrightarrow{u_1'} G'_{1\log}]\to\Mbf'\]
        be a morphism such that both $i_{-1}$ and $i_0$ are injective. Let $i_{\mathrm{c}}:G'_1\to G'$ be the homomorphism induced by $i_0$. Clearly
        $\ker(f_{\mathrm{c}}\circ i_{\mathrm{c}})\subset \ker(f_{\mathrm{c}})$, and thus $\ker(f_{\mathrm{c}}\circ i_{\mathrm{c}})$ is a finite flat group scheme. Let $\overline{G}'_1:=G'_1/\ker(f_{\mathrm{c}}\circ i_{\mathrm{c}})$, $\pi_{\mathrm{c}}:G'_1\to\overline{G}'_1$ the canonical projection, $\pi:G'_{1\log}\to\overline{G}'_{1\log}$ the map induced by $\pi_{\mathrm{c}}$, and $\overline{u}'_1:=\pi\circ u'_1$. Then $f_{\mathrm{c}}\circ i_{\mathrm{c}}$ factors as $G'_1\xrightarrow{\pi_{\mathrm{c}}}\overline{G}'_1\xrightarrow{j_{\mathrm{c}}}G$ for some $j_{\mathrm{c}}$, and thus $f_0\circ i_0$ factors as $G'_{1\log}\xrightarrow{\pi}\overline{G}'_{1\log}\xrightarrow{j}G_{\log}$, so $(f_{-1}\circ i_{-1},f_0\circ i_0)=(f_{-1}\circ i_{-1},j)\circ (1_{Y'_1},\pi)$. Since both $f_{-1}\circ i_{-1}$ and $j$ are injective and $\Mbf$ is simple, $(f_{-1}\circ i_{-1},j)$ is an isogeny. Since  $(1_{Y'_1},\pi)$ and $(f_{-1},f_0)$ are isogenies, so is $(i_{-1},i_0)$. It follows that $\Mbf'$ is simple.

        Conversely, assume that $\Mbf'$ is simple. Let $n$ be a positive integer which kills both $\coker(f_{-1})$ and $\ker(f_{\mathrm{c}})$. Then the multiplication-by-$n$ map $n_{\Mbf'}$ on $\Mbf'$ factors as $n_{\Mbf'}=(g_{-1},g_0)\circ (f_{-1},f_0)$ with $(g_{-1},g_0):\Mbf\to\Mbf'$ an isogeny. Applying the above to $(g_{-1},g_0)$, we have that $\Mbf$ is simple by the simpleness of $\Mbf'$.
    \end{proof}

    In the rest of this subsection, we assume that $S=(\Spec\kbf,M_S)$ is a finite log point as in \eqref{charted log point}. Let $\Gamma:=\Gal(\overline{\kbf}/\kbf)=\overline{\langle\gamma\rangle}$ be as in (\ref{description of Galois gp of finite field}). By \cite[Example 4.7 (a)]{Illusie2002}, we have a short exact sequence
    \begin{equation}\label{s.e.s. of ket fundamental gp}
        1\to \Hom(P^{\gp},\widehat{\Z}'(1)(\overline{\kbf}))\to\pi_1^{\ket}(S)\to \Gamma\to 1,
    \end{equation}
    where $\widehat{\Z}'(1)$ denotes the product of $\Z_{\ell}(1)=\varprojlim_r\mu_{\ell^r}(\overline{\kbf})$ for $\ell$ different from $\mathrm{char}(\kbf)$.

    \begin{proposition}\label{char polynomial of Frobenius of log 1-motive}
        Let $\Mbf=[Y\xrightarrow{u}G_{\log}]$ be a log 1-motive over the log point $S$, and let $\ell$ be a prime number which is coprime to $\mathrm{char}(\kbf)$. Let $\tilde{\gamma}\in \pi_1^{\ket}(S)$ be a lifting of $\gamma\in\Gamma$ along the sequence \eqref{s.e.s. of ket fundamental gp}, and let $\tilde{\rho}_{\Mbf,\ell}$ (resp. $\rho_{G,\ell}$, resp. $\rho_{Y,\ell}$) denote the canonical action of $\pi_1^{\ket}(S)$ (resp. $\Gamma$, resp. $\Gamma$) on $T_{\ell}(\Mbf)$ (resp. $T_{\ell}(G)$, resp. $T_{\ell}(Y[1])=Y\otimes_{\Z}\Z_{\ell}$). Then we have the following.
        \begin{enumerate}
            \item The characteristic polynomial $P_{\tilde{\rho}_{\Mbf,\ell}(\tilde{\gamma})}(\theta)$ of $\tilde{\rho}_{\Mbf,\ell}(\tilde{\gamma})$ factors as
            \[P_{\tilde{\rho}_{\Mbf,\ell}(\tilde{\gamma})}=P_{Y,\pi_Y}(\theta)\cdot P_{B,\pi_B}(\theta)\cdot P_{T,\pi_T}(\theta),\]
            where $P_{Y,\pi_Y}(\theta)$, $P_{B,\pi_B}(\theta)$ and $P_{T,\pi_T}(\theta)$ are as in Lemma \ref{charpoly of lattices, AVs and tori}.
            
            \item The polynomial $P_{\tilde{\rho}_{\Mbf,\ell}(\tilde{\gamma})}(\theta)\in\Z[\theta]$ and is independent of the choice of $\ell$.
            
            \item The polynomial $P_{\tilde{\rho}_{\Mbf,\ell}(\tilde{\gamma})}(\theta)$ is independent of the choice of the lifting $\tilde{\gamma}$ of $\gamma$.

            \item Let $f:\Mbf'\to\Mbf$ be an isogeny of log 1-motives. Then $T_{\ell}(f):T_{\ell}(\Mbf')\to T_{\ell}(\Mbf)$ is injective with finite cokernel. In particular, $P_{\tilde{\rho}_{\Mbf',\ell}(\tilde{\gamma})}(\theta)=P_{\tilde{\rho}_{\Mbf,\ell}(\tilde{\gamma})}(\theta)$.
        \end{enumerate}
    \end{proposition}
    \begin{proof}
        (1) Let $P_{\rho_{G,\ell}(\gamma)}(\theta)$ (resp. $P_{\rho_{Y,\ell}(\gamma)}(\theta)$) be the characteristic polynomial of the linear map $\rho_{G,\ell}(\gamma)$ (resp. $\rho_{Y,\ell}(\gamma)$). The restriction of $\tilde{\rho}_{\Mbf,\ell}(\tilde{\gamma})$ to $T_{\ell}(G)$ is $\rho_{G,\ell}(\gamma)$, and the map induced by $\tilde{\rho}_{\Mbf,\ell}(\tilde{\gamma})$ on $Y\otimes\Z_{\ell}$ is $\rho_{Y,\ell}(\gamma)$. Therefore by the short exact sequence \eqref{l-adic s.e.s. of log 1-motive} we get 
        \[P_{\tilde{\rho}_{\Mbf,\ell}(\tilde{\gamma})}(\theta)=P_{\rho_{G,\ell}(\gamma)}(\theta)\cdot P_{\rho_{Y,\ell}(\gamma)}(\theta).\]
        Then the result follows from Lemma \ref{charpoly of lattices, AVs and tori} (4) (applied to $G$ as a 1-motive) and (5) in \emph{loc.cit}.

        (2) The result follows from (1), Proposition \ref{char poly has integral coefficients and is indep of the choice of the prime}, and Lemma \ref{charpoly of lattices, AVs and tori} (5).

        (3) The result follows from (1).

        (4) By the exact sequence \eqref{l-adic s.e.s. of log 1-motive}, one is reduced to the case of $Y$ and the case of $G$ which are known.
    \end{proof}

    Let $\Mbf$ be as above. By (2) and (3) of the proposition, the polynomial $P_{\tilde{\rho}_{\Mbf,\ell}(\tilde{\gamma})}(\theta)$ is independent of the choice of $\ell$ and the choice of $\tilde{\gamma}$, we rewrite it as
        \begin{equation}\label{equation char poly of Frobenius on log 1-motive}
            P_{\Mbf,\gamma}(\theta):=P_{\tilde{\rho}_{\Mbf,\ell}(\tilde{\gamma})}(\theta)
        \end{equation}
        and call it the \emph{characteristic polynomial of $\gamma$ on $\Mbf$} by abuse of terminology. 

        \begin{remark}\label{char poly of log 1-motive is invariant under isogeny}
            By Proposition \ref{char polynomial of Frobenius of log 1-motive} (1), we have 
            \[P_{\Mbf,\gamma}(\theta)=P_{Y,\pi_Y}(\theta)\cdot P_{B,\pi_B}(\theta)\cdot P_{T,\pi_T}(\theta).\]
            Moreover $P_{\Mbf,\gamma}(\theta)$ is invariant under isogeny by Proposition \ref{char polynomial of Frobenius of log 1-motive} (4).
        \end{remark}

\subsection{Log abelian varieties with constant degeneration}\label{Log abelian varieties with constant degeneration}
        Let $S$ be an fs log scheme. We consider a log 1-motive $\Mbf=[Y\xrightarrow{u}G_{\log}]$ over $S$.
	
	Via the monodromy pairing $\langle-,-\rangle:Y\times X\to (\Gml/\Gm)_{S_{\et}}$ of $\Mbf$, we define the subsheaf 
	\[\cHom(X, (\Gml/\Gm)_{S_{\et}})^{(Y)}\subset \cHom(X, (\Gml/\Gm)_{S_{\et}})\]
	on $(\fs/S)_{\et}$ by letting $\cHom(X, \Gml/\Gm)^{(Y)}(U)$ to be:
	\begin{center}
		$\{ \phi\in \cHom(X, (\Gml/\Gm)_{S_{\et}})(U) :$ for any $u\in U$ and $x\in X_{\bar{u}}$, there exist $y_{u, x}, y_{u, x}'\in Y_{\bar{u}}$, such that $\langle y_{u, x},x\rangle_{\bar{u}} \vert \phi(x)_{\bar{u}} \vert \langle y'_{u, x},x\rangle_{\bar{u}} \}$.
	\end{center}
	for any $U\in (\fs/S)$.  Here $\overline{u}$ is a geometric point over $u$, and in an integral monoid $Q$ the relation $a\mid b$ for $a,b\in Q^\gp$ means that $b=a\cdot c$ for some $c\in Q$.
 
    We then define $G_{\log}^{(Y)}$ as the preimage of $\cHom(X, (\Gml/\Gm)_{S_{\et}})^{(Y)}$ under the natural map 
	\[G_{\log}\to (G_{\log}/G)_{S_{\et}}\cong \cHom(X, (\Gml/\Gm)_{S_{\et}})\]
	on $(\fs/S)_{\et}$. By \cite[Proposition 2.4]{Zhao2017}, $G_{\log}^{(Y)}$ is also a sheaf on $(\fs/S)_{\kfl}$.
	\begin{definition}(\cite[Definition 3.3]{KajiwaraKatoNakayama2008})
		A \emph{log abelian variety with constant degeneration} over $S$ is a sheaf $A$ of abelian groups on $(\fs/S)_{\et}$, which is isomorphic to $G_{\log}^{(Y)}/Y$ for a pointwise polarizable log 1-motive $[Y\xrightarrow{u} G_{\log}]$. 
        
		A \emph{homomorphism} of log abelian varieties with constant degeneration is just a homomorphism of sheaves of abelian groups on $(\fs/S)_{\et}$. We denote the category of log abelian varieties with constant degeneration over $S$ by 
        \[\mathrm{LAVwCD}.\]
		
		If the abelian part of $G$ is zero, we call $A$ \emph{totally degenerate} or \emph{without abelian part}. We denote the full subcategory of $\mathrm{LAV}$ consisting of totally degenerate log abelian varieties with constant degeneration over $S$ by 
        \[\mathrm{LAVwCD}^{\mathrm{ab}=0}.\]
	\end{definition}
	
	By \cite[Theorem 2.1]{Zhao2017}, a log abelian variety with constant degeneration is also a sheaf on $(\fs/S)_{\kfl}$.

    \begin{definition}(\cite[Definition 3.1]{Zhao2017})
        A homomorphism $f:A\to A'$ in $\mathrm{LAVwCD}$ is called an \emph{isogeny}, if $f$ is a surjective map of sheaves on $(\fs/S)_{\kfl}$ and $\ker(f)\in(\mathrm{fin}/S)_{\mathrm{r}}$, where $(\mathrm{fin}/S)_{\mathrm{r}}$ denotes the category of finite Kummer log flat group log schemes over $S$ which are logarithmic analogues of finite flat group schemes, see \cite[Definition A.1]{Zhao2017} or \cite[Definition 1.3]{Kato2023}.
    \end{definition}
	
	In this paper, we do not give the definition of log abelian varieties in general, see \cite[Definition 4.1]{KajiwaraKatoNakayama2008}. This is because we are only concerned with $S$ being a log point, over which log abelian varieties are always log abelian varieties with constant degeneration by \cite[Theorem 4.6]{KajiwaraKatoNakayama2008}.
	
	The following theorem is very important for studying log abelian varieties with constant degeneration. 
	
	\begin{theorem}\label{equivalence from log 1-motives to LAVwCD}
		(\cite[Theorem 3.4]{KajiwaraKatoNakayama2008})
		The association 
		\begin{center}
			$\Mcal_1^{\log,\mathrm{pPol}}\to \mathrm{LAVwCD},\quad [Y\xrightarrow{u} G_{\log}] \mapsto G^{(Y)}_{\log}/Y$
		\end{center}
		defines an equivalence from the category of pointwise polarizable log 1-motives over $S$ to the category of log abelian varieties with constant degeneration over $S$.
	\end{theorem}
        Under the above equivalence of categories, suppose that a log abelian varieties $A$ corresponds to a log 1-motive $\Mbf$, we say $\Mbf$ is the \emph{log 1-motive corresponding to $A$} and \emph{vice versa}.
        
 \begin{proposition}\label{Hom of LAVs is isogeny iff the corresponding hom of log 1-motives is isogeny}
     Let $f:A\to A'$ be a homomorphism of log abelian varieties with constant degeneration over $S$, and let 
     \[(f_{-1},f_0):\Mbf=[Y\to G_{\log}]\to \Mbf'=[Y'\to G'_{\log}]\] 
     be the corresponding morphism of log 1-motives corresponding to $f$ under the equivalence of categories from Theorem \ref{equivalence from log 1-motives to LAVwCD}. Then $f$ is an isogeny if and only if $(f_{-1},f_0)$ is an isogeny, and if this is the case, we have a short exact sequence
     \[0\to\ker(f_0)\to\ker(f)\to \mathrm{coker}(f_{-1})\to 0\]
     of sheaves of abelian groups on $(\fs/S)_{\kfl}$.
 \end{proposition}
 \begin{proof}
     Assume that $f$ is an isogeny. Let $f_{\mathrm{c}}:G\to G'$ be the homomorphism induced by $f_0$ (see Proposition \cite[Proposition 2.5]{KajiwaraKatoNakayama2008}). For any $s\in S$, $f_{\bar{s}}$ is an isogeny. By the equivalence $(1)\Leftrightarrow(4)$ of \cite[Proposition 3.3]{Zhao2017}, $(f_{\mathrm{c}})_{\bar{s}}$ is an isogeny and $(f_{-1})_{\bar{s}}$ is injective with finite cokernel. Then $f_{\mathrm{c}}$ is faithfully flat with finite locally free kernel by Lemma \ref{hom of semiabelian schemes} (2) and $f_{-1}$ is injective with cokernel finite locally constant by Lemma \ref{hom of lattices}. Therefore $(f_{-1},f_0)$ is an isogeny of log 1-motives.

     Conversely, assume that $(f_{-1},f_0)$ is an isogeny. By the diagrams (2.10) and (2.12) from \cite{Zhao2017} and \cite[Theorem 9.1]{Kato2021}, it suffices to show that the homomorphism $\tilde{f}_{\mathrm{d}}:\mathcal{Q}\to\mathcal{Q}'$ in the diagram (2.12) of \cite{Zhao2017} is an isomorphism. Note that both $\Mbf$ and $\Mbf'$ are pointwise polarizable, and thus they are non-degenerate (see \cite[Remark 2.2]{Zhao2017} for the notion of being non-degenerate). Let $f_{\mathrm{t}}:T\to T'$ be the homomorphism induced by $f_{\mathrm{c}}$ on the torus parts, and let $f_{\mathrm{l}}:X'\to X$ be the homomorphism induced by $f_{\mathrm{t}}$ on the character groups. By Lemma \ref{hom of semiabelian schemes} (2) the homomorphism $f_{\mathrm{t}}$ is surjective with finite locally free kernel, and thus $f_{\mathrm{l}}$ is injective with finite cokernel. Then the same argument from the proof of \cite[Lemma 3.2]{Zhao2017} works also here, and thus $\tilde{f}_{\mathrm{d}}$ is an isomorphism. Note that the short exact sequence in the statement also follows.
 \end{proof}

     In the rest of this article, the base $S$ will always be a log point. In particular, a log abelian variety over $S$ will always be a log abelian variety with constant degeneration. We will call a log abelian variety with constant degeneration over $S$ simply a \emph{log abelian variety} over $S$, and write $\mathrm{LAVwCD}$ and $\mathrm{LAVwCD}^{\mathrm{ab}=0}$
     simply as 
     \[\mathrm{LAV} \text{ and } \mathrm{LAV}^{\mathrm{ab}=0}\]
     respectively.

     The following definition is taken from \cite[Definition 3.2]{Zhao2017}.

     \begin{definition}
         A log abelian variety $A$ over a log point $S$ is called \emph{simple}, if any non-zero injective homomorphism $A'\to A$ has to be an isomorphism. 
     \end{definition}

     \begin{proposition}\label{LAV is simple iff its log 1-motive is simple}
         Let $A\in\mathrm{LAV}$ be non-zero, and let $\Mbf=[Y\xrightarrow{u}G_{\log}]\in\Mcal_1^{\log,\mathrm{pPol}}$ be the object corresponding to $A$. Then $A$ is simple if and only if $\Mbf$ is simple.
     \end{proposition}
     \begin{proof}
         Assume that $A$ is simple. It suffices to show that any morphism 
         \[(f_{-1},f_0):\Mbf_1=[Y_1\to G_{1\log}]\to\Mbf\]
         in $\Mcal_1^{\log,\mathrm{pPol}}$, which is injective as a morphism of complexes, has to be an isogeny. Let $f:A_1\to A$ be the homomorphism in LAV induced by $(f_{-1},f_0)$. Since $\Mbf_1\in \Mcal_1^{\log,\mathrm{pPol}}$, it satisfies the conditions in \cite[Lemma 3.4]{Zhao2017}, and thus $\ker(f)\in (\mathrm{fin}/S)_{\mathrm{r}}$. By \cite[Proposition 3.2 (3)]{Zhao2017}, $A_1/\ker(f)$ is also a log abelian variety, and thus a log abelian subvariety of $A$. Since $A$ is simple, $f$ is an isogeny, and thus $(f_{-1},f_0)$ is an isogeny by Proposition \ref{Hom of LAVs is isogeny iff the corresponding hom of log 1-motives is isogeny}.

         Conversely, assume that $\Mbf$ is simple. Let $f\colon A'\to A$ be a non-zero injective homomorphism in LAV. We show that $f$ is an isomorphism. Let $\Mbf'=[Y'\xrightarrow{u'}G'_{\log}]\in\Mcal_1^{\log,\mathrm{pPol}}$ be the log 1-motive corresponding to $A'$, and let $(f_{-1},f_0):\Mbf'\to\Mbf$ be the morphism corresponding to $f$. Let $f_{\mathrm{c}}:G'\to G$ be the homomorphism induced by $f_0$, let $f_{\mathrm{t}}:T'\to T$ be the homomorphism induced by $f_{\mathrm{c}}$, and let $f_{\mathrm{l}}:X\to X'$ be the homomorphism of character groups induced by $f_{\mathrm{t}}$. By the diagram \cite[(2.10)]{Zhao2017}, the injectivity of $f$ implies that $f_{\mathrm{c}}$ is injective, which further implies that $f_{\mathrm{t}}$ is injective and $f_{\mathrm{l}}$ is surjective. Then the surjectivity of $f_{\mathrm{l}}$ implies $\cHom_S(X',\Gml/\Gm)\hookrightarrow\cHom_S(X,\Gml/\Gm)$, in particular the map $\tilde{f}_{\mathrm{d}}:\Qcal'\to\Qcal$ corresponding to the middle vertical map of \cite[(2.12)]{Zhao2017} is injective, and thus $f_{-1}:Y'\to Y$ is injective by \cite[(2.12)]{Zhao2017}. By \cite[Proposition 2.3 (4)]{Zhao2017} and the injectivity of $f_{\mathrm{c}}$, $f_0$ is injecitve. Since $\Mbf$ is simple, the morphism $(f_{-1},f_0)$ as an injective map of complexes has to be an isogeny of log 1-motives. Then the inclusion $f$ is also an isogeny by Proposition \ref{Hom of LAVs is isogeny iff the corresponding hom of log 1-motives is isogeny}, and thus it is an isomorphism. Therefore $A$ is simple.
     \end{proof}

     The simpleness of log abelian varieties is preserved under isogeny.

     \begin{proposition}\label{simpleness of LAV is preserved by isogeny}
         Let $f:A'\to A$ be an isogeny of log abelian varieties. Then $A$ is simple if and only if $A'$ is simple.
     \end{proposition}
     \begin{proof}
         Let $(f_{-1},f_0):\Mbf'\to\Mbf$ be the morphism corresponding to $f$. By Proposition \ref{Hom of LAVs is isogeny iff the corresponding hom of log 1-motives is isogeny}, $(f_{-1},f_0)$ is an isogeny. Then the result follows from Proposition \ref{LAV is simple iff its log 1-motive is simple} and Proposition \ref{simpleness of ppol log 1-motives is preserved by isogeny}.
     \end{proof}

     We end this subsection with the following
     \begin{definition}\label{char poly of log av}
        Let $\kbf$, $S$ and $\gamma$ be as in Proposition \ref{char polynomial of Frobenius of log 1-motive}. Let $A$ be a log abelian variety over the log point $S$ and $\Mbf=[Y\to G_{\log}]$ the corresponding log 1-motive. We define the \emph{characteristic polynomial} $P_{A,\gamma}$ of $\gamma$ on $A$ to be the characteristic polynomial of $\gamma$ on $\Mbf$ (see \eqref{equation char poly of Frobenius on log 1-motive}).
    \end{definition}

    \begin{remark}\label{char poly of LAV is invariant under isogeny}
        Clearly this polynomial $P_{A,\gamma}$ for log abelian varieties satisfies similar properties as listed in Proposition \ref{char polynomial of Frobenius of log 1-motive} for log 1-motives. In particular, it is invariant under isogeny.
    \end{remark}

 \section{Lattice pairings}\label{section lattice pairings}

 \subsection{Lattice pairings}\label{subsection lattice pairings}
 Let $\kbf$ be a field, let $P$ be a sharp fs monoid, and let $S=(\Spec \kbf,M_S)$ be a log point such that its log structure $M_S$ admits a chart $P_S\to M_S$ satisfying $P\xrightarrow{\cong}\Gamma(\Spec\kbf,M_S)/\kbf^\times$ (here $P_S$ is the constant sheaf over $S$ associated to $P$). We will identify $P$ with $\Gamma(\Spec\kbf,M_S)/\kbf^\times$ via this isomorphism.

 We fix $\overline{\mathbf{k}}$ an algebraic closure of the field $\mathbf{k}$, and let $\Gamma:=\mathrm{Gal}(\overline{\mathbf{k}}/\mathbf{k})$ be the absolute Galois group of $\kbf$. Note that $P^{\gp}$ is a free $\mathbb{Z}$-module of finite rank. We will define three categories in the following
	\[
	\Mod_\Gamma,
	\quad
        \MorGammaP,
        \quad
        \MorGammaPpPol,
	\]
 which will be used to study the ``pure log part'' (with respect to the given chart of $S$) of log abelian varieties over $S$.

 \begin{definition}\label{free Gamma-modules}
     We write 
     \[\Mod_{\Gamma}\]
     for the category of free $\Gamma$-modules: 
     \begin{enumerate}
         \item A \emph{free $\Gamma$-module} is a free $\mathbb{Z}$-module of finite rank $M$ (equipped with the discrete topology) together with a continuous action of $\Gamma$ (note that any such action of $\Gamma$ necessarily factors through a finite quotient of $\Gamma$).
         \item A \emph{morphism} of free $\Gamma$-modules is a $\mathbb{Z}$-linear $\Gamma$-equivariant map $M\rightarrow N$.
     \end{enumerate}  
     We also call an object of $\Mod_{\Gamma}$ a \emph{lattice}.
    \end{definition}

    A morphism $M\to N$ in $\Mod_{\Gamma}$ is called an \emph{isogeny} if it is injective and has finite cokernel. In this case, we say that $M$ is isogenous to $N$ (by the morphism $M\rightarrow N$). By Lemma \ref{isogeny between free Z-modules} below, $M$ is isogenous to $N$ if and only if $N$ is isogenous to $M$.
    
    \begin{lemma}\label{isogeny between free Z-modules}
		Let $M,N$ be two objects in $\Mod_\Gamma$. For $l$ injective $\Gamma$-equivariant $\mathbb{Z}$-linear maps $F_i\colon M\rightarrow N$ with finite cokernel ($i=1,\cdots,l$), there are $\mathbb{Z}$-linear maps $G_i\colon N\rightarrow M$ ($i=1,\cdots,l$) and a positive integer $r$ such that
		\normalfont
		\[
		F_i\circ G_i=r\times\mathrm{Id}_N,
		\quad
		G_i\circ F_i=r\times\mathrm{Id}_M,
		\quad
		\forall
		i=1,\cdots,l.
		\]
    \end{lemma}
    \begin{proof}
            Let $r$ be a positive integer which kills $\mathrm{coker}(F_i)$ for all $i=1,\cdots,l$. Since the composition $N\xrightarrow{r}N\xrightarrow{\mathrm{pr}}\mathrm{coker}(F_i)$ is trivial in the diagram
            \[\xymatrix{
            &&N\ar[d]^{r}\ar@{-->}[ld]_{G_i} \\
            0\ar[r] &M\ar[r]_{F_i} &N\ar[r]_-{\mathrm{pr}} &\mathrm{coker}{F_i}\ar[r] &0,
            }\]
            there exists a unique map $G_i:N\to M$ such that $F_i\circ G_i=r\times\mathrm{Id}_N$. The equality $G_i\circ F_i=r\times\mathrm{Id}_M$ is an easy exercise of matrices.
    \end{proof}

        For $M\in \mathrm{Mod}_{\Gamma}$, the \emph{dual} of $M$ is 
    \[
    M^\vee=\mathrm{Hom}_{\mathbb{Z}}(M,\mathbb{Z}),
    \]
    which is equipped with the induced action of $\Gamma$: for $\gamma\in\Gamma$, $f\in M^\vee$ and $m\in M$, $(\gamma f)(m)=f(\gamma^{-1}m)$). Thus the natural pairing between $M$ and $M^\vee$ is $\Gamma$-equivariant
    \begin{equation}\label{natural pairing is Gammma-equivariant}
        M\times M^\vee\to\Z,
        \quad
        (m,f)
        \mapsto
        f(m)=(\gamma f)(\gamma m).
    \end{equation}    
    Moreover we have a canonical isomorphism in $\mathrm{Mod}_{\Gamma}$
	\[
	M\simeq(M^\vee)^\vee.
	\]    
    For a morphism $\psi\colon M\rightarrow N$ in $\mathrm{Mod}_{\Gamma}$, we write $\psi^\vee\colon N^\vee\rightarrow M^\vee$ for the induced morphism on the duals of $M$ and $N$.

    \begin{definition}\label{simple free Gamma-modules}
        A free $\Gamma$-module $M$ is called \emph{simple}, if there is no proper $\mathbb{Z}$-submodule $M'$ of $M$ stable under $\Gamma$ such that $M/M'$ is an abelian group of rank $>0$. In other words, $M$ is simple if there is no proper $\mathbb{Z}$-submodule $M'$ of $M$ stable under $\Gamma$ such that $M'\otimes_{\mathbb{Z}}\mathbb{Q}$ is a proper subspace of $M\otimes_{\mathbb{Z}}\mathbb{Q}$.
    \end{definition}

    Simple objects in $\Mod_{\Gamma}$ are preserved under isogenies:
    \begin{lemma}\label{simpleness of lattices is preversed under isogeny}
        Let $f:M\to M'$ be an isogeny in $\Mod_{\Gamma}$. Then $M$ is a simple free $\Gamma$-module if and only if so is $M'$.
    \end{lemma}
    \begin{proof}
        The proof is easy, and we leave it as an exercise to the reader.
    \end{proof}

     \begin{definition}
     We write 
     \[\mathrm{Mor}_{\Gamma}^{P}\]
     for the following category:
     \begin{enumerate}
         \item Objects are tuples $(M,N,\langle-,-\rangle)$ where $M,N$ are free $\Gamma$-modules and $\langle-,-\rangle\colon M\times N\rightarrow P^{\gp}$ is a $\Gamma$-equivariant bilinear map. We give $\mathrm{Hom}(N,P^{\gp})$ the natural $\Gamma$-module structure:
	   for $\gamma\in\Gamma$, $f\in\mathrm{Hom}(N,P^{\gp})$ and $x\in N$,
	   $(\gamma f)(x)=f(\gamma^{-1}x)$. The bilinear map $\langle-,-\rangle$ is naturally identified with a morphism 
	  \[\Phi\colon M\rightarrow\mathrm{Hom}(N,P^{\gp})=N^{\vee}\otimes P^{\gp}\]
	  in $\mathrm{Mod}_{\Gamma}$. So we also write $(M,N,\Phi)$ for $(M,N,\langle-,-\rangle)$.
         \item A morphism 
	\[\psi=(\psi_1,\psi_2)\colon(M,N,\Phi)\rightarrow(M',N',\Phi')\]
	is a pair of morphisms $\psi_1\colon M\rightarrow M'$ and $\psi_2\colon N'\rightarrow N$ in $\mathrm{Mod}_{\Gamma}$ such that the following diagram commutes
	\begin{equation}\label{morphism of log free Gamma-modules}
		\begin{tikzcd}
			M
			\arrow[r,"\Phi"]
			\arrow[d,"\psi_1"]
			&
			\mathrm{Hom}(N,P^{\gp})=N^\vee\otimes P^\gp
			\arrow[d,"\psi_2^\vee\otimes \id_{P^\gp}"]
			\\
			M'
			\arrow[r,"\Phi'"]
			&
			\mathrm{Hom}(N',P^{\gp})=(N')^\vee\otimes P^\gp
		\end{tikzcd}	
	\end{equation}
     \end{enumerate}

     We also call an object of $\MorGammaP$ a \emph{lattice pairing with respect to $P$} or simply a \emph{lattice pairing} when the context is clear.
	\end{definition}	

     Given $(M,N,\langle-,-\rangle)\in\MorGammaP$, we get another object $(N,M,\langle-,-\rangle^*)\in\MorGammaP$, where $\langle-,-\rangle^*:N\times M\to P^\gp$ is the pairing  obtained by switching the two factors of $\langle-,-\rangle$. We call this object the \emph{dual} of $(M,N,\langle-,-\rangle)$. Let $\Phi^*:N\to M^\vee\otimes P^\gp$ be the morphism corresponding to $\Phi$ and $\langle-,-\rangle$ under the following natural identifications     
     \[
     \Hom_{\Mod_\Gamma}(M,N^\vee\otimes P^\gp)
     \simeq
     \Hom_{\Mod_\Gamma}(M\otimes N,P^\gp)
     \simeq
     \Hom_{\Mod_\Gamma}(N,M^\vee\otimes P^\gp),
     \]
     then the dual of $(M,N,\langle-,-\rangle)=(M,N,\Phi)$ can be expressed as $(N,M,\langle-,-\rangle^*)=(N,M,\Phi^*)$.
     
     A morphism $\psi$ is called an \emph{isogeny} if both $\psi_1$ and $\psi_2$ are isogenies in $\mathrm{Mod}_{\Gamma}$. In this case, we say that $(M,N,\Phi)$ is \emph{isogenous} to $(M',N',\Phi')$. It is not hard to see that $(M,N,\Phi)$ is isogenous to $(M',N',\Phi')$ if and only if $(M',N',\Phi')$ is isogenous to $(M,N,\Phi)$. Indeed, if $(M,N,\Phi)$ is isogenous to $(M',N',\Phi')$, then by Lemma \ref{isogeny between free Z-modules}, we can find morphisms $\psi_1'\colon M'\rightarrow M$ and $\psi_2'\colon N\rightarrow N'$ in $\mathrm{Mod}_{\Gamma}$ such that $\psi_1'\circ\psi_1=r\times\mathrm{Id}_{M}$ and 
     $\psi_2\circ\psi_2'=r\times\mathrm{Id}_{N}$ for a positive integer $r$. So in the following diagram
	\[
	\begin{tikzcd}
		M
		\arrow[r,"\Phi"]
		\arrow[d,"\psi_1"]
		&
		\mathrm{Hom}(N,P^{\gp})
		\arrow[d,"\psi_2^\vee\otimes \id_{P^\gp}"]
		\\
		M'
		\arrow[r,"\Phi'"]
		\arrow[d,"\psi_1'"]
		&
		\mathrm{Hom}(N',P^{\gp})
		\arrow[d,"(\psi_2')^\vee\otimes\id_{P^{\gp}}"]
		\\
		M
		\arrow[r,"\Phi"]
		&
		\mathrm{Hom}(N,P^{\gp})
	\end{tikzcd}
	\]
	both the upper square and the bigger square  commute, then so does the lower square. This implies that $\psi'=(\psi_1',\psi_2')\colon(M',N',\Phi')\rightarrow(M,N,\Phi)$ is an isogeny. 
 
	A \emph{polarization} on $(M,N,\Phi)=(M,N,\langle-,-\rangle)$ is a morphism of the form 
    \[(\lambda,\lambda):(M,N,\Phi)\to (N,M,\Phi^*)\]
    such that 
        \begin{enumerate}
	    \item $\lambda\colon M\to N$ an isogeny in $\mathrm{Mod}_{\Gamma}$,
            \item and for any $m\in M\backslash\{0\}$,
		$\Phi(m)(\lambda(m))\in P\backslash\{0\}$ (recall that $P$ is sharp, i.e. $P^\times=\{0\}$).   
	\end{enumerate}
    Note that $(\lambda,\lambda)$ being a morphism means that the diagram
		\[
		\begin{tikzcd}
			M
			\arrow[r,"\Phi"]
			\arrow[d,"\lambda"]
			&
			N^\vee\otimes P^{\gp}
			\arrow[d,"\lambda^\vee\otimes\id_{P^\gp}"]
			\\
			N
			\arrow[r,"\Phi^*"]
			&
			M^\vee\otimes P^{\gp}
		\end{tikzcd}
		\]
    commutes, which amounts to that the bilinear map $\langle-,\lambda(-)\rangle\colon M\times M\rightarrow P^{\gp}$ is symmetric. For $P=\mathbb{N}$, the above condition (2) means that the symmetric bilinear map $\langle-,\lambda(-)\rangle$ is positive definite after we embed $\N$ into $\R$ by $1\mapsto 1$. By abuse of terminology, we will simply \emph{call $\lambda$ a polarization} of $(M,N,\Phi)$.
 
    We call $(M,N,\Phi)$ \emph{polarizable} if it has a polarization.  
    We call $(M,N,\Phi)$ \emph{pointwise polarizable}, if there exists an open subgroup $\Gamma'\subset\Gamma$ such that $(M,N,\Phi)$ is polarizable as an object of $\mathrm{Mor}_{\Gamma'}^{P}$.

    \begin{definition}
        We denote the full subcategory of $\MorGammaP$ consisting of the  pointwise polarizable objects by
        \[
        \MorGammaPpPol.
        \]
    \end{definition}

    \begin{definition}\label{def simpleness of LP}
        An object $(M,N,\Phi)\in \MorGammaPpPol$ is called \emph{simple}, if any morphism 
    \[\psi=(\psi_1,\psi_2)\colon(M',N',\Phi')\rightarrow(M,N,\Phi)\]
    in $\MorGammaPpPol$ with both $\psi_1$ and $\psi_2^\vee$ injective, must be an isogeny.
    \end{definition}

    Here are several properties of lattice pairings in $\MorGammaP, \MorGammaPpPol$ that are preserved under isogeny.
    \begin{proposition}\label{simpleness of LP is preserved by isogenies}
        \begin{enumerate}
            \item Let $(\psi_1,\psi_2): (M',N',\Phi')\to (M,N,\Phi)$ be an isogeny in $\MorGammaPpPol$. Then $(M',N',\Phi')$ is simple if and only if so is $(M,N,\Phi)$.

            \item Let $(\psi_1,\psi_2): (M',N',\Phi')\to (M,N,\Phi)$ be a morphism in $\MorGammaP$ with $\psi_1$ injective. If $(M,N,\Phi)$ is polarizable (resp. pointwise polarizable), then so is $(M',N',\Phi')$. In particular, if $(\psi_1,\psi_2)$ is an isogeny, then $(M,N,\Phi)$ is polarizable (resp. pointwise polarizable) if and only if so is $(M',N',\Phi')$.
        \end{enumerate}
    \end{proposition}
    \begin{proof}
        \begin{enumerate}
            \item Assume that $(M,N,\Phi)$ is simple. We show that $(M',N',\Phi')$ is simple. Let $(\iota_1,\iota_2):(M'_1,N'_1,\Phi'_1)\to (M',N',\Phi')$ be a morphism in $\MorGammaPpPol$ with $\iota_1$ and $\iota_2^\vee$ injective. Since $\psi_1\circ\iota_1$ and  $(\iota_2\circ\psi_2)^\vee$ are both injective and $(M,N,\Phi)$ is simple, $(\psi_1\circ\iota_1,\iota_2\circ\psi_2)$ is an isogeny. It follows that $(\iota_1,\iota_2)$ is also an isogeny. Therefore $(M',N',\Phi')$ is simple. 

            Assume that  $(M',N',\Phi')$ is simple. We show that $(M,N,\Phi)$ is simple. Let $n$ be a positive integer which kills both $\coker(\psi_1)$ and $\coker(\psi_2)$. Then we have $(n_{M'},n_{N'})=(\phi_1,\phi_2)\circ(\psi_1,\psi_2)$ with $(\phi_1,\phi_2):(M,N,\Phi)\to (M',N',\Phi')$ an isogeny. Then we are reduced to the above case.

            \item Suppose $(M,N,\Phi)$ is polarizable in $\Mor_{\Gamma'}^P$ with polarization $\lambda\colon M\to N$ for an open subgroup $\Gamma'$ of $\Gamma$. Then we claim that
            \[
            \lambda'=\psi_2\circ\lambda\circ\psi_1
            \colon
            M'\to N'
            \]
            is a polarization for $(M',N',\Phi')=(M',N',\langle-,-\rangle')$ in $\Mor_{\Gamma'}^P$. Indeed, for $m_1',m_2'\in M'$,
            \begin{align*}
            \langle m_1',\lambda'(m_2')\rangle'
            &
            =
            \langle m_1',(\psi_2\circ\lambda\circ\psi_1)(m_2')\rangle'
            =
            \langle\psi_1(m_1'),\lambda(\psi_1(m_2'))\rangle
            \\
            &
            =
            \langle\psi_1(m_2'),\lambda(\psi_1(m_1'))\rangle
            =
            \langle m_2',\lambda'(m_1')\rangle' .
            \end{align*}
            Moreover, if $m_1'\neq0$, then $\psi_1(m_1')\neq0$, thus
            \[
            \langle m_1',\lambda'(m_1')\rangle'
            =
            \langle\psi_1(m_1'),\lambda(\psi_1(m_1'))\rangle\in
            P\backslash\{0\}.
            \]
        \end{enumerate}
        
    \end{proof}

	We have the following simple observation.
	\begin{lemma}\label{pPol LP with simple factor is simple and the converse is true for rank 1 monoid}
			Let $(M,N,\Phi)\in \MorGammaPpPol$, then $M$ is a simple free $\Gamma$-module if and only so is $N$. Moreover, if $M$ is a simple free $\Gamma$-module, then $(M,N,\Phi)$ is simple.
			
			Now suppose $P=\mathbb{N}$.  If $(M,N,\Phi)$ is simple, then $M$ is a simple free $\Gamma$-module (or equivalently $N$ is a simple free $\Gamma$-module).
	\end{lemma}
	\begin{proof}
		We first show that for such a lattice pairing, $M$ is a simple free $\Gamma$-module if and only if so is $N$. We write $P^{\gp}=\mathbb{Z}^n$ and $\pi_i\colon P^{\gp}\rightarrow\mathbb{Z}$ for the projection to the $i$-th component. Then we have $\Gamma$-equivariant pairings
        \[
        \langle-,-\rangle_i=\pi_i\circ\langle-,-\rangle\colon M\times N\rightarrow\mathbb{Z},
        \quad
        i=1,\cdots,n.
        \]
        It is clear that $\langle-,-\rangle$ is the zero pairing if and only if $\langle-,-\rangle_i$ is the zero pairing for all $i=1,\cdots,n$. Now we argue by contradiction: suppose $M$ is a simple free $\Gamma$-module while $N$ is a free $\Gamma$-module that is not simple. Since $(M,N,\Phi)=(M,N,\langle-,-\rangle)$ is pointwise polarizable, $M$ and $N$ have the same $\mathbb{Z}$-rank, it follows that any $\Gamma$-equivariant pairing $M\times N\rightarrow\mathbb{Z}$ is zero, thus so are all these $\langle-,-\rangle_i$. Therefore $\langle-,-\rangle=0$, that is, $\Phi=0$. However, since $(M,N,\Phi)$ is polarizable as an object in $\mathrm{Mor}_{\Gamma'}^P$ for some open subgroup $\Gamma'\subset\Gamma$, in particular, there is a polarization $\lambda\colon M\rightarrow N$ of $(M,N,\Phi)\in\mathrm{Mor}_{\Gamma'}^P$. So for any non-zero $m\in M$, we have
        \[
        \langle m,\lambda(m)\rangle=
        \Phi(m)(\lambda(m))\neq0.
        \]
        This contradicts $\Phi=0$ and we have shown that $M$ is a simple free $\Gamma$-module if and only if so is $N$.

        The first part of the lemma is easy.
		
		For the second part, suppose that $M$ is not a simple free $\Gamma$-module, so in particular there is a proper free $\Gamma$-submodule $\widetilde{M}$ of $M$ such that $M/\widetilde{M}$ has rank $>0$. Then it is easy to see that $(\widetilde{M},\widetilde{M}^\vee,\mathrm{Id}_{\widetilde{M}})$ is an object in $\MorGammaP$ together with a morphism $\psi=(\psi_1,\psi_2)\colon(\widetilde{M},\widetilde{M}^\vee,\mathrm{Id}_{\widetilde{M}})\rightarrow(M,N,\Phi)$, where $\psi_1\colon\widetilde{M}\hookrightarrow M$ is the inclusion map and $\psi_2$ is given by the dual of the composition $\widetilde{M}\hookrightarrow M\xrightarrow{\Phi}N^\vee$ (see the commutative diagram below).
        \[
        \begin{tikzcd}
     \widetilde{M}
     \arrow[r,"\mathrm{Id}"]
     \arrow[d,hookrightarrow,"\psi_1"]
     &
     \widetilde{M}
     \arrow[d,"\psi_2^\vee"]
     \\
     M
     \arrow[r,"\Phi"]
     &
     N^\vee
       \end{tikzcd}
       \]
       By Proposition \ref{simpleness of LP is preserved by isogenies} (2), we know $(\widetilde{M},\widetilde{M}^\vee,\mathrm{Id}_{\widetilde{M}})\in\MorGammaPpPol$.      
       By definition, $(M,N,\Phi)$ is not simple. This finishes the proof.

	\end{proof}

  \begin{remark}
     The second part of Lemma \ref{pPol LP with simple factor is simple and the converse is true for rank 1 monoid} does not hold for sharp fs monoid $P$ with $\mathrm{rk}(P^\gp)>1$.
     For example, let $P=\N^2$, $M=\Z^2=N$, and we endow $M$ and $N$ with the trivial action of $\Gamma$. Any map $\Delta\colon M\rightarrow\mathrm{Hom}(N,P^{\mathrm{gp}})=(N^\vee)^2$ is determined by its two projections $\Delta_i\colon M\xrightarrow{\Delta}(N^\vee)^2\xrightarrow{\mathrm{pr_i}}N^\vee$. Under the canonical basis of $M$ and $N^\vee$, we let $\Phi_1$ and $\Phi_2$ be represented by $2\times2$-matrices  $A_1=\begin{pmatrix}
         1 & 0 \\ 0 & 1
     \end{pmatrix}$ and $A_2=\begin{pmatrix}
         1 & 1 \\ 1 & 2
     \end{pmatrix}$. And let $\Phi:M\to N^\vee\otimes P^\gp$ be the map corresponding to $(\Phi_1,\Phi_2)$. One can check that $(M,N,\Phi)\in\MorGammaPpPol$ with a polarization given by $\lambda=\mathrm{Id}_{\Z^2}:M\to N$.

     We can show that $(M,N,\Phi)$ is simple. Otherwise, there is a sub-object $(\widetilde{M}=\Z,\widetilde{N}=\Z,\widetilde{\Phi})$ of $(M,N,\Phi)$ where $\widetilde{\Phi}_1(1)=r_1$ and $\widetilde{\Phi}_2(1)=r_2$ for integers $r_1,r_2\in\Z$, that is, we have the following commutative diagrams
     \[
     \begin{tikzcd}
        \widetilde{M}
        \arrow[r,"\widetilde{\Phi}_1\colon1\mapsto r_1"]
        \arrow[d,hookrightarrow]
        &
        \widetilde{N}^\vee
        \arrow[d,hookrightarrow]
        &
        \widetilde{M}
        \arrow[r,"\widetilde{\Phi}_2\colon1\mapsto r_2"]
        \arrow[d,hookrightarrow]
        &
        \widetilde{N}^\vee
        \arrow[d,hookrightarrow]
        \\
        M
        \arrow[r,"\Phi_1"]
        &
        N^\vee
        &
        M
        \arrow[r,"\Phi_2"]
        &
        N^\vee
     \end{tikzcd}     
    \]
 It is easy to see that $r_1,r_2$ are both non-zero: indeed, the maps $\Phi_1,\Phi_2$ are injective (as the matrices $A_1,A_2$ are non-singular), thus the commutativity of the above diagrams shows that $\Phi_1|_{\widetilde{M}}=\widetilde{\Phi}_1$ is injective, and thus $r_1\neq0$. Similarly we have $r_2\neq0$.

 Choose a basis $\{e_1,e_2\}$ of $M$ such that $\widetilde{M}=\Z ae_1$ for a positive integer $a$, similarly choose a basis $\{e_1',e_2'\}$ of $N^\vee$ such that $\widetilde{N}^\vee=\Z a'e_1'$ for a positive integer $a'$. From the commutativity of the above diagrams, one deduces immediately that under these new basis, the morphisms $\Phi_1$ and $\Phi_2$ are both represented by upper triangular $2\times2$-matrices $B_1,B_2$ with entries in $\Z$. In particular, $B_1^{-1}B_2$ is an upper triangular matrix with entries in $\Q$. On the other hand, write $C$ (resp. $D$) for the transition matrix (which are elements in $\mathrm{GL}_2(\mathbb{Z})$) between the canonical basis and the new basis $\{e_1,e_2\}$ (resp. $\{e_1',e_2'\}$) of $M$ (resp. $N^\vee$). Therefore
 \[
B_1=D^{-1}A_1C,
\quad
B_2=D^{-1}A_2C.
 \]
 Thus the upper triangular matrix $B_1^{-1}B_2=C^{-1}A_1^{-1}A_2C$ is similar to the matrix $A_1^{-1}A_2=\begin{pmatrix}
     1 & 1 \\ 1 & 2
 \end{pmatrix}$ via the matrix $C\in\mathrm{GL}_2(\Z)$. The diagonal entries of $B_1^{-1}B_2\in\mathrm{GL}_2(\Q)$ are rational numbers and are equal to  the eigenvalues of $A_1^{-1}A_2$. However it is clear that the eigenvalues of $A_1^{-1}A_2$ are irrational numbers, which is a contradiction. This shows that $(M,N,\Phi)$ is a simple object in $\MorGammaPpPol$. On the other hand, neither $M$ nor $N$ is simple in $\Mod_{\Gamma}$.
 \end{remark}

\subsection{Lattice pairings and log 1-motives without abelian part}\label{Lattice pairings and log 1-motives without abelian part}
	Now we are going to relate the category $\MorGammaP$ (resp. $\MorGammaPpPol$) to the category $\Mcal^{\log,\mathrm{ab}=0}_1$ (resp. $\Mcal^{\log,\mathrm{ab}=0,\mathrm{pPol}}_1$) of log 1-motives (resp. pointwise polarizable log 1-motives) without abelian part over $S$. 
	
	Since the objects of $\Mcal^{\log,\mathrm{ab}=0}_1$ are certain two-term complexes, it carries a structure of additive category in a natural way. We refer to \cite[\href{https://stacks.math.columbia.edu/tag/09SE}{Tag 09SE}, \href{https://stacks.math.columbia.edu/tag/010M}{Tag 010M}]{Stacks-project} for the definition of additive category and related. For similar reason, the category $\MorGammaP$ is also naturally an additive category.
	
	Let $(M,N,\Phi)$ be an object of $\mathrm{Mor}_{\Gamma}^{P}$, and let $T:=\cHom_S(N,\Gm)=N^\vee\otimes\Gm$. The map $\Phi$ induces a map
$M\to N^\vee\otimes P^{\gp}_S$ which we still denote by $\Phi$. The chart $P_S\xrightarrow{\alpha} M_S$ induces a map $P^{\gp}_S\to \Gml$ which we still denote by $\alpha$. The composition
	\begin{equation}\label{construction of M_}
	 u_{\Phi}:M\xrightarrow{\Phi} N^\vee\otimes P^{\gp}_S\xrightarrow{\id_{N^\vee}\otimes\alpha} N^\vee\otimes\Gml=T_{\log}
	\end{equation}
	gives rise to a log 1-motive without abelian part $\Mbf_{\Phi}:=[M\xrightarrow{u_{\Phi}} T_{\log}]$. The association of $\Mbf_{\Phi}$ to $(M,N,\Phi)$ clearly gives rise to a functor
	\begin{equation}\label{funcotr from pairs of Galois modules to log 1-motives}
		\Mbf_{-}:\mathrm{Mor}_{\Gamma}^{P}\to \Mcal^{\log,\mathrm{ab}=0}_1,\quad (M,N,\Phi)\mapsto \Mbf_{\Phi}.
	\end{equation}

	The functor $\Mbf_{-}$ admits a left inverse constructed as follows. Let $\Mbf=[M\xrightarrow{u}T_{\log}]$ be a log 1-motive without abelian part, and let $N:=\cHom_S(T,\Gm)$ be the character group of the torus $T$. Consider the commutative diagram
	\[\xymatrix{M\ar[r]^u &T_{\log}\ar[r]^a\ar@{=}[d] &T_{\log}/T\ar@{=}[d] \\
	&N^\vee\otimes \Gml\ar[r]^-{\id_{N^\vee}\otimes b} &N^\vee\otimes (\Gml/\Gm)_{S_{\et}}}\]
	where $a$ and $b$ are the canonical quotient maps. The restriction of $(\Gml/\Gm)_{S_{\et}}$ to the small \'etale site of $\Spec \kbf$ is just the constant sheaf associated to $P^{\gp}$, and $M$ and $N$ are \'etale locally constant. Therefore the restriction of  $(\id_{N^\vee}\otimes b)\circ u$ to the small \'etale site of $\Spec \kbf$ is of the form
	\begin{equation}\label{construction of LP}
	M\xrightarrow{u}N^\vee\otimes M_S^{\gp}\xrightarrow{\id_{N^\vee}\otimes b}  N^\vee\otimes M_S^{\gp}/\Ocal_S^\times=N^\vee\otimes P^{\gp}_S
	\end{equation}
	which amounts to a map $\Phi_u:M\to N^\vee\otimes P^{\gp}$ in $\mathrm{Mod}_{\Gamma}$. The association of $(M,N,\Phi_u)$ to $\Mbf$ gives rise to a functor
	\begin{equation}\label{funcotr from log 1-motives without abelian part to pairs of Galois modules}
		\mathbf{LP}:\Mcal^{\log,\mathrm{ab}=0}_1\to \mathrm{Mor}_{\Gamma}^{P},\quad \Mbf=[M\xrightarrow{u}T_{\log}]\mapsto (M,N,\Phi_u).
	\end{equation}

        \begin{proposition}\label{the functor LP is left inverse to the functor M_-}
            The functor $\mathbf{LP}$ is left inverse to the functor $\Mbf_{-}$.
        \end{proposition}
        \begin{proof}
            Let $\psi=(\psi_1,\psi_2):(M,N,\Phi)\to (M',N',\Phi')$ be a morphism in $\mathrm{Mor}_{\Gamma}^{P}$. Then $\Mbf_{-}(\psi)$ is given by the outer rectangle of the commutative diagram
            \[\xymatrix{
            M\ar[r]^-{\Phi}\ar[d]_{\psi_1} &N^\vee\otimes P^{\gp}_S\ar[r]\ar[d]^{\psi_2^\vee\otimes\id} &N^\vee\otimes\Gml\ar[d]^{\psi_2^\vee\otimes\id} \\
            M'\ar[r]^-{\Phi'} &N^{\prime\vee}\otimes P^{\gp}_S\ar[r] &N^{\prime\vee}\otimes\Gml
            }.\]
		And thus $\mathbf{LP}(\Mbf_{-}(\psi))$ is given by restricting the outer rectangle of the commutative diagram
            \[\xymatrix{
            M\ar[r]^-{\Phi}\ar[d]_{\psi_1} &N^\vee\otimes P^{\gp}_S\ar[r]\ar[d]^{\psi_2^\vee\otimes\id} &N^\vee\otimes\Gml\ar[d]^{\psi_2^\vee\otimes\id}\ar[r] &N^\vee\otimes (\Gml/\Gm)_{S_{\et}}\ar[d]^{\psi_2^\vee\otimes\id}\\
            M'\ar[r]^-{\Phi'} &N^{\prime\vee}\otimes P^{\gp}_S\ar[r] &N^{\prime\vee}\otimes\Gml\ar[r] &N^{\prime\vee}\otimes (\Gml/\Gm)_{S_{\et}}
            }\]
		to the small \'etale site of $\Spec\kbf$. Since $P_S\to M_S\to M_S/\kbf^\times=P_S$ is the identity map, we have that $\mathbf{LP}(\Mbf_{\Phi})$ is given by $\Phi$, i.e. $\mathbf{LP}\circ\Mbf_{-}((M,N,\Phi))=\mathbf{LP}(\Mbf_{\Phi})=(M,N,\Phi)$. Similarly, we also have $\mathbf{LP}\circ\Mbf_{-}((M',N',\Phi'))=(M',N',\Phi')$ and it is easy to see that $\mathbf{LP}\circ\Mbf_{-}(\psi)=\psi$ from the above diagram. Hence $\mathbf{LP}$ is left inverse to $\Mbf_{-}$.
        
        \end{proof}

        \begin{proposition}\label{properties of the functor M_-}
		\begin{enumerate}
			\item The functor $\Mbf_{-}$ is additive and fully faithful.
                \item The functor $\Mbf_-$ sends the dual of $(M,N,\Phi)$ to the dual of $\Mbf_\Phi$.
			\item A map $\psi$ in $\mathrm{Mor}_{\Gamma}^{P}$ is an isogeny (resp. polarization) if and only if $\Mbf_{-}(\psi)$ is an isogeny (resp. polarization) of log 1-motives.
			\item The functor $\Mbf_{-}$ restricts to a functor
			\[\Mbf_{-}:\mathrm{Mor}_{\Gamma}^{P,\mathrm{pPol}}\to \Mcal^{\log,\mathrm{ab}=0,\mathrm{pPol}}_1.\]
		\end{enumerate}
	\end{proposition}
	\begin{proof}
		Let $(M,N,\Phi)$, $(M',N',\Phi')$ be two objects of $\mathrm{Mor}_{\Gamma}^{P}$. A morphism 
		\[\Mbf_{\Phi}=[M\xrightarrow{u_{\Phi}} T_{\log}]\to \Mbf_{\Phi'}=[M'\xrightarrow{u_{\Phi'}} T'_{\log}]\]
		of log 1-motives consists of $f_{-1}:M\to M'$ and $f_0:T_{\log}\to T'_{\log}$ satisfying $f_0\circ u_{\Phi}=u_{\Phi'}\circ f_{-1}$. By \cite[Proposition 2.5]{KajiwaraKatoNakayama2008}, we have $\Hom_S(T_{\log},T'_{\log})=\Hom_S(T,T')=\Hom_{\Mod_{\Gamma}}(N',N)$. Let $g\in \Hom_{\Mod_{\Gamma}}(N',N)$ corresponding to $f_0$, and the equality $f_0\circ u_{\Phi}=u_{\Phi'}\circ f_{-1}$ amounts to $(g^\vee\otimes\id_{\Gml})\circ \Phi=\Phi'\circ f_{-1}$. And the latter equality implies that $(g^\vee\otimes\id_{P^\gp})\circ \Phi=\Phi'\circ f_{-1}$, and one can check that $(f_{-1},f_0)=\Mbf_{-}((f_{-1},g))$. It follows that the functor $\Mbf_{-}$ is full. The faithfulness follows from Proposition \ref{the functor LP is left inverse to the functor M_-}.
		
		The rest is easy, and we leave it to the reader.
	\end{proof}
	
	\begin{proposition}\label{properties of the functor LP}
		\begin{enumerate}
			\item The functor $\mathbf{LP}$ is additive and fully faithful.
                \item The functor $\LP$ sends the dual of $\Mbf$ to the dual of $\LP(\Mbf)$.
			\item A map $\psi$ in $\Mcal^{\log,\mathrm{ab}=0}_1$ is an isogeny (resp. polarization) if and only if 
            $\mathbf{LP}(\psi)$ is an isogeny (resp. polarization) in $\MorGammaP$. 
        	\item The functor $\mathbf{LP}$ restricts to a functor 
			\[\mathbf{LP}:\Mcal^{\log,\mathrm{ab}=0,\mathrm{pPol}}_1\to \mathrm{Mor}_{\Gamma}^{P,\mathrm{pPol}}.\]
		\end{enumerate}
        \end{proposition}
	\begin{proof}
            The proofs are easy, and we leave them to the reader.
	\end{proof}
	
	\begin{corollary}
		Recall that $\mathrm{LAV}^{\mathrm{ab}=0}$ denotes the category of log abelian varieties without abelian part over $S$. The functors $\Mbf_{-}$ and $\mathbf{LP}$ induce functors
		\[\mathrm{Mor}_{\Gamma}^{P,\mathrm{pPol}}\to \mathrm{LAV}^{\mathrm{ab}=0}\quad 
		\text{and} \quad \mathrm{LAV}^{\mathrm{ab}=0}\to \mathrm{Mor}_{\Gamma}^{P,\mathrm{pPol}}\]
		respectively. And the second functor is left inverse to the first functor.
	\end{corollary}
	\begin{proof}
		This follows from Proposition \ref{the functor LP is left inverse to the functor M_-} and Theorem \ref{equivalence from log 1-motives to LAVwCD}.
	\end{proof}
	
	\begin{proposition}\label{decomposition of log 1-motives}
		Let $\Mbf=[M\xrightarrow{u}T_{\log}]$ be an object of $\Mcal^{\log,\mathrm{ab}=0}_1$. We denote $\Mbf_{-}\circ\mathbf{LP}(\Mbf)$ by $[M\xrightarrow{u^{\log}}T_{\log}]$, and let $u^{\mathrm{c}}:=u-u^{\log}$. Then
		\begin{enumerate}
			\item $u^{\mathrm{c}}$ factors through $T\hookrightarrow T_{\log}$;
			\item $\Mbf_{-}\circ\mathbf{LP}([M\xrightarrow{u^{\log}}T_{\log}])=[M\xrightarrow{u^{\log}}T_{\log}]$;
			\item and $\Mbf$ lies in the image of the functor $\Mbf_{-}$ if and only if $u^{\mathrm{c}}=0$.
		\end{enumerate}
	\end{proposition}
	\begin{proof}
		Part (1) follows from the claim that the composition
		\[M\xrightarrow{u^{\mathrm{c}}} T_{\log}\xrightarrow{a} T_{\log}/T=N^\vee\otimes (\Gml/\Gm)_{S_{\et}}\]
		is the zero map. We are reduced to show the claim. Since both $M$ and $N$ are \'etale locally constant, it suffices to check the vanishing of the composition
		\[M\xrightarrow{u^{\mathrm{c}}} N^\vee\otimes M_S^{\gp}\xrightarrow{\id_{N^\vee}\otimes b} N^\vee\otimes M_S^{\gp}/\Ocal_S^\times=N^\vee\otimes P^{\gp}_S\]
		on the small \'etale site of $\Spec\kbf$, where $P^{\gp}_S$ denotes the constant sheaf associated to $P^{\gp}$. By the constructions (\ref{construction of M_}) and (\ref{construction of LP}), we have that $(\id_{N^\vee}\otimes b)\circ u^{\log}$ is just the composition
		\begin{align*}
		M\xrightarrow{u} N^\vee\otimes M_S^{\gp}\xrightarrow{\id_{N^\vee}\otimes b} N^\vee\otimes P^{\gp}_S\xrightarrow{\id_{N^\vee}\otimes \alpha} N^\vee\otimes M_S^{\gp}\xrightarrow{\id_{N^\vee}\otimes b}N^\vee\otimes P^{\gp}_S.
		\end{align*}
		Since $(\id_{N^\vee}\otimes b)\circ (\id_{N^\vee}\otimes \alpha)=\id$, we get $(\id_{N^\vee}\otimes b)\circ u^{\log}=(\id_{N^\vee}\otimes b)\circ u$ and thus $(\id_{N^\vee}\otimes b)\circ u^{\mathrm{c}}=0$.
		
		Part (2) follows from $\mathbf{LP}\circ\Mbf_{-}=\id_{\mathrm{Mor}_{\Gamma}^{P}}$.
		
		Assume $u^{\mathrm{c}}=0$, then $u=u^{\log}$, and thus $\Mbf=[M\xrightarrow{u}T_{\log}]=[M\xrightarrow{u^{\log}}T_{\log}]=\Mbf_{-}\circ\mathbf{LP}(\Mbf)$ lies in the image of $\Mbf_{-}$. Conversely, assume $\Mbf=\Mbf_{-}((M,N,\Phi))$ for some $(M,N,\Phi)\in\mathrm{Mor}_{\Gamma}^P$. By definition $u^{\log}$ is the map for the log 1-motive 
		\begin{align*}
		\Mbf_{-}\circ\mathbf{LP}(\Mbf)=\Mbf_{-}\circ\mathbf{LP}(\Mbf_{-}((M,N,\Phi)))=&\Mbf_{-}\circ(\mathbf{LP}\circ \Mbf_{-})((M,N,\Phi)) \\
		=&\Mbf_{-}((M,N,\Phi))=\Mbf,
		\end{align*}
		i.e. $u^{\log}=u$, and thus $u^{\mathrm{c}}=0$.
	\end{proof}
	
	\begin{definition}
		Let $\Mbf=[M\xrightarrow{u}T_{\log}]$ be an object of $\Mcal^{\log,\mathrm{ab}=0}_1$. We call $[M\xrightarrow{u^{\log}}T_{\log}]$ (resp. $[M\xrightarrow{u^{\mathrm{c}}}T_{\log}]$) the \emph{log (resp. classical) part} of $\Mbf$ with respect to the given chart $P_S\to M_S$. We also call $[M\xrightarrow{u^{\mathrm{c}}}T]$ the \emph{classical part} of $\Mbf$ by abuse of terminology.
	\end{definition}

	We denote by $\MorGammaP\otimes\Q$ the isogeny category associated to $\MorGammaP$, i.e. $\MorGammaP\otimes\Q$ has the same objects as $\MorGammaP$ and 
	\[\Hom_{\MorGammaP\otimes\Q}((M,N,\Phi),(M',N',\Phi'))=\Hom_{\MorGammaP}((M,N,\Phi),(M',N',\Phi'))\otimes\Q.\]
	Similarly we denote by $\Mcal^{\log,\mathrm{ab}=0}_1\otimes\Q$ (resp. $\mathrm{LAV}^{\mathrm{ab}=0}\otimes\Q$) the isogeny category associated to $\Mcal^{\log,\mathrm{ab}=0}_1$ (resp. $\mathrm{LAV}^{\mathrm{ab}=0}$).
	
	\begin{lemma}
        Let $\mathcal{C}$ be either $\mathrm{Mor}_{\Gamma}^{P}$, or 
        $\Mcal^{\log,\mathrm{ab}=0}_1$, or $\mathrm{LAV}^{\mathrm{ab}=0}$. 
        \begin{enumerate}
            \item For any isogeny $\psi:\mathfrak{a}\to\mathfrak{b}$ in $\mathcal{C}$, there exist a map $\rho:\mathfrak{b}\to\mathfrak{a}$ in $\mathcal{C}$ and a positive integer $n$ such that $\rho\circ\psi=n_{\mathfrak{a}}$ and $\psi\circ\rho=n_{\mathfrak{b}}$, where $n_{(-)}$ denotes the multiplication-by-$n$ map on $(-)$.
            \item A map in $\mathcal{C}\otimes\Q$ is an isomorphism if and only if it is of the form $\psi\otimes a$ with $\psi$ an isogeny in $\mathcal{C}$ and $a\in\Q^\times$.
        \end{enumerate}
	\end{lemma}
	\begin{proof}
        Clearly part (2) follows from part (1). Below we show part (1).
		
		First we consider the case of $\mathrm{Mor}_{\Gamma}^{P}$. The isogeny $\psi$ amounts to a pair $(M\xrightarrow{\psi_1}M',N'\xrightarrow{\psi_2}N)$  in $\mathrm{Mod}_{\Gamma}$ such that the diagram 
		\[\xymatrix{
			M\ar[r]^-{\Phi}\ar[d]_{\psi_1} &N^\vee\otimes P^{\gp}\ar[d]^{\psi_2^\vee\otimes\id_{P^{\gp}}} \\
			M'\ar[r]^-{\Phi'} &N^{'\vee}\otimes P^{\gp}    
		}\]
		is commutative, and $\psi_1$ and $\psi_2$ are injective with finite cokernel. By Lemma \ref{isogeny between free Z-modules}, there exists a positive integer $n$,  $\rho_1:M'\to M$ and $\rho_2:N\to N'$  such that $\rho_1\circ\psi_1=n_M$, $\psi_1\circ\rho_1=n_{M'}$, $\psi_2\circ\rho_2=n_N$, and $\rho_2\circ\psi_2=n_{N'}$. Then the pair $(\rho_1,\rho_2)$ is the map with the required property. 
		
		Now we consider the case of $\Mcal^{\log,\mathrm{ab}=0}_1$. Given an isogeny 
		\[\psi=(f_{-1},f_0):[M\xrightarrow{u}T_{\log}]\to [M'\xrightarrow{u'}T'_{\log}],\]
		we have that $f_{-1}$ is injective with finite cokernel and the map $f_0^{\mathrm{c}}:T\to T'$ induced by $f_0$ is an isogeny of tori. Let $n$ be a positive integer such that n kills $\mathrm{coker}(f_{-1})$ and $\mathrm{ker}(f_0^{\mathrm{c}})$. Then there exists a map $g_{-1}:M'\to M$ and a map $g_0^{\mathrm{c}}:T'\to T$ such that $g_{-1}\circ f_{-1}=n_M$, $g_0^{\mathrm{c}}\circ f_0^{\mathrm{c}}=n_T$, $f_{-1}\circ g_{-1}=n_{M'}$, and $f_0^{\mathrm{c}}\circ g_0^{\mathrm{c}}=n_{T'}$. Let $g_0:T'_{\log}\to T_{\log}$ be the map induced by $g_0^{\mathrm{c}}$. Then the pair $(g_{-1},g_0)$ is the map with required property.
		
		At last, the case of $\mathrm{LAV}^{\mathrm{ab}=0}$ follows from the case of $\Mcal^{\log,\mathrm{ab}=0}_1$ with the help of Theorem \ref{equivalence from log 1-motives to LAVwCD} and Proposition \ref{Hom of LAVs is isogeny iff the corresponding hom of log 1-motives is isogeny}.
	\end{proof} 

        \begin{proposition}\label{log 1-motive with ab=0 is simple iff its LP is simple}
            Let $\Mbf=[M\xrightarrow{u}T_{\log}]$ be an object of $\Mcal^{\log,\mathrm{ab}=0,\mathrm{pPol}}_1$ and let $(M,N,\Phi):=\LP(\Mbf)$ which belongs to $\mathrm{Mor}_{\Gamma}^{P,\mathrm{pPol}}$ by Proposition \ref{properties of the functor LP} (4). 
            \begin{enumerate}
                \item Assume that $(M,N,\Phi)$ is simple in $\mathrm{Mor}_{\Gamma}^{P,\mathrm{pPol}}$. Then $\Mbf$ is simple in $\Mcal^{\log,\mathrm{ab}=0,\mathrm{pPol}}_1$.
                \item Assume that $\Mbf$ is simple in $\Mcal^{\log,\mathrm{ab}=0,\mathrm{pPol}}_1$ and $\kbf$ is a finite field. Then  $(M,N,\Phi)$ is simple in $\mathrm{Mor}_{\Gamma}^{P,\mathrm{pPol}}$. 
            \end{enumerate} 
        \end{proposition}
        \begin{proof}
            Assume that $(M,N,\Phi)$ is simple. We show that $\Mbf$ is simple. It suffices to show that any morphism $(f_{-1},f_0):\Mbf'=[M'\xrightarrow{u'}T'_{\log}]\to [M\xrightarrow{u}T_{\log}]=\Mbf$ with both $f_{-1}$ and $f_0$ injective has to be an isogeny. Let $f_{\mathrm{c}}:T\to T'$ be the morphism induced by $f_0$, and let $f_{\mathrm{l}}:N'\to N$ be the corresponding morphism of character groups. The injectivity of $f_0$ implies that of $f_{\mathrm{c}}$, and thus $f_{\mathrm{l}}$ is surjective. Further one sees that $f_{\mathrm{l}}^\vee$ is injective. Since $(M,N,\Phi)$ is simple, the  morphism $\LP((f_{-1},f_0))=(f_{-1},f_{\mathrm{l}})$ has to be an isogeny by definition (see Definition \ref{def simpleness of LP}). Then $(f_{-1},f_0)$ is also an isogeny by Proposition \ref{properties of the functor LP} (3).
         
            Assume that $\Mbf$ is simple and $\kbf$ is finite. We show that $\LP(\Mbf)$ is simple. Let $[M\xrightarrow{u^{\log}}T_{\log}]:=\Mbf_{-}\circ\LP(\Mbf)$ and $u^{\mathrm{c}}:M\to T$ be as in Proposition \ref{decomposition of log 1-motives}. Since $\kbf$ is finite, there exists a positive integer $n$ such that $nu^{\mathrm{c}}=0$. Then $(1_M,n_{T_{\log}}):\Mbf\to [M\xrightarrow{nu}T_{\log}]$ is an isogeny and 
            \[\Mbf_{-}((M,N,n\Phi))=\Mbf_{-}\circ\LP([M\xrightarrow{nu}T_{\log}])=[M\xrightarrow{nu}T_{\log}].\]
            By Proposition \ref{simpleness of ppol log 1-motives is preserved by isogeny}, $[M\xrightarrow{nu}T_{\log}]$ is simple. By Proposition \ref{properties of the functor LP} (3), $\LP(1_M,n_{T_{\log}})$ is an isogeny, and thus we are reduced to show that $\LP([M\xrightarrow{nu}T_{\log}])=(M,N,n\Phi)$ is simple by Proposition \ref{simpleness of LP is preserved by isogenies} (1). 
            Let 
            \[(\psi_1,\psi_2):(M',N',\Phi')\to (M,N,n\Phi)\]
            be a morphism in $\MorGammaP$ with both $\psi_1$ and $\psi_2^\vee$ injective. We claim that $\coker(\psi_2)$ is finite. For this, we may assume that both $N$ and $N'$ have trivial $\Gamma$-action. Then the result follows from the exact sequence
            \[0\to\Hom_{\Z}(\coker(\psi_2^\vee),\Z)\to N\xrightarrow{\psi_2} N'\to\Ext^1_{\Z}(\coker(\psi_2^\vee),\Z).\]
            Let $\overline{N}':=\mathrm{im}(\psi_2)$ and thus $\psi_2$ factors as $N\xrightarrow{\overline{\psi}_2}\overline{N}'\xrightarrow{\iota} N'$. Then the two factors of $\Mbf_{-}(\psi_1,\overline{\psi}_2)$ are both injective. By the simpleness of $\Mbf_{-}((M,N,n\Phi))=[M\xrightarrow{nu}T_{\log}]$, $\Mbf_{-}(\psi_1,\overline{\psi}_2)$ has to be an isogeny, and thus $(\psi_1,\overline{\psi}_2)$ is an isogeny by Proposition \ref{properties of the functor M_-} (3). It follows that $(\psi_1,\psi_2)$ is also an isogeny. Therefore $\LP(\Mbf)$ is simple.
        \end{proof}

	\section{Log abelian varieties over finite log points}\label{log av over finite log points}

    We assume in the rest of this article that $\kbf=\mathbb{F}_q$ is finite, so as in (\ref{description of Galois gp of finite field}), we have
	\[	 
        \Gamma=\Gal(\overline{\kbf}/\kbf)
        =\overline{\langle\gamma\rangle}
        \text{ with } \gamma:\overline{\kbf}\to\overline{\kbf},x\mapsto x^q.
	\]
	
	\subsection{Decomposition of log abelian varieties up to isogeny over finite log points}\label{Decomposition of log abelian varieties up to isogeny over finite log points}
        Assume that $S=(\Spec\kbf,M_S)$ be an fs log point.
	
	\begin{lemma}\label{torsion phenomenon over finite fields}
		Let $Y$ be a group scheme over $\kbf$ which is \'etale locally 
		isomorphic to a free abelian group of finite rank. Let $G$ be an 
		extension of an abelian variety $B$ by a torus $T$ over $\kbf$.
		\begin{enumerate}
			\item The groups $\Hom_{\kbf}(Y,G)$ and $\Ext_{\kbf}^1(B,T)$ are
			finite.
			\item Let $\Mbf:=[Y\xrightarrow{u}G_{\log}]$ be a log 1-motive
			over $S$. Then there exists a positive integer $n$ s.t. 
			$Y\xrightarrow{nu}G_{\log}$ factors through 
			$T_{\log}\hookrightarrow G_{\log}$.
			\item There exists a positive integer $m$ such that the
			pullback of the extension $G$ of $B$ by $T$ along 
			$B\xrightarrow{m}B$ splits.
		\end{enumerate}
	\end{lemma}
	\begin{proof}
		(1) Let $\kbf'$ be a finite field extension of $\kbf$ such that $Y\times_{\Spec\kbf}\Spec\kbf'\cong\Z^r$ and $T\times_{\Spec\kbf}\Spec\kbf'\cong\Gm^s$. 
        
        We have $\Hom_{\kbf}(Y,G)\hookrightarrow \Hom_{\kbf'}(Y,G)=G(\kbf')^r$. Since $\kbf'$ is a finite field, $G(\kbf')$ is finite and thus $\Hom_{\kbf}(Y,G)$ is finite.
        
        Since $\cHom_{\kbf}(B,T)=0$, we have $\Ext^1_{\kbf}(B,T)\cong H^0(\Spec\kbf,\cExt^1_{\kbf}(B,T))$. We also have 
        \begin{align*}
            H^0(\Spec\kbf,\cExt^1_{\kbf}(B,T))\hookrightarrow &H^0(\Spec\kbf',\cExt^1_{\kbf}(B,T)) \\
            =&H^0(\Spec\kbf',\cExt^1_{\kbf'}(B\times_{\Spec\kbf}\Spec\kbf',\Gm^s))
            \\=&B^*(\kbf')^s,
        \end{align*}
        where $B^*$ denotes the dual abelian variety of $B$.
        Since $B^*(\kbf')$ is finite, we get the finiteness of $\Ext^1_{\kbf}(B,T)$.
		
		(2) Let $v$ be the composition $Y\xrightarrow{u}G_{\log}\to B$. By
		(1), there exists a positive integer $n$ such that $nv=0$. Then
		the result follows.
		
		(3) The result follows from (1).	     
	\end{proof}

    \begin{theorem}\label{Honda-Tate theorem for 1-motives}
        \begin{enumerate}
            \item The simple objects in $\Mcal_1$ are of the form $Y[1]$, $T$ and $B$, where $Y$ is a simple lattice (i.e. a simple object in $\Mod_{\Gamma}$), $T$ is a torus with its character group a simple lattice and $B$ is a simple abelian variety. 
            \item Let $[Y\xrightarrow{u}G]$ be a 1-motive over $\kbf$. Then $[Y\xrightarrow{u}G]$ is isogenous to the 1-motive $[Y\xrightarrow{0}T\times B]$ which is the direct product of $Y[1]$, $T$ and $B$, and thus each 1-motive is isogenous to a direct product of simple 1-motives.
            \item
            \emph{(Honda-Tate theorem for 1-motives)}
            Let 
            \[\qW(2)^0:=\{q\alpha\mid \alpha\in\qW(0)\}\subsetneq\qW(2).\]
            The bijection of sets from Theorem \ref{Honda-Tate for AVs} extends to a bijection of sets
            \[
            \left\{
		\text{simple 1-motives over }\kbf
		\right\}/\text{isogeny} \xrightarrow{\simeq}           
            \qW(0)\bigsqcup \qW(1) \bigsqcup \qW(2)^0
            \]
            where $Y[1]$ (resp. $B$, resp. $T$) a simple lattice (resp. abelian variety, resp. torus) is mapped to the Galois conjugacy class of a root of the characteristic polynomial $P_{Y,\pi_Y}$ (resp. $P_{B,\pi_B}$, resp. $P_{T,\pi_T}$).
        \end{enumerate} 
        
    \end{theorem}
    \begin{proof}
        (1) is just a restatement of Remark \ref{simple 1-motives}. (2) follows from (1) and Lemma \ref{torsion phenomenon over finite fields}. (3) follows from (1) and Lemma \ref{charpoly of lattices, AVs and tori}.
    \end{proof}
	
	\begin{corollary}\label{log 1-motive splits up to isogeny}
		Let $\Mbf:=[Y\xrightarrow{u}G_{\log}]$ be a log 1-motive over $S$ 
		as in Lemma \ref{torsion phenomenon over finite fields}, let $m$ and $n$ be as in Lemma \ref{torsion phenomenon over finite fields}, and let $i_{\log}:T_{\log}\to G_{\log}$  be the inclusion induced by the canonical inclusion $i:T\to G$. Then $n\cdot u$ factors as $Y\xrightarrow{u_1}T_{\log}\xrightarrow{i_{\log}} G_{\log}$, and we have a morphism 
            \begin{equation}\label{a morphism of log 1-motives}
                \xymatrix{
			Y\ar[r]^-{(u_1,0)}\ar[d]_{n_Y} &T_{\log}\times B\ar[d]^{f_{\log}} \\
			Y\ar[r]^-u &G_{\log}
		}
            \end{equation}
		of log 1-motives, where $f_{\log}$ is induced by the map $f$ in the pullback diagram 
        \[\xymatrix{0\ar[r] &T\ar[r]^-{(\id_T,0)}\ar@{=}[d] &T\times B\ar[r]\ar[d]^f &B\ar[d]^{m_B}\ar[r] &0\\ 
        0\ar[r] &T\ar[r]^i &G\ar[r] &B\ar[r] &0}.\]
	\end{corollary}
	\begin{proof}
            The factorization $n\cdot u= i_{\log}\circ u_1$ is clear by Lemma \ref{torsion phenomenon over finite fields}. By the above pullback diagram, we have $i=f\circ (\id_T,0)$ and thus $i_{\log}=f_{\log}\circ (\id_{T_{\log}},0)$. Therefore 
            \[f_{\log}\circ (u_1,0)=f_{\log}\circ ((\id_{T_{\log}},0)\circ u_1)=i_{\log}\circ u_1=n\cdot u,\]
            and thus \eqref{a morphism of log 1-motives} gives rise to a morphism of log 1-motives.
	\end{proof}

	\begin{theorem}\label{isogeneous decomposition of log AV}
		Let $A$ be a log abelian variety over $S$. Let $\Mbf:=[Y\xrightarrow{u}G_{\log}]$ be the pointwise polarizable log 1-motive over $S$ corresponding to $A$ (see Theorem \ref{equivalence from log 1-motives to LAVwCD}), and let $T$ and $B$ be the torus and the abelian part of $G$ respectively. Let $n$ be a positive integer such that $n\cdot u$ factors through $T_{\log}\to G_{\log}$, and let $(u_1,0):Y\to T_{\log}\times B$ be the map from Corollary \ref{log 1-motive splits up to isogeny}. Then 
		\begin{enumerate}
			\item the log 1-motive $\Mbf_1:=[Y\xrightarrow{u_1}T_{\log}]$ is pointwise polarizable;
			\item $A$ is isogenous to $A_1\times B$, where $A_1$ is the log abelian variety over $S$ corresponding to $\Mbf_1$;
                \item if moreover $A$ is isogenous to $A_1'\times B'$ with $A_1'$ totally degenerate and $B'$ an abelian variety, then $A_1'$ is isogenous to $A_1$ and $B'$ is isogenous to $B$.
		\end{enumerate}
	\end{theorem}
	\begin{proof}
		(1) We may assume that $\Mbf$ is polarizable. Let $X$ denote the character group of $T$, $T^*$ denote the torus with character group $Y$, and $B^*$ denote the dual abelian variety of $B$. Let $\Mbf^*=[X\xrightarrow{u^*}G^*_{\log}]$ be the dual log 1-motive of $\Mbf$, and let $(h_{-1},h_0):\Mbf\to \Mbf^*$ be a polarization. 
		
        Let $\Mbf_{\mathrm{s}}$ be the log 1-motive $[Y\xrightarrow{(u_1,0)}T_{\log}\times B]$, and we have a morphism 
        \[(g_{-1},g_0):=(n_Y,f_{\log}):\Mbf_{\mathrm{s}}\to \Mbf\]
        of log 1-motives as in Corollary \ref{log 1-motive splits up to isogeny}.  Then the dual of $\Mbf_{\mathrm{s}}$ is $\Mbf_{\mathrm{s}}^*=[X\xrightarrow{(u_1^*,0)}T^*_{\log}\times B^*]$.
		
		Let $(g_{-1}^*,g_0^*)$ be the dual of $(g_{-1},g_0)$, and let $\gamma_{-1}:=g_{-1}^*\circ h_{-1}\circ g_{-1}$ and $\gamma_0:=g_{0}^*\circ h_{0}\circ g_{0}$. Then $(\gamma_{-1},\gamma_0):\Mbf_{\mathrm{s}}\to \Mbf_{\mathrm{s}}^*$ is a polarization by \cite[Lemma 3.4]{Zhao2017}. Clearly $(\gamma_{-1},\gamma_0)$ induces a polarization $\Mbf_1\to \Mbf_1^*=[X\xrightarrow{u_1^*}T^*_{\log}]$.
		
		(2) The log abelian variety associated to the pointwise polarizable log 1-motive $\Mbf_{\mathrm{s}}$ is clearly just $A_1\times B$. By \cite[Proposition 3.3]{Zhao2017}, the morphism $(g_{-1},g_0)$ induces an isogeny $g:A_1\times B\to A$.

            (3) Apparently there is no non-trivial morphism between abelian varieties and log 1-motives without abelian part, therefore there is no non-trivial morphism between abelian varieties and totally degenerate log abelian varieties. Then the result follows from \cite[Cor. 3.1]{Zhao2017}.
      
	\end{proof}

	\subsection{Isogeny classes of log abelian varieties without abelian part over finite log points}\label{subsection LAV without abelian part over finite log points are classified by LP up to isogeny}

    We further assume that $S$ is as in Subsection \ref{subsection lattice pairings}.
    By Theorem \ref{isogeneous decomposition of log AV}, we know that the classification of isogeny classes of log abelian varieties $A$ (with corresponding log 1-motive $[Y\rightarrow G_{\log}]$) over $S$ is reduced to the following two problems
	\begin{enumerate}
		\item 
		the classification of isogeny classes of abelian varieties $B$ over $\mathbf{k}$;

		\item 
		the classification of isogeny classes of log abelian varieties $A_1$ (corresponding to $[Y\rightarrow T_{\log}]$) over $S$.
	\end{enumerate}
	Problem (1) is answered by the classical Honda-Tate theorem for abelian varieties. We reduce Problem (2) to the study of $\MorGammaPpPol$, where $P$ is the chart monoid of $S$ as in Subsection \ref{subsection lattice pairings} and $\Gamma=\Gal(\overline{\kbf}/\kbf)=\overline{\langle\gamma\rangle}$ is as in the beginning of this section.

    Under the assumption that $\kbf$ is finite, the functors $\Mbf_{-}$ and $\mathbf{LP}$ admit better properties. 
    
	\begin{theorem}
		The functor 
		\[\Mbf_{-}\otimes\Q:\mathrm{Mor}_{\Gamma}^{P}\otimes\Q\to \Mcal^{\log,\mathrm{ab}=0}_1\otimes\Q,\quad (M,N,\Phi)\mapsto \Mbf_{\Phi}\]
		induced by $\Mbf_{-}$ is an equivalence of categories with a quasi-inverse $\mathbf{LP}\otimes\Q$.
	\end{theorem}
	\begin{proof}
		Since $\Mbf_{-}$ is fully faithful, so is $\Mbf_{-}\otimes\Q$. We are left to prove that $\Mbf_{-}\otimes\Q$ is essentially surjective.
		
		Let $\Mbf=[M\xrightarrow{u} T_{\log}]$ be a log 1-motive without abelian part, and let $N$ be the character group of the torus $T$. Let $[M\xrightarrow{u^{\log}} T_{\log}]$ and $[M\xrightarrow{u^{\mathrm{c}}} T]$ be the log part and the classical part respectively. By Lemma \ref{torsion phenomenon over finite fields} (1), $u^{\mathrm{c}}$ is killed by some positive integer $r$. Then $ru=r(u^{\log}+u^{\mathrm{c}})=ru^{\log}$. Since $[M\xrightarrow{u^{\log}} T_{\log}]$ lies in the image of the additive functor $\Mbf_-$, so is $[M\xrightarrow{ru} T_{\log}]$ which is isogeous to $\Mbf$. This finishes the proof.
	\end{proof}
	
	\begin{corollary}\label{equivalence from pairs of Galois modules to LAVs}
		The functors $\Mbf_{-}\otimes\Q$ and $\mathbf{LP}\otimes\Q$ restrict to functors
		\[\xymatrix{
			&\mathrm{Mor}_{\Gamma}^{P,\mathrm{pPol}}\otimes\Q\ar@<1ex>[r]^-{\Mbf_{-}\otimes\Q}   &\Mcal^{\log,\mathrm{ab}=0,\mathrm{pPol}}_1\otimes\Q\ar@<1ex>[l]^-{\mathbf{LP}\otimes\Q}
		}\]
		which are quasi-inverse to each other. Furthermore, by composing with the equivalence from Theorem \ref{equivalence from log 1-motives to LAVwCD}, $\Mbf_{-}\otimes\Q$ induces an equivalence of categories
		\[\mathrm{Mor}_{\Gamma}^{P,\mathrm{pPol}}\otimes\Q\simeq \mathrm{LAV}^{\mathrm{ab}=0}\otimes\Q.\]
	\end{corollary}
	
	By Corollary \ref{equivalence from pairs of Galois modules to LAVs}, to classify isogeny classes of log abelian varieties without abelian part over $S$, it suffices to study the category $\mathrm{Mor}_{\Gamma}^{P,\mathrm{pPol}}\otimes\Q$.

\section{Isogeny classes in $\MorGammaPpPol$ over the standard finite log point}\label{Isogeny classes in MorGammaPpPol over the standard finite log point}
    In this sectoin, let $\kbf=\F_q$, $\Gamma=\overline{\langle\gamma\rangle}$, $P$, and $S$ be as in Subsection \ref{subsection LAV without abelian part over finite log points are classified by LP up to isogeny}. 
    We further assume that $P=\N$, i.e. $S$ is the standard log point with underlying scheme $\Spec\kbf=\Spec\mathbb{F}_q$. Then the categories $\Mod_{\Gamma}$ and $\MorGammaNpPol$ admit explicit descriptions as we will give in this section.
    Note that in this case, for $(M,N,\Phi)\in\MorGammaNpPol$, the corresponding bilinear form $\langle-,-\rangle$ takes value in $P^\gp=\Z$ and thus we can talk about the non-degeneracy of $\langle-,-\rangle$.

        \subsection{Representation theory of $\Gamma$}
        In this subsection, we collect some facts on the theory of rational representations of $\Gamma$. Then we deduce some simple properties of the two categories $\Mod_\Gamma$ and $\MorGammaNpPol$.

	We fix a compatible family of roots of unity $(\zeta_{r})_{r>0}$ (that is, $\zeta_{rr'}^r=\zeta_{r'}$ for any $r,r'>0$). For a positive integer $n$, we write the (unique) finite cyclic quotient of order $n$ of $\Gamma$ as follows
	\[
	\Gamma_n:=\Gamma/\Gamma^n\simeq\mathbb{Z}/n\mathbb{Z}.
	\]
        We still denote the image of $\gamma$ in $\Gamma_n$ by $\gamma$. Then we have an isomorphism of $\mathbb{Q}$-algebras induced by the projections $\Gamma_n\rightarrow\Gamma_r$ for all $r\mid n$:
	\begin{equation}\label{Q[Gamma_n] factorization}
		\mathbb{Q}[\Gamma_n]
		\simeq
		\prod_{r\mid n}
		\mathbb{Q}(\zeta_r).
	\end{equation}
    In particular, each $\mathbb{Q}(\zeta_r)$ for $r\mid n$ is a \emph{rational} representation of $\Gamma_n$ and we can make the representation explicit as follows
    \begin{equation}\label{Gamma_n and Q(zeta_r)}
        \Gamma_n\times\mathbb{Q}(\zeta_r)\to\mathbb{Q}(\zeta_r),
    \quad
    (\gamma,x)\mapsto\zeta_r x\,
    (\text{the multiplication of $\zeta_r$ and $x$ in $\mathbb{Q}(\zeta_r)$}).
    \end{equation}
    We write $\phi\colon\mathbb{N}\to\mathbb{N}$ for Euler's totient function, in particular, we have
	\[
        \phi(r)
	=
	\mathrm{dim}_{\mathbb{Q}}(\mathbb{Q}(\zeta_{r})).
	\]
    We have the following easy observation.
    \begin{lemma}
		Fix a positive integer $n>0$, then any irreducible rational representation of $\Gamma_n$ is of the form $\mathbb{Q}(\zeta_r)$ for some $r\mid n$ and these $\mathbb{Q}(\zeta_r)$ are mutually non-isomorphic.
	\end{lemma}
	\begin{proof}
            The proof is well known to experts, which we include here for the convenience of the reader.
		First we show that each $\mathbb{Q}(\zeta_r)$ is an irreducible $\mathbb{Q}$-linear representation of $\Gamma_n$ for $r|n$. Indeed, $\mathbb{Q}(\zeta_r)$ is a $\mathbb{Q}[\Gamma_n]$-module and the image of $\mathbb{Q}[\Gamma_n]$ in $\mathrm{End}_{\mathbb{Q}}(\mathbb{Q}(\zeta_r))$ is exactly $\mathbb{Q}(\zeta_r)$, which is a field. We deduce that $\mathbb{Q}(\zeta_r)$ is irreducible as $\mathbb{Q}[\Gamma_n]$-module.

		Next we show that for distinct positive integers $r>r'$, $\mathbb{Q}(\zeta_{r})$ and $\mathbb{Q}(\zeta_{r'})$ are not isomorphic to each other as representations of $\Gamma_n$. Indeed, the action of $\Gamma_n$ on $\mathbb{Q}(\zeta_r)$ factors through $\Gamma_r$ and does not factor through $\Gamma_{r/d}$ for any positive integer $1<d\mid r$. If $\mathbb{Q}(\zeta_r)$ and $\mathbb{Q}(\zeta_{r'})$ are isomorphic, then the action of $\Gamma_n$ on $\mathbb{Q}(\zeta_r)$ factors through $\Gamma_{\mathrm{gcd}(r,r')}$, this is a contradiction since $\mathrm{gcd}(r,r')<r$.

		Let $V$ be an irreducible $\mathbb{Q}$-linear representation of $\mathbb{Q}[\Gamma_n]$. Then the following induced map
        \[
        \mathbb{Q}[\Gamma_n]
        \to
        \mathrm{End}_{\mathbb{Q}}(V)
        \]
        is necessarily non-zero.
        So we can assume that there is a factor $r$ of $n$ such that the image of $\mathbb{Q}(\zeta_{r})$ (viewed as a subring of $\mathbb{Q}[\Gamma_n]$ via (\ref{Q[Gamma_n] factorization})) in $\mathrm{End}_{\mathbb{Q}}(V)$ is non-zero. Take any non-zero vector $v\in V$ and write $W=\mathbb{Q}(\zeta_{r})v$. Since $\Gamma_n$ is an abelian group, $W$ is a non-zero subrepresentation of $V$ of dimension equal to $ \phi(r)$. Since $V$ is irreducible, we must have $W=V$. Thus the image of $\mathbb{Q}[\Gamma_n]$ in $\mathrm{End}_{\mathbb{Q}}(V)$ is exactly $\mathbb{Q}(\zeta_{r})$ and so $V$ is isomorphic to $\mathbb{Q}(\zeta_{r})$.
	\end{proof}

        Clearly each $\mathbb{Q}(\zeta_r)$ is a rational representation of $\Gamma$ induced by the projection $\Gamma\to\Gamma_n$ (for \emph{any} positive integer $n$ divisible by $r$), and all these rational representations of $\Gamma$ are irreducible and mutually non-isomorphic. So in the following, we will view $\mathbb{Q}(\zeta_r)$ as an irreducible rational representation of any $\Gamma_n$ (with $r\mid n$) and $\Gamma$.

        From (\ref{Gamma_n and Q(zeta_r)}), it is easy to see that the ring of integers $\mathbb{Z}[\zeta_r]$ of $\mathbb{Q}(\zeta_r)$ is stable under the action of $\Gamma$. Thus we can and will view $\mathbb{Z}[\zeta_r]$ as an object in $\Mod_\Gamma$.
        Recall that the $\Gamma$-module structure on $\mathbb{Z}[\zeta_r]^\vee$ is given as follows: for any $f\in\mathbb{Z}[\zeta_r]^\vee$, $\gamma f$ is the element in $\mathbb{Z}[\zeta_r]^\vee$ sending $y\in\mathbb{Z}[\zeta_r]$ to $f(\gamma^{-1}y)=f(\zeta_r^{-1}y)$. We have the following natural map
        \begin{equation}\label{Z[zeta_r] embeds into its dual}
            \mathbb{Z}[\zeta_r]
        \to
        \mathbb{Z}[\zeta_r]^\vee,\quad
        x\mapsto
        (\phi_x\colon
        y\mapsto\mathrm{Tr}_{\mathbb{Q}(\zeta_r)/\mathbb{Q}}(x\overline{y})).
        \end{equation}
	Here $\overline{y}$ is the complex conjugation of $y$.  This map is in fact $\Gamma$-equivariant. Indeed, for $x,y\in\mathbb{Z}[\zeta_r]$, we have the following
	\begin{align*}
	\phi_{\gamma x}(y)
        &
        =
        \mathrm{Tr}_{\mathbb{Q}(\zeta_r)/\mathbb{Q}}
        ((\gamma x)\overline{y})
        =
        \mathrm{Tr}_{\mathbb{Q}(\zeta_r)/\mathbb{Q}}
        (\zeta_rx\overline{y})
        \\
        &
        =
        \mathrm{Tr}_{\mathbb{Q}(\zeta_r)/\mathbb{Q}}
        (x\overline{\zeta_r^{-1}y})
        =
        (\gamma\phi_x)(y).
	\end{align*}
	Thus $\phi_{\gamma x}=\gamma\phi_x$.

	\begin{proposition}
		Any free $\Gamma$-module $M$ in {\normalfont $\textrm{Mod}_{\Gamma}$} is isogenous to an object of the form $\prod_{r=1}^\infty\mathbb{Z}[\zeta_r]^{a(r)}$ where $(a(r))_{r>0}$ is a sequence of non-negative integers such that $a(r)=0$ for $r\gg0$.
	\end{proposition}
	\begin{proof}
		The action of $\Gamma$ on $M$ factors through $\Gamma_n$ for some $n$, so the vector space $M\otimes_{\mathbb{Z}}\mathbb{Q}$ is a $\mathbb{Q}$-linear representation of $\Gamma_n$. By the previous lemma, we have an isomorphism of $\mathbb{Q}$-linear representations of $\Gamma_n$
		\[
		F
		\colon
		M\otimes_{\mathbb{Z}}\mathbb{Q}
		\simeq
		\prod_{r=1}^\infty\mathbb{Q}(\zeta_{r})^{a(r)}
		\]
		for a sequence of integers $(a(r))_{r>0}$ such that $a(r)=0$ for $r\gg0$.

		If $a(r)>0$, we write $e_{r,1},\cdots,e_{r,a(r)}$ for the standard $\mathbb{Q}(\zeta_{r})$-basis of $\mathbb{Q}(\zeta_{r})^{a(r)}$, which we also view as a subspace of $\prod_{r=1}^\infty\mathbb{Q}(\zeta_{r})^{a(r)}$ in a natural way. Then we can find $m_{r,1},\cdots,m_{r,a(r)}\in M$ and $t_{r,1},\cdots,t_{r,a(r)}\in\mathbb{Q}^\times$ such that
		\[
		F(m_{r,i}\otimes t_{r,i})
		=
		e_{r,i},
		\quad
		\forall
		i=1,\cdots,a(r).
		\]
		Then the $\mathbb{Z}$-submodule  $M'=\prod_{r=1}^\infty\prod_{i=1}^{a(r)}\mathbb{Z}[\Gamma_n]m_{r,i}$ of $M$ is stable under $\Gamma_n$ and moreover $M'\otimes\mathbb{Q}=M\otimes\mathbb{Q}$. So by Lemma \ref{isogeny between free Z-modules}, $M$ is isogenous to $M'$, which is isomorphic to $\prod_{r=1}^\infty\mathbb{Z}[\zeta_{r}]^{a(r)}$.
		
	\end{proof}

	\begin{corollary}\label{canonical form of log free Gamma-module}
		Any lattice pairing $(M,N,\Phi)=(M,N,\langle-,-\rangle)$ in $\MorGammaN$ with $\langle-,-\rangle$ \emph{non-degenerate} is isogenous to a lattice pairing of the form {\normalfont $(R,R^\vee,\id_R)$}, where $R=\prod_{r=1}^\infty\mathbb{Z}[\zeta_{r}]^{a(r)}$ for a sequence of non-negative integers $(a(r))_{r>0}$ such that $a(r)=0$ for $r\gg0$.
	\end{corollary}
	\begin{proof}
		By definition, $\Phi$ is an isogeny in $\Mod_{\Gamma}$. Then it is easy to see that $(M,M^\vee,\id_M)$ is isogenous to $(M,N,\Phi)$ by the following commutative diagram
		\[
		\begin{tikzcd}
			M
			\arrow[r, "\id_M"]
			\arrow[d, "\id_M"']
			&
			(M^\vee)^\vee=M
			\arrow[d,"\Phi"]
			\\
			M
			\arrow[r,"\Phi"]
			&
			N^\vee
		\end{tikzcd}
		\]
		Suppose that $R$ is isogenous to $M$ as in the corollary by a morphism $F\colon R\rightarrow M$, then $(R,R^\vee,\mathrm{Id}_R)$ is isogenous to $(M,M^\vee,\id_M)$ by the following commutative diagram
		\[
		\begin{tikzcd}
			R
			\arrow[r,"\id_R"]
			\arrow[d,"F"']
			&
			(R^\vee)^\vee=R
			\arrow[d,"(F^\vee)^\vee=F"]
			\\
			M
			\arrow[r,"\mathrm{Id}_M"]
			&
			(M^\vee)^\vee=M
		\end{tikzcd}
		\]
	\end{proof}


	\begin{proposition}\label{non-deg iff polar iff ppolar}
            Let $(M,N,\Phi)=(M,N,\langle-,-\rangle)$ be an object of $\mathrm{Mor}_{\Gamma}^\N$. Then the following are equivalence.
		\begin{enumerate}
		    \item  $\langle-,-\rangle$ is non-degenerate. 
                \item $(M,N,\Phi)$ is pointwise polarizable, that is, it lies in $\MorGammaNpPol$.
                \item $(M,N,\Phi)$ is polarizable.
		\end{enumerate}
  
	\end{proposition}
	\begin{proof}
		We show $(1)\Rightarrow(3)$. By Corollary \ref{canonical form of log free Gamma-module} and Proposition \ref{simpleness of LP is preserved by isogenies} (2), we can assume
		\[
		(M,N,\Phi)=(\mathbb{Z}[\zeta_{r}],\mathbb{Z}[\zeta_{r}]^\vee,\mathrm{Id}).
		\]
        The map $\lambda\colon\mathbb{Z}[\zeta_r]\hookrightarrow\mathbb{Z}[\zeta_r]^\vee$ from  \eqref{Z[zeta_r] embeds into its dual} defines a polarization of $(\Z[\zeta_r],\Z[\zeta_r]^\vee,\id)$. Indeed, by the correspondence between the pairing $\langle-,-\rangle$ and the map $\mathrm{Id}\colon\mathbb{Z}[\zeta_r]\to\mathbb{Z}[\zeta_r]$, we have
        \[
        \langle x,\lambda(y)\rangle
        =
        \mathrm{Tr}_{\mathbb{Q}(\zeta_r)/\mathbb{Q}}(y\overline{x})
        =
        \mathrm{Tr}_{\mathbb{Q}(\zeta_r)/\mathbb{Q}}(x\overline{y})
        =
        \langle y,\lambda(x)\rangle,
        \quad
        \forall
        x,y\in\mathbb{Z}[\zeta_r]
        \]
        and
        \[
        \langle x,\lambda(x)\rangle
        =
        \mathrm{Tr}_{\mathbb{Q}(\zeta_r)/\mathbb{Q}}(x\overline{x})>0,
        \quad
        \forall\,0\neq x\in\mathbb{Z}[\zeta_r].
        \]        
        From this we deduce that $\lambda$ is a polarization on $(M,N,\Phi)$.

        $(3)\Rightarrow(2)$ is trivial. 
        
        We show $(2)\Rightarrow(1)$. Since $(M,N,\Phi)\in\MorGammaNpPol$, there exists an open subgroup $\Gamma'\subset \Gamma$ such that $(M,N,\Phi)$ is polarizable in $\Mor_{\Gamma'}^{\N}$ with a polarization $\lambda:M\to N$. In particular, $\langle-,\lambda(-)\rangle:M\times M\to\Z$ is symmetric and $\langle-,\lambda(-)\rangle_{\R}:(M\otimes_{\Z}\R)\times (M\otimes_{\Z}\R)\to\R$ is positive definite. Then it is easy to see that the bilinear form $\langle-,-\rangle:M\times N\to\Z$ is non-degenerate.
    
	\end{proof}

        The implication $(2)\Rightarrow (1)$ in Proposition \ref{non-deg iff polar iff ppolar} holds for $\kbf$ being a general field. But the implication $(1)\Rightarrow(3)$ relies on the condition that $\kbf$ is a finite field, as $\Gamma\cong\widehat{\Z}$ is key to the morphism \eqref{Z[zeta_r] embeds into its dual} being an isogeny in $\Mod_{\Gamma}$.

    \begin{corollary}\label{log 1-motive polarizable-rank 1}
        Any pointwise polarizable log 1-motive over the standard log point $S$ is polarizable.
    \end{corollary}
    \begin{proof}
    By the proof of Theorem \ref{isogeneous decomposition of log AV} (1), $\Mbf$ is isogenous to the product of $[Y\xrightarrow{u_1}T_{\log}]$ and $B$ with $[Y\xrightarrow{u_1}T_{\log}]$ pointwise polarizable. By Corollary \ref{equivalence from pairs of Galois modules to LAVs}, we may assume that $[Y\xrightarrow{u_1}T_{\log}]$ lies in the image of the functor $\Mbf_{-}$. Then by Proposition \ref{non-deg iff polar iff ppolar} and Proposition \ref{properties of the functor M_-} (3), $[Y\xrightarrow{u_1}T_{\log}]$ is actually polarizable. It follows that $\Mbf$ is polariable by \cite[Lemma 3.4(1)]{Zhao2017}.
    \end{proof}

        We now give the description of objects in $\MorGammaNpPol$.

    \begin{proposition}\label{description of pPol object up to isogeny}
    \begin{enumerate}
        \item For any $r>0$, $(\Z[\zeta_r],\Z[\zeta_r]^\vee,\id)$  admits a polarization given by the natural map $\mathbb{Z}[\zeta_r]\hookrightarrow\mathbb{Z}[\zeta_r]^\vee$ (see (\ref{Z[zeta_r] embeds into its dual})), and is a simple object of $\MorGammaNpPol$.
        \item Any object $(M,N,\Phi)$ in $\MorGammaNpPol$ is isogenous to an object of the form $(R,R^\vee,\id_R)$ where $R=\prod_{r=1}^\infty\mathbb{Z}[\zeta_r]^{a(r)}$
        for a sequence of non-negative integers $(a(r))_{r>0}$ such that $a(r)=0$ for $r\gg0$.
        \item Any simple object in $\MorGammaNpPol$ is isogenous to $(\Z[\zeta_{r}],\Z[\zeta_{r}],\mathrm{Id})$ for some positive integer $r>0$.
    \end{enumerate}
    \end{proposition}
    \begin{proof}
            (1) The first half of the statement has already appeared in the proof of the implication $(1)\Rightarrow(3)$ in Proposition \ref{non-deg iff polar iff ppolar}. Since $\Z[\zeta_{r}]$ is a simple free $\Gamma$-module, $(\Z[\zeta_r],\Z[\zeta_r]^\vee,\id)$ is a simple object of $\MorGammaNpPol$ by Lemma \ref{pPol LP with simple factor is simple and the converse is true for rank 1 monoid}.
           
            (2) Since $(M,N,\Phi)$ is an object of $\MorGammaNpPol$, the bilinear map $M\times N\to\Z$ is non-degenerate by Proposition \ref{non-deg iff polar iff ppolar}. Then by Corollary \ref{canonical form of log free Gamma-module} it is isogenous to an object of the form $(R,R^\vee,\id_R)$ which also lies in $\MorGammaNpPol$ by Proposition \ref{simpleness of LP is preserved by isogenies} (2).

           (3) This follows from (1) and (2).
	\end{proof}

    \subsection{Polarizations in $\MorGammaNpPol$}
    In this subsection, we describe the polarizations on certain objects in $\MorGammaNpPol$.

	Fix a positive integer $r>0$. We write $\mathfrak{d}_{\mathbb{Z}[\zeta_{r}]}$ for the different ideal of $\mathbb{Z}[\zeta_{r}]$ whose inverse is given by
	\[
	\mathfrak{d}_{\mathbb{Z}[\zeta_{r}]}^{-1}
	=
	\left\{
	t\in\mathbb{Q}(\zeta_{r})
	|
	\mathrm{Tr}_{\mathbb{Q}(\zeta_{r})/\mathbb{Q}}(ty)\in\mathbb{Z},\,\forall y\in\mathbb{Z}[\zeta_{r}]
	\right\}.
	\]
	One checks that $\mathfrak{d}_{\Z[\zeta_{r}]}^{-1}$ is stable under the action of $\Gamma$.
	\begin{lemma}
	    The natural inclusion $\Z[\zeta_{r}]\hookrightarrow \Z[\zeta_{r}]^\vee$ (see \eqref{Z[zeta_r] embeds into its dual}) extends to an isomorphism of free $\Gamma$-modules
	\begin{equation}\label{identification of different and dual}
		\mathfrak{d}_{\Z[\zeta_{r}]}^{-1}
		\xrightarrow{\simeq}
		\Z[\zeta_{r}]^\vee,
		\quad
		x
		\mapsto
		(\phi_x\colon y
		\mapsto
		\mathrm{Tr}_{\Q(\zeta_{r})/\Q}(x\overline{y})).
	\end{equation}
        Moreover, this isomorphism extends to an isomorphism $\Q(\zeta_{r})\simeq\Q(\zeta_{r})^\vee=\Q(\zeta_{r})$ of $\mathbb{Q}$-linear representations of $\Gamma$.
	\end{lemma}
	\begin{proof}
            Clearly $\eqref{identification of different and dual}$ extends \eqref{Z[zeta_r] embeds into its dual}. From the $\Gamma$-equivariance of (\ref{Z[zeta_r] embeds into its dual}), we deduce that (\ref{identification of different and dual}) is also $\Gamma$-equivariant. Moreover it is clear that the map (\ref{identification of different and dual}) is an isomorphism of $\mathbb{Z}$-modules.
	\end{proof}

	For $t\in\mathbb{Q}(\zeta_{r})$, we have the following $\mathbb{Q}$-bilinear map
	\[
	B_t(-,-)
	\colon
	\mathbb{Q}(\zeta_{r})
	\times
	\mathbb{Q}(\zeta_{r})
	\rightarrow
	\Q,
	\quad
	(x,y)
	\mapsto
	\mathrm{Tr}_{\Q(\zeta_{r})/\Q}(xt\overline{y}).
	\]
        It is symmetric if and only if $t\in\mathbb{Q}(\zeta_r)^+$, the maximal totally real subfield of $\mathbb{Q}(\zeta_r)$ (or equivalently, the subfield fixed by complex conjugation):
        \begin{align}\label{B_t symmetric iff t totally real}
            \begin{split}
                  &
          \mathrm{Tr}_{\mathbb{Q}(\zeta_r)/\mathbb{Q}}(xt\overline{y})
        =
        \mathrm{Tr}_{\mathbb{Q}(\zeta_r)/\mathbb{Q}}(yt\overline{x}),
        \quad
        \forall x,y\in\mathbb{Q}(\zeta_r)
        \\
        \iff
        &
        \mathrm{Tr}_{\mathbb{Q}(\zeta_r)/\mathbb{Q}}(xt\overline{y})
        =
        \mathrm{Tr}_{\mathbb{Q}(\zeta_r)/\mathbb{Q}}(x\overline{ty}),
        \quad
        \forall x,y\in\mathbb{Q}(\zeta_r)
        \\
        \iff
        &
        \mathrm{Tr}_{\mathbb{Q}(\zeta_r)/\mathbb{Q}}(x\overline{y}(t-\overline{t}))=0,
        \quad
        \forall
        x,y\in\mathbb{Q}(\zeta_r)
        \\
        \iff
        &
        t-\overline{t}=0
        \iff
        t\in\mathbb{Q}(\zeta_r)^+.
            \end{split}
        \end{align}
        
        Moreover this bilinear map is $\Gamma$-equivariant. Indeed, by bilinearity, it suffices to show
	$B_t(\gamma\zeta_{r}^j,\gamma\zeta_{r}^{j'})=B_t(\zeta_{r}^j,\zeta_{r}^{j'})$ for any $j,j'$. However we know $\gamma\zeta_{r}^j=\zeta_{r}^{j+1}$ (similarly for $j'$). So we have
	\[
	B_t(\gamma\zeta_{r}^j,\gamma\zeta_{r}^{j'})
	=
	\mathrm{Tr}(\zeta_{r}^{j+1}t\overline{\zeta_{r}^{j'+1}})
	=
	\mathrm{Tr}(\zeta_{r}^jt\overline{\zeta_{r}^{j'}})
	=
	B_t(\zeta_{r}^j,\zeta_{r}^{j'}).
	\]
	Using the isomorphism (\ref{identification of different and dual}), one checks that $B_{t}(-,-)$ corresponds to a $\Gamma$-equivariant linear map
        \[
        \lambda_t
	\colon
	\mathbb{Q}(\zeta_{r})
	\rightarrow
	\mathbb{Q}(\zeta_{r})^\vee
	\simeq
	\mathbb{Q}(\zeta_{r}),
	\quad
	z
	\mapsto
	tz.
        \]

	If $t\in\mathfrak{d}_{\mathbb{Z}[\zeta_{r}]}^{-1}$, then $B_t(-,-)$ restricts to a bilinear form
	\[
	B_t(-,-)
	\colon
	\mathbb{Z}[\zeta_{r}]\times\mathbb{Z}[\zeta_{r}]
	\rightarrow
	\mathbb{Z},
	\quad
	(x,y)
	\mapsto
	\mathrm{Tr}_{\mathbb{Q}(\zeta_{r})/\mathbb{Q}}(xt\overline{y}).
	\]
	As above, this gives rise to an isogeny of free $\Gamma$-modules
        \begin{equation}\label{lambda_t}
        \lambda_t
	\colon
	\mathbb{Z}[\zeta_{r}]
	\rightarrow\mathbb{Z}[\zeta_{r}]^\vee\simeq\mathfrak{d}_{\Z[\zeta_{r}]}^{-1},
	\quad
	z\mapsto
	tz.
        \end{equation}

	We define the following subset of $\mathfrak{d}_{\mathbb{Z}[\zeta_r]}^{-1}$:
	\[
	(\mathfrak{d}_{\mathbb{Z}[\zeta_{r}]}^{-1})_{>0}
	=
	\left\{
	t\in\mathfrak{d}_{\mathbb{Z}[\zeta_{r}]}^{-1}\bigcap
        \mathbb{Q}(\zeta_r)^+
	\mid t \text{ is totally positive}
	\right\}.
	\]

 	The following result gives all the possible polarizations on a simple object in $\MorGammaN$.
	\begin{proposition}\label{polarization on (Z[zeta_r],Z[zeta_r]^vee,Id)}
		Any polarization on {\normalfont $(\mathbb{Z}[\zeta_{r}],\mathbb{Z}[\zeta_{r}]^\vee,\mathrm{Id})$} is of the form $\lambda_t$ for some $t\in(\mathfrak{d}_{\mathbb{Z}[\zeta_{r}]}^{-1})_{>0}$.	
	\end{proposition}
	\begin{proof}
		Consider the case $(\mathbb{Z}[\zeta_r],\mathbb{Z}[\zeta_r]^\vee,\mathrm{Id})=(\mathbb{Z}[\zeta_r],\mathbb{Z}[\zeta_r]^\vee,\langle-,-\rangle)$.
		Let $\lambda\colon\mathbb{Z}[\zeta_{r}]\rightarrow\mathbb{Z}[\zeta_{r}]^\vee$ be a polarization, which is $\Gamma$-equivariant. So $\lambda$ is necessarily of the form $\lambda=\lambda_t$ for some $0\neq t\in\mathfrak{d}_{\mathbb{Z}[\zeta_r]}^{-1}$. Indeed, since $\lambda(\zeta_r^j)=\lambda(\gamma(\zeta_r^{j-1}))=\gamma(\lambda(\zeta_r^{j-1}))$ for any $j$, $\lambda$ is determined by $\lambda(1)\in\mathbb{Z}[\zeta_r]^\vee\simeq\mathfrak{d}_{\mathbb{Z}[\zeta_r]}^{-1}$, so $\lambda=\lambda_t$ with $t=\lambda(1)$. Moreover, as in (\ref{B_t symmetric iff t totally real}), we know that $t\in\mathbb{Q}(\zeta_r)^+$. To conclude the proof, we need to show that $t$ is totally positive.

        The positive-definiteness condition in the polarization $\lambda$ becomes
		\[
		\langle x,\lambda(x)\rangle=\mathrm{Tr}_{\mathbb{Q}(\zeta_{r})/\mathbb{Q}}(xt\overline{x})>0,
		\quad
		\forall
		x\in\mathbb{Z}[\zeta_{r}]\backslash\{0\}.
		\]
		We extend $\langle-,-\rangle$ to a bilinear form on $\Q(\zeta_{r})\times\Q(\zeta_{r})$. We know that $\mathbb{Q}(\zeta_r)$ is dense in
        \[
        \mathbb{Q}(\zeta_r)\otimes_{\mathbb{Q}}\mathbb{R}
        \simeq\prod_{\mathrm{Hom}_{\mathbb{Q}}(\mathbb{Q}(\zeta_r),\mathbb{C})/\sim}\mathbb{C}
        \]
        Here $/\sim$ means up to complex conjugation.

        Fix an arbitrary embedding $\sigma_0\colon\mathbb{Q}(\zeta_r)\to\mathbb{C}$, below we prove $\sigma_0(t)>0$ by contradiction.
        \begin{enumerate}
            \item Suppose $\sigma_0(t)=0$, then $t=0$ and thus $\langle x,\lambda(x)\rangle=0$ for any $0\neq x\in\mathbb{Z}[\zeta_r]$. We obtain a contradiction. 
            \item Suppose $\sigma_0(t)<0$. Let $x\in\Q(\zeta_r)^\times$ be such that the absolute values of $\sigma_0(x)-1$ and $\sigma(x)$ (for all $\sigma\neq\sigma_0,\overline{\sigma_0}$) are all sufficiently small (written as $\sigma_0(x)-1,\sigma(x)\sim 0$). Thus
            \begin{align*}
            \langle x,\lambda(x)\rangle
            &
            =
            \sigma_0(xt\overline{x})+\overline{\sigma_0}(xt\overline{x})
            +            \sum_{\sigma\neq\sigma_0,\overline{\sigma_0}}\sigma(xt\overline{x})
            \\
            &
            =
            2\sigma_0(x)\sigma_0(\overline{x})\sigma_0(t)
            +            \sum_{\sigma\neq\sigma_0,\overline{\sigma_0}}\sigma(xt\overline{x})
            \sim 2\sigma_0(t)<0.
            \end{align*}
            This is again a contradiction.
        \end{enumerate}
        Thus we conclude that $\sigma(t)>0$ for all field embeddings $\sigma$ and so $t\in(\mathfrak{d}_{\mathbb{Z}[\zeta_{r}]}^{-1})_{>0}$.
        

	\end{proof}

	For integers $r>0$, $n\ge0$ and a fractional ideal $J$ of $\Z[\zeta_{r}]$ stable under the complex conjugation, we write $\mathrm{Herm}_n(J)$ for the set of Hermitian $n\times n$-matrices with entries in $J$. We then write
    \[
    \mathrm{Herm}_n(\mathfrak{d}_{\Z[\zeta_{r}]}^{-1})_{>0}
    :=
    \left\{X\in\mathrm{Herm}_n(\mathfrak{d}_{\Z[\zeta_{r}]}^{-1})
    \mid
    \sigma(X)\text{ positive definite, }
    \forall
    \sigma\in\mathrm{Hom}_{\mathbb{Q}}(\mathbb{Q}(\zeta_r),\mathbb{C})
    \right\}
    \] 
	\begin{corollary}\label{polarization in (M,M^vee,Id)}
		Any polarization $\lambda$ on {\normalfont  $(M,M^\vee,\mathrm{Id}_M)$} with $M=\prod_{r=1}^\infty\mathbb{Z}[\zeta_{r}]^{a(r)}$ ($a(r)=0$ for $r\gg0$) is of the form $\lambda=\prod_{r=1}^{\infty}\lambda_{t(r)}$ for some $t(r)\in\mathrm{Herm}_{a(r)}(\mathfrak{d}_{\mathbb{Z}[\zeta_{r}]}^{-1})_{>0}$. Here $\lambda_{t(r)}$ is given by 
		\[
		\lambda_{t(r)}
		\colon
		\mathbb{Z}[\zeta_{r}]^{a(r)}
		\to(\mathbb{Z}[\zeta_r]^{a(r)})^\vee,
            \quad
            x
		\mapsto
            (y\mapsto
		\mathrm{Tr}_{\mathbb{Q}(\zeta_{r})/\mathbb{Q}}
		(x^{\mathrm{t}}t(r)\overline{y})).
		\]
        (we view elements $x$ in $\mathbb{Z}[\zeta_{r}]^{a(r)}$ as column matrices).
	\end{corollary}
        \begin{proof}
            Clearly any $\mathbb{Z}[\zeta_r]$-linear (or equivalently, $\Gamma$-equivariant) map $\mathbb{Z}[\zeta_r]^{a(r)}\to(\mathbb{Z}[\zeta_r]^{a(r)})^\vee$ is represented by a matrix $t(r)\in\mathrm{Mat}_{a(r)}(\mathfrak{d}_{\mathbb{Z}[\zeta_{r}]}^{-1})$, as given in the Corollary. Suppose $\mathrm{Id}_{\mathbb{Z}[\zeta_r]^{a(r)}}$ (induced from $\mathrm{Id}_M$) corresponds to the pairing $\langle-,-\rangle\colon\mathbb{Z}[\zeta_r]^{a(r)}\times(\mathbb{Z}[\zeta_r]^{a(r)})^\vee\to\mathbb{Z}$,  then we have
            \[
            \langle x,\lambda_{t(r)}(y)\rangle
            =
            \langle y,\lambda_{t(r)}(x)\rangle,
            \quad
            \forall
            x,y\in\mathbb{Z}[\zeta_r]^{a(r)}.
            \]
            In other words, we have
            $\mathrm{Tr}_{\mathbb{Q}(\zeta_{r})/\mathbb{Q}}
		(x^{\mathrm{t}}t(r)\overline{y})
            =
            \mathrm{Tr}_{\mathbb{Q}(\zeta_{r})/\mathbb{Q}}
		(y^{\mathrm{t}}t(r)\overline{x})$, that is,
            \[
            \mathrm{Tr}_{\mathbb{Q}(\zeta_{r})/\mathbb{Q}}
		(x^{\mathrm{t}}(\overline{t(r)}^{\mathrm{t}}- 
            t(r))\overline{y})
            =0,
            \quad
            \forall x,y\in\mathbb{Q}(\zeta_r)^{a(r)}.
            \]
            Write the standard basis for $\mathbb{Q}(\zeta_r)^{a(r)}$ as $E_1,\cdots,E_{a(r)}$, then taking $x=sE_i$ and $y=E_j$ with $s\in\mathbb{Q}(\zeta_r)$, we get
            \[
            0
            =
            \mathrm{Tr}_{\mathbb{Q}(\zeta_{r})/\mathbb{Q}}
		(x^{\mathrm{t}}(\overline{t(r)}^{\mathrm{t}}- 
            t(r))\overline{y})
            =
            \mathrm{Tr}_{\mathbb{Q}(\zeta_{r})/\mathbb{Q}}
		(s(\overline{t(r)}^{\mathrm{t}}- 
            t(r))_{i,j})
            \]
            (here the subscript $_{i,j}$ means the $(i,j)$-th entry of the matrix). Thus we have $(\overline{t(r)}^{\mathrm{t}}- 
            t(r))_{i,j}=0$, so $\overline{t(r)}^{\mathrm{t}}=
            t(r)$.

            Next we show that for any field embedding $\sigma\colon\mathbb{Q}(\zeta_r)\to\mathbb{C}$, $\sigma(t(r))$ is positive definite. Since $t(a)$ is a Hermitian matrix, up to choosing a basis for $\mathbb{Q}(\zeta_r)^{a(r)}$, we can assume that $t(r)$ is a diagonal matrix
            \[
            t(r)
            =
            \mathrm{diag}(t_1,\cdots,t_{a(r)}).
            \]
            So we have $t_1,\cdots,t_{a(r)}\in\mathbb{Q}(\zeta_r)^+$. Now as in the proof of the preceding proposition, we can show each $t_i$ is totally positive. Thus we conclude that $\sigma(t(r))$ is positive definite for all field embeddings $\sigma$.
        \end{proof}

	\begin{proposition}
		Fix an integer $r>0$, then
		{\normalfont $\mathrm{End}_{\mathrm{Mod}_{\Gamma}}(\mathbb{Z}[\zeta_{r}])$} is equal to $\mathbb{Z}[\zeta_{r}]$.
	\end{proposition}
	\begin{proof}
		Clearly $\mathbb{Z}[\zeta_{r}]$ is contained in
		$\mathrm{End}_{\mathrm{Mod}_{\Gamma}}(\mathbb{Z}[\zeta_{r}])$. For any $F\in\mathrm{End}_{\mathrm{Mod}_{\Gamma}}(\mathbb{Z}[\zeta_{r}])$, we have
		\[
		F(\zeta_{r}x)
		=
		F(\gamma x)
		=
		\gamma(F(x))
		=
		\zeta_{r}F(x)
		\quad
		\forall
		x\in\mathbb{Z}[\zeta_{r}].
		\]
		As a result, for any $y\in\mathbb{Z}[\zeta_{r}]$, $F(yx)=yF(x)$. Thus $F(x)=xF(1)$ and $F$ is determined by $F(1)\in\mathbb{Z}[\zeta_{r}]$. This proves the proposition.
	\end{proof}

	\begin{corollary}
		Fix an integer $r>0$, then {\normalfont $(\mathbb{Z}[\zeta_{r}],\mathbb{Z}[\zeta_{r}]^\vee,\mathrm{Id})$} has endomorphism ring in $\MorGammaNpPol$ equal to $\mathbb{Z}[\zeta_{r}]$.
	\end{corollary}
	\begin{proof}
            By definition, we have
		\[
		\mathrm{End}_{\MorGammaNpPol}
		(\mathbb{Z}[\zeta_{r}],\mathbb{Z}[\zeta_{r}]^\vee,\mathrm{Id})
		=
		\mathrm{End}_{\mathrm{Mod}_{\Gamma}}(\mathbb{Z}[\zeta_{r}])
		=
		\mathbb{Z}[\zeta_{r}].
		\]
	\end{proof}

	\begin{corollary}
		Fix a free $\Gamma$-module
		$M=\prod_{r=1}^\infty\mathbb{Z}[\zeta_{r}]^{a(r)}$ ($a(r)=0$ for $r\gg0$), then
		\[
		\mathrm{End}_{\MorGammaNpPol}
		(M,M^\vee,\mathrm{Id})
		=
		\prod_{r=1}^{\infty}\mathrm{Mat}_{a(r)}(\mathbb{Z}[\zeta_{r}]).
		\]
	\end{corollary}

        \subsection{Isogeny classes in $\MorGammaNpPol$}
        First we have the following simple lemma.
	\begin{lemma}\label{charpoly of Z[zeta_r]}
		Fix a positive integer $r>0$. The characteristic polynomial $F_r(\theta)$ of the Frobenius $\gamma\in\Gamma$ acting on the simple object $\mathbb{Z}[\zeta_r]$ in $\Mod_{\Gamma}$ is the $r$-th cyclotomic polynomial (i.e. the minimal polynomial over $\Q$ of the primitive $r$-th roots of unity). In particular, the roots of $F_r(\theta)$ are primitive $r$-th roots of unity.
	\end{lemma}
	\begin{proof}
		We have seen that $\gamma$ acts on $\mathbb{Z}[\zeta_{r}]$ by sending $\zeta_{r}^j$ to $\zeta_{r}^{j+1}$, thus $\gamma^r$ acts as identity on $\mathbb{Z}[\zeta_{r}]$. Moreover $F_r(\theta)$ is irreducible over $\mathbb{Q}$ since $\mathbb{Q}(\zeta_{r})$ is an irreducible \emph{rational} representation of $\Gamma$. So the characteristic polynomial $F_r(\theta)\in\mathbb{Q}[\theta]$ of $\gamma$ divides $\theta^r-1$. Hence $F_r(\theta)$ is the $r$-th cyclotomic polynomial.
	\end{proof}
        In particular, $F_r(\theta)$ is a monic polynomial with coefficients in $\Z$ whose constant term is $\pm1$.
        
    In the following, we relate categories of pointwise polarizable log 1-motives/lattice pairings, log abelian varieties to prove Theorem \ref{main theorem-1}. We denote by $\mathrm{AV}$ the category of abelian varieties over $\kbf$. Let $\cC$ be either $\Mcal_1^{\log,\mathrm{ab}=0,\mathrm{pPol}}$, or $\mathrm{LAV}$, or $\mathrm{LAV}^{\mathrm{ab}=0}$, or $\mathrm{AV}$, or $\MorGammaNpPol$, and let $\mathcal{S}(\cC)$ be the set of isogeny classes of simple objects in $\cC$. For any object $A$ of $\cC$, we denote by $\mathrm{cl}(A)$ the isogeny class represented by $A$. By Theorem \ref{isogeneous decomposition of log AV}, we have
    \begin{equation}\label{descomposition of LAV up to isogeny}
        \mathcal{S}(\mathrm{LAV})=\mathcal{S}(\mathrm{LAV}^{\mathrm{ab}=0})\bigsqcup \mathcal{S}(\mathrm{AV}).
    \end{equation}
    By Proposition \ref{LAV is simple iff its log 1-motive is simple}, Proposition \ref{Hom of LAVs is isogeny iff the corresponding hom of log 1-motives is isogeny}, Proposition \ref{simpleness of ppol log 1-motives is preserved by isogeny}, and Proposition \ref{simpleness of LAV is preserved by isogeny}, the equivalence from Theorem \ref{equivalence from log 1-motives to LAVwCD} induces a bijection
    \[\mathcal{S}(\mathrm{LAV}^{\mathrm{ab}=0})\xrightarrow{\simeq} \mathcal{S}(\Mcal_1^{\log,\mathrm{ab}=0,\mathrm{pPol}}).\]
    By Proposition \ref{log 1-motive with ab=0 is simple iff its LP is simple}, Proposition \ref{properties of the functor LP} (3),  Proposition \ref{simpleness of ppol log 1-motives is preserved by isogeny}, and Proposition \ref{simpleness of LP is preserved by isogenies} (1), the functor $\LP$ from Proposition \ref{properties of the functor LP} (4) induces a bijection
    \[\mathcal{S}(\Mcal_1^{\log,\mathrm{ab}=0,\mathrm{pPol}})\xrightarrow{\simeq} \mathcal{S}(\MorGammaNpPol),\quad
    \mathrm{cl}(\Mbf)\mapsto \mathrm{cl}(\LP(\Mbf)),\]
    whose inverse is given by 
    \[\mathrm{cl}((M,N,\Phi))\mapsto \mathrm{cl}(\Mbf_{-}((M,N,\Phi)))\]
    by Corollary \ref{equivalence from pairs of Galois modules to LAVs}.
    By Proposition \ref{description of pPol object up to isogeny} (3), we have 
    \[\mathcal{S}(\MorGammaNpPol)=\{\mathrm{cl}((\Z[\zeta_r],\Z[\zeta_r]^\vee,\id))\mid r\in\N_{>0}\}\xrightarrow{\simeq} \N_{>0}.\]
    Let $\Mbf_{[r]}:=\Mbf_-((\Z[\zeta_r],\Z[\zeta_r]^\vee,\id))$ and let $A_{[r]}$ denote the log abelian variety corresponding to $\Mbf_{[r]}$.
    It follows that we have 
    \begin{align}\label{describption of simple LAVs and simple ppol log 1-mot}
        \begin{split}
            \mathcal{S}(\mathrm{LAV}^{\mathrm{ab}=0})&
            =\{\mathrm{cl}(A_{[r]})\mid r\in\N_{>0}\},
            \\
        \mathcal{S}(\Mcal_1^{\log,\mathrm{ab}=0,\mathrm{pPol}})&
        =\{\mathrm{cl}(\Mbf_{[r]})\mid r\in\N_{>0}\}.
        \end{split}
    \end{align}
    By Remark \ref{char poly of log 1-motive is invariant under isogeny} (resp. Remark \ref{char poly of LAV is invariant under isogeny}), we have  $P_{\Mbf_{[r]},\gamma}=P_{\Mbf,\gamma}$ (resp. $P_{A_{[r]},\gamma}=P_{A,\gamma}$) for any $\Mbf\in\mathrm{cl}(\Mbf_{[r]})$ (resp. $A\in \mathrm{cl}(A_{[r]})$), and $P_{\Mbf_{[r]},\gamma}=P_{A_{[r]},\gamma}$ by Definition \ref{char poly of log av}. Write $Y_{[r]}:=\Z[\zeta_r]$, $T_{[r]}:=\cHom(\Z[\zeta_r]^\vee,\Gm)$, and $\Mbf_{[r]}=[Y_{[r]}\to T_{[r]}]$. Recall that $Y_{[r]}=\Z[\zeta_r]$ is isogenous to the character group $\Z[\zeta_r]^\vee$ of $T_{[r]}$ by \eqref{lambda_t}, and thus
    \[P_{T_{[r]},\pi_{T_{[r]}}}(\theta)=\frac{(-\theta)^{\phi(r)}}{\det(\pi_{Y_{[r]}})}P_{Y_{[r]},\pi_{Y_{[r]}}}(\frac{q}{\theta})=\frac{(-\theta)^{\phi(r)}}{\det(\pi_{Y_{[r]}})}F_r(\frac{q}{\theta})\]
    by Lemma \ref{charpoly of lattices, AVs and tori} (2) and Lemma \ref{charpoly of Z[zeta_r]}. We denote $\frac{(-\theta)^{\phi(r)}}{\det(\pi_{Y_{[r]}})}F_r(\frac{q}{\theta})$ by $G_r(\theta)$. Then the polynomial $F_r(\theta)$ (resp. $G_r(\theta)$) is irreducible with roots the Galois conjugates of $\zeta_r$ (resp. $q\zeta_r^{-1}$) by Lemma \ref{charpoly of Z[zeta_r]}. Recall that $\qW(0,2)$ is the subset of $\qW(0)\times\qW(2)$ consisting of elements of the form $([\alpha],[q\alpha^{-1}])$ for $[\alpha]\in\qW(0)$ (see (\ref{q-Weil(0,2)})). Therefore we have a bijection
    \begin{equation}\label{description of q-Weil(0,2)}
        \begin{split}
            \{F_r(\theta)G_r(\theta)\mid r\in\N_{>0}\}&\to \qW(0,2) \\ F_r(\theta)G_r(\theta)&\mapsto ([\zeta_r],[q\zeta_r^{-1}]).
        \end{split}
    \end{equation}
    We also have
    \begin{equation}\label{charpol of A_[r]}
        P_{A_{[r]},\gamma}(\theta)=P_{\Mbf_{[r]},\gamma}(\theta)=P_{Y_{[r]},\pi_{Y_{[r]}}}(\theta)\cdot P_{T_{[r]},\pi_{T_{[r]}}}(\theta)=F_r(\theta)G_r(\theta).
    \end{equation}

 Combining \eqref{describption of simple LAVs and simple ppol log 1-mot}, \eqref{description of q-Weil(0,2)} and \eqref{charpol of A_[r]}, we have the following.
 
	\begin{theorem}\label{classificaton of isogenous classes of log Gamma-module of rank 1}
        The canonical map
        \begin{align}\label{simple LAV without abelian bijection with qW(0,2)}
            \mathcal{S}(\mathrm{LAV}^{\mathrm{ab}=0})\to \qW(0,2),\quad A_{[r]}\mapsto ([\zeta_r],[q\zeta_r^{-1}])
        \end{align}
        is a bijection. Note that $\zeta_r$ is a root of the irreducible factor $F_r(\theta)$ of $P_{A_{[r]},\gamma}(\theta)$ and $q\zeta_r^{-1}$ is a root of the other irreducible factor $G_r$.
         \end{theorem}

	Combined with the discussion in the end of §\ref{subsection LAV without abelian part over finite log points are classified by LP up to isogeny} (or \eqref{descomposition of LAV up to isogeny}), Theorem \ref{classificaton of isogenous classes of log Gamma-module of rank 1} proves Theorem \ref{main theorem-1}.


	\section{Isogeny classes in $\MorGammaPpPol$ over a general finite log point}\label{Isogeny classes in MorGammaPpPol over a general finite log point}
    In this sectoin, let $\kbf=\F_q$, $\Gamma=\overline{\langle\gamma\rangle}$, $P$, and $S$ be as in Subsection \ref{subsection LAV without abelian part over finite log points are classified by LP up to isogeny}. We further assume that $P=\N^k$.
    We will use results in the preceding section to classify a certain class of log abelian varieties over $S$ up to isogeny.

	We write $\pi_i\colon P=\mathbb{N}^k\rightarrow\mathbb{N}$ for the projection onto the $i$-th component ($i=1,\cdots,k$). Then for $(M,N,\langle-,-\rangle)$ in $\MorGammaPpPol$,
	we write the $i$-th component of $\langle-,-\rangle$ as
	\[
	\langle-,-\rangle_i\colon M\times N\rightarrow\mathbb{Z}^k\xrightarrow{\pi_i}\mathbb{Z}.
	\]
	We write $\Phi_i$ for the corresponding map $M\rightarrow\mathrm{Hom}(N,\mathbb{Z})$ induced by $\langle-,-\rangle_i$ (and similarly for $\Phi_i^\vee$). We will see in the following that the pair $(\langle-,-\rangle,\lambda)$ is in fact determined by the tuples $(\langle-,-\rangle_i,\lambda)_{i=1,\cdots,k}$.

    \subsection{Isogeny classes in $\MorGammaPpPol$}
    In this subsection, we give a preliminary description of isogeny classes in $\MorGammaPpPol$.
	\begin{proposition}\label{polarization of higher rank vs rank-1}
		Let $(M,N,\langle-,-\rangle)$ be an object in $\MorGammaP$ with a polarization $\lambda\colon M\to N$. Then for each $i=1,\cdots,k$, we have a commutative diagram
		  \begin{equation}\label{factors of polar-diagram}
                \begin{tikzcd}
			M
			\arrow[r,"\Phi_i"]
			\arrow[d,"\lambda"]
			&
			N^\vee
			\arrow[d,"\lambda^\vee"]
			\\
			N
			\arrow[r,"\Phi_i^\vee"]
			&
			M^\vee
		\end{tikzcd}
            \end{equation}
		and $\langle m,\lambda(m)\rangle_i\ge0$ for any non-zero $m\in M$.
		Moreover $\sum_{i=1}^k\langle m,\lambda(m)\rangle_{i}>0$ for any $0\neq m\in M$.

		Conversely, suppose that we have a morphism  $\lambda\colon M\rightarrow N$ in $\mathrm{Mod}_\Gamma$ and that for each $i=1,\cdots,k$, we have an object $(M,N,\langle-,-\rangle_i)$ in $\mathrm{Mor}_{\Gamma}^{\mathbb{N}}$
		such that the above diagram (\ref{factors of polar-diagram}) commutes. Suppose moreover that for each $i=1,\cdots,k$, $\langle m,\lambda(m)\rangle_i\ge0$ for all $m\in M$ and $\sum_{i=1}^k\langle m,\lambda(m)\rangle_{i}>0$ for any $0\neq m\in M$. Define
		\[
		\langle-,-\rangle
		=(\langle-,-\rangle_i)_{i=1}^k
		\colon
		M\times N
		\rightarrow
		\mathbb{Z}^k=P^{\gp}.
		\]
		Then $(M,N,\langle-,-\rangle)$ is an object in $\MorGammaPpPol$ and $\lambda$ is a polarization on it.
	\end{proposition}
	\begin{proof}
		For the first part, by the definition of polarization, we know that for any $m\in M$, $\langle m,\lambda(m)\rangle\in P=\N^k$. Thus for any $i=1,\cdots,k$, we have $\langle m,\lambda(m)\rangle_i\ge0$. Moreover for any $0\neq m\in M$, $\langle m,\lambda(m)\rangle\in\mathbb{N}^k\backslash\{0\}$, that is, at least one entry of $\langle m,\lambda(m)\rangle$ is positive, thus $\sum_{i=1}^k\langle m,\lambda(m)\rangle_i>0$.

        Conversely, the assumptions show that for any $m\in M$, $\langle m,\lambda(m)\rangle\in\mathbb{N}^k$ since $\langle m,\lambda(m)\rangle_i\ge0$ for any $i$. Moreover, for $0\neq m\in M$, $\sum_{i=1}^k\langle m,\lambda(m)\rangle_i>0$ implies $\langle m,\lambda(m)\rangle\in\mathbb{N}^k\backslash\{0\}$. Since we also have the commutativity of the diagrams \eqref{factors of polar-diagram} for $i=1,\cdots,k$, the morphism $\lambda$ is a polarization in $\MorGammaPpPol$.
	\end{proof}

        From this we deduce the following relation between $\MorGammaNpPol$ and $\MorGammaPpPol$
        \begin{corollary}\label{map from high rank to rank 1}
            Let $(M,N,\langle-,-\rangle)$ be an object in $\MorGammaPpPol$. Define
            \[
            \sum_{i=1}^k\langle-,-\rangle_i
            \colon
            M\times N\to\mathbb{Z},
            \quad
            (x,y)\mapsto\sum_{i=1}^k\langle x,y\rangle_i
            \]
            and similarly for $\sum_{i=1}^k\Phi_i$.
            Then $(M,N,\sum_i\langle-,-\rangle_i)=(M,N,\sum_i\Phi_i)$ is an object in $\MorGammaNpPol$.
        \end{corollary}
        \begin{proof}
            First we show $\sum_i\langle-,-\rangle_i$ is non-degenerate. Suppose that $(M,N,\langle-,-\rangle)$ is polarizable in $\Mor_{\Gamma'}^P$ with polarization $\lambda\colon M\to N$ where $\Gamma'$ is an open subgroup of $\Gamma$. From Proposition \ref{polarization of higher rank vs rank-1}, we know $\sum_i\langle x,\lambda(x)\rangle_i>0$ for any $0\neq x\in M$. This shows that $\sum_i\langle-,-\rangle$ is non-degenerate. Now it follows from Proposition \ref{non-deg iff polar iff ppolar} that $(M,N,\sum_i\langle-,-\rangle_i)$ is an object in $\MorGammaNpPol$.

            The correspondence between $\sum_i\langle-,-\rangle_i$ and $\sum_i\Phi_i$ is clear.
        \end{proof}

        Next we give an alternative description of Proposition \ref{polarization of higher rank vs rank-1} in terms of matrices:
        for an object $(M,M^\vee,\Phi)$ in $\MorGammaPpPol$, suppose that $(M,M^\vee,\Phi)$ is polarizable in $\Mor_{\Gamma'}^P$ with polarization $\lambda\colon M\to M^\vee$ where $\Gamma'$ is an open subgroup of $\Gamma$. Without loss of generality, we can assume that $\Gamma'$ acts trivially on $M$. Fix a $\Z$-basis for $M$ and consider its dual basis in $M^\vee$:
        \[
        M=\Z(e_1,\cdots,e_h),
        \quad
        M^\vee=\Z(f_1,\cdots,f_h).
        \]
        In the following, we view elements in $M$ and $M^\vee$ as column matrices using these bases.
        Then the natural $\Gamma$-equivariant pairing between $M$ and $M^\vee$ is given as follows (\emph{cf.} (\ref{natural pairing is Gammma-equivariant})):
        \[
        M\times M^\vee\to\Z,
        \quad
        (x,f)
        \mapsto
        f(x)=x^\mathrm{t}f=f^\mathrm{t}x.
        \]
        Any $\Phi_i$ in $\Phi=(\Phi_1,\cdots,\Phi_k)$ is represented by a matrix $X_i\in\mathrm{End}_{\Mod_\Gamma}(M)\subset\mathrm{Mat}_{\mathrm{rk}_{\Z}(M)}(\Z)$ and thus the corresponding pairing $\langle-,-\rangle_i$ is given by
        \[
        \langle-,-\rangle_i\colon
        M\times M^\vee\to\Z,
        \quad
        (x,f)
        \mapsto
        (X_ix)^\mathrm{t}f
        =
        x^\mathrm{t}X_i^\mathrm{t}f.
        \]
        Similarly, the polarization $\lambda$ is represented by a matrix $\Lambda\in\mathrm{Mat}_{\mathrm{rk}_{\Z}(M)}(\Z)$ and thus we have
        \[
        \langle x,\lambda(y)\rangle_i
        =
        x^\mathrm{t}X_i^\mathrm{t}\Lambda y,
        \quad
        x,y\in M.
        \]
        From Proposition \ref{polarization of higher rank vs rank-1}, we deduce
        \begin{corollary}\label{polarization of higher rank vs rank-1-in terms of matrices}
            Let $(M,M^\vee,\Phi)$ be an object in $\MorGammaPpPol$, and let $X_1,\cdots,X_k,\Lambda $ be as above. Then all $X_1^\mathrm{t}\Lambda,\cdots,X_k^\mathrm{t}\Lambda$ are symmetric, positive semi-definite matrices and $\sum_{i=1}^kX_i^\mathrm{t}\Lambda $ is symmetric and positive definite.

            Conversely, let $(M,M^\vee,\Phi)$ be an object in $\MorGammaP$ with $X_1,\cdots,X_k$ as above. If there is a matrix $\Lambda \in\mathrm{Mat}_{\mathrm{rk}_{\Z}(M)}(\Z)$ such that all $X_i^\mathrm{t}\Lambda $ are symmetric, positive semi-definite matrices and $\sum_{i=1}^kX_i^\mathrm{t}\Lambda $ is symmetric and positive definite, then $(M,M^\vee,\Phi)$ is an object in $\MorGammaPpPol$.
        \end{corollary}

        The following proposition describes the isogeny classes and simple objects in $\MorGammaPpPol$.
        \begin{proposition}\label{isogeny classes in Mor_Gamma^{pPol}}
            \begin{enumerate}
                \item Any object in $\MorGammaPpPol$ of the form $(\mathbb{Z}[\zeta_r],\mathbb{Z}[\zeta_r]^\vee,\langle-,-\rangle)$ is simple.

                \item Any object $(M,N,\langle-,-\rangle)$ in $\MorGammaPpPol$ is isogenous to an object of the form $(R,R^\vee,\langle-,-\rangle')$ with $R=\prod_{r= 1}^\infty\mathbb{Z}[\zeta_r]^{a(r)}$ for a sequence of non-negative integers $(a(r))_{r>0}$ such that $a(r)=0$ for $r\gg0$.

                \item Fix two non-zero $\Gamma$-modules $M',M''$ such that
                \[
                \Hom_{\Mod_\Gamma}(M',M'')=0.
                \]
                Then any object in $\MorGammaPpPol$ of the form
                \[
                (M'\oplus M'',(M')^\vee\oplus (M'')^\vee,\langle-,-\rangle)
                \]
                is \emph{not} simple.
                In particular, any \emph{simple} object $(M,N,\langle-,-\rangle)$ in $\MorGammaPpPol$ is isogenous to a simple object of the form $(\mathbb{Z}[\zeta_r]^a,(\mathbb{Z}[\zeta_r]^\vee)^a,\langle-,-\rangle')$ for some $r,a>0$.
            \end{enumerate}           
        \end{proposition}
        \begin{proof}
            \begin{enumerate}
                \item This is clear by Lemma \ref{pPol LP with simple factor is simple and the converse is true for rank 1 monoid}.

                \item Suppose $(M,N,\langle-,-\rangle)=(M,N,\Phi)$ is polarizable in $\Mor_{\Gamma'}^P$ with polarization $\lambda\colon M\to N$ for an open subgroup $\Gamma'\subset\Gamma$. We know that $(M,N,\sum_i\Phi_i)$ is an object in $\MorGammaNpPol$. By Corollary \ref{canonical form of log free Gamma-module}, we can find $R$ as in the corollary and $(R,R^\vee,\mathrm{Id}_R)$ an object in $\MorGammaNpPol$ such that there is an isogeny $\psi=(\psi_1,\psi_2)\colon(R,R^\vee,\mathrm{Id}_R)\to(M,N,\sum_i\Phi_i)$ in $\MorGammaNpPol$. Here $\psi_1\colon R\to M$ and $\psi_2\colon N\to R^\vee$ ($\psi_2^\vee\colon R\to N^\vee$) are morphisms in $\Mod_{\Gamma}$. The map $\lambda\colon M\to N$ in $\Mod_{\Gamma'}$ induces a map
                \[
                \lambda'=\psi_2\circ\lambda\circ\psi_1\colon R\to R^\vee
                \]
                in $\Mod_{\Gamma'}$.
                Then $(R,R^\vee,\mathrm{Id}_R)$ is polarizable in $\Mor_{\Gamma'}^{\N}$ with polarization $\lambda'$.
                
                Write
                \[
                n=[N^\vee\colon\psi_2^\vee(R)]<\infty,
                \]
                then for any $x\in N^\vee$, $nx\in\psi_2^\vee(R)$.
                Moreover, $(R,R^\vee,n\mathrm{Id}_R)$ is also an object in $\MorGammaNpPol$ and there is an isogeny
                $\psi'=(n\psi_1,\psi_2)\colon (R,R^\vee,n\mathrm{Id}_R)\to(M,N,\sum_i\Phi_i)$ in $\MorGammaN$.
                For each $i=1,\cdots,k$, we define a morphism $\Phi'_i\colon R\to R$ using the following diagram:
                \[
                \begin{tikzcd}
                    R\arrow[r,dashrightarrow,"{\Phi_i'}"]\arrow[d,"n\psi_1"']
                    &
                    R\arrow[d,"{\psi_2^\vee}"]                    
                    \\
                    M\arrow[r,"\Phi_i"] & N^\vee
                \end{tikzcd}
                \]
                more precisely, for any $x\in R$, $\Phi_i(n\psi_1(x))=n\Phi_i(\psi_1(x))\in N^\vee$ lies in $\psi_2^\vee(R)$ by definition of $n$, then we choose (the unique) $x'\in R$ such that $\psi_2^\vee(x')=n\Phi_i(\psi_1(x))$ and set
                \[
                \Phi_i'(x):=x'.
                \]
                Since $\psi_1,\Phi_i,\psi_2^\vee$ are all linear and $\Gamma$-equivariant, so is $\Phi_i'$. The pairing $\langle-,-\rangle_i'\colon R\times R^\vee\to\mathbb{Z}$ corresponding to $\Phi_i'$ is given as follows: for any $x\in R$ and $f\in R^\vee$,
                \[
                \langle x,f\rangle_i'
                :=
                \langle \psi_1(x),\psi_2^{-1}(nf)\rangle_i.
                \]
                Then we put $\Phi'=(\Phi_1',\cdots,\Phi_k')$ and $\langle-,-\rangle'=(\langle-,-\rangle'_1,\cdots,\langle-,-\rangle_k')$.
                
                It is clear that $(R,R^\vee,\Phi')\in\MorGammaP$, and $(\psi_1,\psi_2)\colon(R,R^\vee,\Psi')\to(M,N,\Phi)$ is an isogeny. By Proposition \ref{simpleness of LP is preserved by isogenies} (2), we have $(R,R^\vee,\Phi')\in\MorGammaPpPol$.

                \item It suffices to prove the first part. We fix a $\mathbb{Z}$-basis for $M'$ (resp. $M''$) and consider the dual basis for $(M')^\vee$ (resp. $(M'')^\vee$).
                By assumption, there is no non-zero $\Gamma$-equivariant morphism between $M'$ and $M''$, so by the discussion before Corollary \ref{polarization of higher rank vs rank-1-in terms of matrices}, the morphisms $\Phi_i\colon M'\oplus M''\to M'\oplus M''$ ($i=1,\cdots,k$) are represented by block matrices of the form
                \[
                X_i=\begin{pmatrix}
                    X_i' & 0 \\ 0 & X_i''
                \end{pmatrix}
                \]
                with $X_i'\in\mathrm{End}_{\Mod_\Gamma}(M')\subset\mathrm{Mat}_{\mathrm{rk}_{\mathbb{Z}}(M')}(\mathbb{Z})$ and $X_i''\in\mathrm{End}_{\Mod_\Gamma}(M'')\subset\mathrm{Mat}_{\mathrm{rk}_{\mathbb{Z}}(M'')}(\mathbb{Z})$.
                Since $(M'\oplus M'',(M')^\vee\oplus(M'')^\vee,\Phi)$ is polarizable in $\Mor_{\Gamma'}^P$ for some open subgroup $\Gamma'\subset\Gamma$ (we can assume $\Gamma'$ acts trivially on $M'$ and $M''$) with a polarization $\lambda\colon M'\oplus M''\to (M')^\vee\oplus(M'')^\vee$. Then under the fixed basis, $\lambda$ is represented by a block matrix
                \[
                \Lambda =\begin{pmatrix}
                    A & B \\ C & D
                \end{pmatrix}
                \]
                with $A\in\mathrm{Mat}_{\mathrm{rk}_{\mathbb{Z}}(M')}(\mathbb{Z})$ and $D\in\mathrm{Mat}_{\mathrm{rk}_{\mathbb{Z}}(M'')}(\mathbb{Z})$. By Corollary \ref{polarization of higher rank vs rank-1-in terms of matrices}, the matrices $X_1^\mathrm{t}\Lambda , \cdots, X_k^\mathrm{t}\Lambda $ are all symmetric, positive semi-definite and $\sum_{i=1}^kX_i^\mathrm{t}\Lambda$ is symmetric and positive definite. We deduce that $(X_1')^\mathrm{t}A,\cdots,(X_k')^\mathrm{t}A$ and $(X_1'')^\mathrm{t}D,\cdots,(X_k'')^\mathrm{t}D$ are all symmetric and positive semi-definite and moreover $\sum_{i=1}^k(X_i')^\mathrm{t}A$ and $\sum_{i=1}^k(X_i'')^\mathrm{t}D$ are both symmetric and positive definite. We define a linear map $\lambda'\colon M'\to(M')^\vee$ represented by the matrix $A$.
                It follows from Corollary \ref{polarization of higher rank vs rank-1-in terms of matrices} that
                \[
                (M',(M')^\vee,\Phi':=\Phi\mid_{M'})
                \]
                is polarizable in $\Mor_{\Gamma'}^P$ with polarization $\lambda'$. In particular, we have constructed an object $(M',(M')^\vee,\Phi')$ in $\MorGammaPpPol$ together with a morphism
                \[
                \psi=(\psi_1,\psi_2)
                \colon
                (M',(M')^\vee,\Phi')
                \to
                (M'\oplus M'',(M')^\vee\oplus(M'')^\vee,\Phi)
                \]
                where $\psi_1$ (resp. $\psi_2$) is the canonical inclusion (resp. projection) map. Clearly $\coker(\psi_1)$ and $\coker(\psi_2^\vee)$ are both of rank equal to $\mathrm{rk}_{\mathbb{Z}}(M'')>0$, so $\psi$ is not an isogeny and thus $(M'\oplus M'',(M')^\vee\oplus(M'')^\vee,\Phi)$ is \emph{not} simple.

            \end{enumerate}
        \end{proof}

        \subsection{Polarizations on $(\mathbb{Z}[\zeta_r],\mathbb{Z}[\zeta_r]^\vee,\langle-,-\rangle)$}
	In this subsection we will concentrate on objects in $\MorGammaP$ of the form
        \[
        (M,N,\Phi)=(M,N,\langle-,-\rangle)=(\mathbb{Z}[\zeta_{r}],\mathbb{Z}[\zeta_{r}]^\vee,\langle-,-\rangle).
        \]
        Then the morphisms $\Phi_i\colon M\rightarrow N^\vee$ in $\mathrm{Mod}_{\Gamma}$ are of the form
	\begin{equation}\label{t_1,...,t_k}
		\Phi_i
		\colon
		\mathbb{Z}[\zeta_{r}]
		\rightarrow
		\mathbb{Z}[\zeta_{r}],
		\quad
		m
		\mapsto
		t_im
	\end{equation}
	for some $t_i\in\mathbb{Z}[\zeta_{r}]$.

	Let $\lambda\colon M\rightarrow N$ be a morphism in $\mathrm{Mod}_{\Gamma}$, so it is of the form $\lambda=\lambda_t$ for some non-zero $t\in\mathfrak{d}_{\mathbb{Z}[\zeta_{r}]}^{-1}$ (see the proof of Proposition \ref{polarization on (Z[zeta_r],Z[zeta_r]^vee,Id)}). 
	Then for $i=1,\cdots,k$ and $x,y\in\mathbb{Z}[\zeta_{r}]$, we have
	\[
	\langle
	x,\lambda(y)
	\rangle_i
	=
	\mathrm{Tr}_{\mathbb{Q}(\zeta_{r})/\mathbb{Q}}(xt_i\overline{ty}).
	\]

	\begin{proposition}\label{condition for lambda_t being a polarization}
            Let $(M,N,\langle-,-\rangle)=(\mathbb{Z}[\zeta_{r}],\mathbb{Z}[\zeta_{r}]^\vee,\langle-,-\rangle)$ and $\lambda=\lambda_t$ be as above.
		\begin{enumerate}
		    \item Suppose $\lambda$ is a polarization on $(M,N,\langle-,-\rangle)$. Then for any $i=1,\cdots,k$, $t_i\overline{t}$ is either totally positive or equal to $0$. Fix one such $i$, if moreover $\langle m_0,\lambda(m_0)\rangle_i>0$ for some $m_0\in M$, then $t_i\overline{t}$ is totally positive. In particular, $\sum_{i=1}^kt_i\overline{t}$ is totally positive.

                \item Conversely, if for any $i=1,\cdots,k$, $t_i\overline{t}$ is either totally positive or equal to $0$ and $\sum_{i=1}^kt_i\overline{t}$ is totally positive, then $\lambda$ is a polarization on $(M,N,\langle-,-\rangle)$. 
		\end{enumerate}
	\end{proposition}
	\begin{proof}
		\begin{enumerate}
		    \item By definition, for any $x,y\in M$,
            \[
            \langle x,\lambda(y)\rangle_i
            =
            \mathrm{Tr}_{\mathbb{Q}(\zeta_r)/\mathbb{Q}}(xt_i\overline{ty}),
            \]
            which is symmetric in $x$ and $y$, thus as in the proof of Proposition \ref{polarization on (Z[zeta_r],Z[zeta_r]^vee,Id)}, we know $t_i\overline{t}$ is totally real. Moreover, for any $x\in M$,
            \[
            \langle x,\lambda(x)\rangle_i\ge0.
            \]
            Again as in the proof of Proposition \ref{polarization on (Z[zeta_r],Z[zeta_r]^vee,Id)}, $t_i\overline{t}$ is totally positive or zero. This gives the first part of the proposition.
        
        For the second part, by assumption,
		\[
		\mathrm{Tr}_{\mathbb{Q}(\zeta_{r})/\mathbb{Q}}
		(m_0\overline{m_0}t_i\overline{t})\ne0.
		\]
		This clearly implies that $t_i\overline{t}\ne0$ and thus it is totally positive.	

               \item We apply Proposition \ref{polarization of higher rank vs rank-1}.
		\end{enumerate}	
	\end{proof}

	\begin{corollary}\label{characterization of polarizations-higher rank}
	    Let $(M,N,\langle-,-\rangle)=(\mathbb{Z}[\zeta_{r}],\mathbb{Z}[\zeta_{r}]^\vee,\langle-,-\rangle)$ be as above. It has a polarization if and only if $t_1,\cdots,t_k$ are not all zero and for any $i,j=1,\cdots,k$, $t_i\overline{t_j}$ is either totally positive or zero.
	\end{corollary}
	\begin{proof}
	    The direction $\Rightarrow$ is clear by Proposition \ref{condition for lambda_t being a polarization} (1).

        Conversely, the morphism $\lambda_{t}\colon M\to N$ with $t=t_1+\cdots+t_k$ is a polarization on $(M,N,\langle-,-\rangle)$. Indeed, suppose that $t_{i_0}\neq0$ for some $i_0=1,\cdots,k$. Then $t_i/t_{i_0}=t_i\overline{t_{i_0}}/(t_{i_0}\overline{t_{i_0}})$ is totally positive or zero, and not all of them are zero. Thus $t/t_{i_0}=\sum_{i=1}^kt_i/t_{i_0}$ is totally positive, in particular, $t\neq0$. Now for any $i=1,\cdots,k$, we have
        \[        t_i\overline{t}=t_i\overline{t_{i_0}}\times\overline{t/t_{i_0}},
        \]
        which is totally positive or zero. Moreover, $\sum_{i=1}^kt_i\overline{t}=t\overline{t}$ is totally positive. By Proposition \ref{condition for lambda_t being a polarization} (2), we conclude.
	\end{proof}

	\subsection{Isogeny classes of $(\mathbb{Z}[\zeta_r]^a,(\mathbb{Z}[\zeta_r]^a)^\vee,\langle-,-\rangle)$}\label{Isogeny classes of Z^a}
        We are now ready for the classification of a certain class of objects in $\MorGammaPpPol$ up to isogeny as follows. Let us first give a definition
	\begin{definition}\label{T_{r,k}}
            Recall that $P=\N^k$.
            \begin{enumerate}
                \item For a positive integer $r>0$, we define
            \[
            T_{r,k}
            =
            \left\{(t_1,\cdots,t_k)\in(\mathbb{Q}(\zeta_r)^+)^k
            \mid
            t_1,\cdots,t_k\text{ totally positive or zero and }\sum_{i=1}^kt_i=1
            \right\}
            \]

                \item We fix a $\mathbb{Z}$-basis of $\mathbb{Z}[\zeta_r]$, which induces an embedding $\mathbb{Q}(\zeta_r)\hookrightarrow\mathrm{Mat}_{\phi(r)}(\mathbb{Q})$. For positive integers $a,r>0$, we define
                \[
                T_{r,k}^{(a)}
                \]
                to be the set of $k$-tuples
                $(X_1,\cdots,X_k)\in\mathrm{Mat}_{a}(\mathbb{Q}(\zeta_r))^k$ with $\sum_{i=1}^kX_i=1$ such that there is a matrix $\Lambda \in\mathrm{GL}_{a\phi(r)}(\mathbb{Q})$ with $X_1^{\mathrm{t}}\Lambda ,\cdots,X_k^{\mathrm{t}}\Lambda $ all symmetric and positive semi-definite (here we view $X_1,\cdots,X_k$ also as elements in $\mathrm{Mat}_{a\phi(r)}(\mathbb{Q})$ via the embedding $\mathbb{Q}(\zeta_r)\hookrightarrow\mathrm{Mat}_{\phi(r)}(\mathbb{Q})$). We define an action of $\mathrm{GL}_{a}(\mathbb{Q}(\zeta_r))$ on $T_{r,k}^{(a)}$ by
                \[
                Y(X_1,\cdots,X_k)
                :=
                (YX_1Y^{-1},\cdots,YX_kY^{-1}),
                \quad
                \forall\,
                Y\in\mathrm{GL}_{a\phi(r)}(\mathbb{Q}),\,
                (X_1,\cdots,X_k)\in T_{r,k}^{(a)}.
                \]
                Then we write
                \[
                T_{r,k}^{(a)}/\sim
                \]
                for the quotient of $T_{r,k}^{(a)}$ by the action of $\mathrm{GL}_{a}(\mathbb{Q}(\zeta_r))$. We denote the image of $(X_1,\cdots,X_k)\in T_{r,k}^{(a)}$ in $T_{r,k}^{(a)}/\sim$ by $[X_1,\cdots,X_k]$.
            \end{enumerate}
            
	\end{definition}

	\begin{remark}
            \begin{enumerate}
                \item For $\Q(\zeta_r)^+=\Q$ (that is, $r=1,2,3,4,6$), regarding $\Q(\zeta_r)^+$ as a subfield of $\R$ canonically, we have a canonical embedding into the $(k-1)$-simplex
                \[T_{r,k}\hookrightarrow \Delta^{k-1}:=\{(t_1,\cdots,t_k)\in\R_{\geq0}^k\mid \sum_i t_i=1\}\]
                thus $T_{r,k}$ is just $\Delta^{k-1}\cap \Q^k$.

                \item It is easy to see that both $T_{r,k}^{(a)}$ and $T_{r,k}^{(a)}/\sim$ are independent of the choice of the $\mathbb{Z}$-basis of $\mathbb{Z}[\zeta_r]$.
            \end{enumerate}           
        \end{remark}
 
	We have the following two classification results, which are the main results of this section
	\begin{theorem}\label{classificaton of isogenous classes of log Gamma-module of higher ranks}
		Let $P=\mathbb{N}^k$.
		The following map defines a bijection
		\begin{align}\label{map of classification of isogeny log Gamms-modules of higher ranks}
			\begin{split}
				\left\{
                (M,N,\Phi)\in
				\MorGammaPpPol
                \mid M
                \text{ simple free $\Gamma$-module}
				\right\}/\text{isogeny}
				&
				\rightarrow
				\bigsqcup_{r>0}T_{r,k}
				\\
				(\mathbb{Z}[\zeta_{r}],\mathbb{Z}[\zeta_{r}]^\vee,\Phi)
				&
				\mapsto
				(t_1/t,\cdots,t_k/t)
				\in
				T_{r,k}
			\end{split}			
		\end{align}
		Here $t_1,\cdots,t_k$ are determined by $\Phi$ as in (\ref{t_1,...,t_k}) and $t=\sum_{i=1}^kt_i$.
	\end{theorem}
	\begin{proof}
		We first show that the map (\ref{map of classification of isogeny log Gamms-modules of higher ranks}) is well-defined. We can assume $t_{i_0}\neq0$ for some $i_0\in\{1,\cdots,k\}$. Then as in the proof of Corollary \ref{characterization of polarizations-higher rank}, $t_i/t_{i_0}$ is totally positive or zero for all $i=1,\cdots,k$, and $t/t_{i_0}$ is totally positive, thus $t_i/t=t_i/t_{i_0}\times t_{i_0}/t$ is totally positive or zero and $\sum_{i=1}^kt_i/t=1$. So $(t_1/t,\cdots,t_k/t)$ indeed lies in $T_{r,k}$. Any isogeny in $\MorGammaPpPol$
		\[
		\psi=(\psi_1,\psi_2)
		\colon
		(\mathbb{Z}[\zeta_{r}],\mathbb{Z}[\zeta_{r}]^\vee,\Phi)\rightarrow
		(\mathbb{Z}[\zeta_{r}],\mathbb{Z}[\zeta_{r}]^\vee,\widetilde{\Phi})
		\]
		is given by two morphisms $\psi_1,\psi_2^\vee\colon\mathbb{Z}[\zeta_{r}]\rightarrow\mathbb{Z}[\zeta_{r}]$  in $\mathrm{Mod}_{\Gamma}$ (\textit{cf.} the two vertical arrows in (\ref{morphism of log free Gamma-modules})) such that the following diagrams are commutative ($i=1,\cdots,k$)
		\[
		\begin{tikzcd}
			\mathbb{Z}[\zeta_{r}]
			\arrow[r,"m\mapsto t_im"]
			\arrow[d,"\psi_1"]
			&
			\mathbb{Z}[\zeta_{r}]
			\arrow[d,"\psi_2^\vee"]
			\\
			\mathbb{Z}[\zeta_{r}]
			\arrow[r,"m\mapsto \widetilde{t}_im"]
			&
			\mathbb{Z}[\zeta_{r}]
		\end{tikzcd}
		\]
		Here $(t_1,\cdots,t_k),(\widetilde{t}_1,\cdots,\widetilde{t}_k)$ are determined by $\Phi,\widetilde{\Phi}$ respectively. Note that $\psi_1$ and $\psi_2^\vee$ are necessarily of the form $\Psi_{x_1}$ and $\Psi_{x_2}$:
		\[
		\Psi_{x_i}
		\colon
		\mathbb{Z}[\zeta_{r}]
		\rightarrow
		\mathbb{Z}[\zeta_{r}],
		\quad
		m
		\mapsto
		x_im
		\,
		(i=1,2)
		\]
		for $x_1,x_2$ non-zero elements in $\mathbb{Z}[\zeta_{r}]$.
		It follows that
		\[
		t_ix_2=x_1\widetilde{t}_i.
		\]
		So $(t_1,\cdots,t_k)$ and $(\widetilde{t}_1,\cdots,\widetilde{t}_k)$ have the same image in $T_{r,k}$ and thus (\ref{map of classification of isogeny log Gamms-modules of higher ranks}) is well-defined.

		For the surjectivity, given $(t_1,\cdots,t_k)\in T_{r,k}$, we can find $t\in\mathbb{Q}(\zeta_r)^\times$ such that $tt_1,\cdots,tt_k\in\mathbb{Z}[\zeta_r]$. Then we define morphisms $\Phi_i$ in $\Mod_{\Gamma}$ by
        \[
        \Phi_i\colon
        \mathbb{Z}[\zeta_r]\to\mathbb{Z}[\zeta_r],
        \quad
        x\mapsto tt_ix.
        \]
	Put $\Phi=(\Phi_1,\cdots,\Phi_k)$, then $(M,N,\Phi)$ is an object in $\MorGammaP$ and by Corollary \ref{characterization of polarizations-higher rank}, it is polarizable, thus it is an object in $\MorGammaPpPol$. Moreover, its image in $T_{r,k}$ is exactly $(t_1,\cdots,t_k)$.

		Next we prove the injectivity of (\ref{map of classification of isogeny log Gamms-modules of higher ranks}): suppose we have two objects $(\mathbb{Z}[\zeta_{r}],\mathbb{Z}[\zeta_{r}]^\vee,\Phi)$ and
		$(\mathbb{Z}[\zeta_{r}],\mathbb{Z}[\zeta_{r}]^\vee,\widetilde{\Phi})$ in $\MorGammaPpPol$ such that
        \[
        (t_1/t,\cdots,t_k/t)=(\widetilde{t_1}/\widetilde{t},\cdots,\widetilde{t_k}/\widetilde{t}).
        \]
        Then we have $t_i\widetilde{t}=\widetilde{t_i}t$ and $0\neq t,\widetilde{t}\in\mathbb{Z}[\zeta_r]$. It is easy to see that the morphism
        \[
        \psi=(\psi_1=\Psi_{t}, \psi_2=\Psi_{\widetilde{t}}^\vee)
        \colon
        (\mathbb{Z}[\zeta_{r}],\mathbb{Z}[\zeta_{r}]^\vee,\Phi)
        \to(\mathbb{Z}[\zeta_{r}],\mathbb{Z}[\zeta_{r}]^\vee,\widetilde{\Phi})
        \]
        is an isogeny in $\MorGammaP$. This gives the injectivity.
	\end{proof}
	Combined with the discussion in the end of §\ref{subsection LAV without abelian part over finite log points are classified by LP up to isogeny}, this gives Theorem \ref{main theorem-2}.

	\begin{remark}
		For $k=1$, this is compatible with Theorem \ref{classificaton of isogenous classes of log Gamma-module of rank 1}: in this case, each $T_{r,1}$ is simply a singleton.
	\end{remark}

	Fix a positive integer $a>0$. Then just as in the case of $(\mathbb{Z}[\zeta_r],\mathbb{Z}[\zeta_r]^\vee,\Phi)\in\MorGammaP$, for an object $(\Z[\zeta_r]^a,(\mathbb{Z}[\zeta_r]^a)^\vee,\Phi)\in\MorGammaP$, the morphisms $\Phi_1,\cdots,\Phi_k\colon\Z[\zeta_r]^a\to\Z[\zeta_r]^a$ are necessarily of the form
    \begin{equation}\label{X_1,...,X_k}
        \Phi_i
    \colon
    \mathbb{Z}[\zeta_r]^a
    \to
    \mathbb{Z}[\zeta_r]^a,
    \quad
    m
    \mapsto
    X_im
    \end{equation}
    for some $X_i\in\mathrm{Mat}_a(\mathbb{Z}[\zeta_r])$ (\emph{cf.} (\ref{t_1,...,t_k})). Here we view elements in $\mathbb{Z}[\zeta_r]^a$ as column matrices with entries in $\Z[\zeta_r]$.
    
        \begin{theorem}\label{classificaton of isogenous classes of log Gamma-module of higher ranks-2}
            Let $P=\mathbb{N}^k$. The following map defines a bijection
            \begin{align}\label{map of classification of isogeny log Gamms-modules of higher ranks-2}
			\begin{split}
				\left\{
                (\mathbb{Z}[\zeta_r]^a,(\mathbb{Z}[\zeta_r]^a)^\vee,\Phi)\in
				\MorGammaPpPol
                \mid a>0
				\right\}/\text{isogeny}
				&
				\rightarrow
				\bigsqcup_{a,r>0}T_{r,k}^{(a)}/\sim
				\\
				(\mathbb{Z}[\zeta_{r}]^a,(\mathbb{Z}[\zeta_{r}]^a)^\vee,\Phi)
				&
				\mapsto
                    [X_1X^{-1},\cdots,X_kX^{-1}]
				\in
				T_{r,k}^{(a)}/\sim
			\end{split}			
		\end{align}
		Here $X_1,\cdots,X_k$ are determined by $\Phi$ as in (\ref{X_1,...,X_k}) and $X=\sum_{i=1}^kX_i$. Moreover, $(\mathbb{Z}[\zeta_{r}]^a,(\mathbb{Z}[\zeta_{r}]^a)^\vee,\Phi)$ is \emph{simple} if and only if there is \emph{no} $Z\in\mathrm{GL}_a(\mathbb{Q}(\zeta_r))$ such that $ZX_iZ^{-1}$ are all of the form
        \[
            ZX_iZ^{-1}
            =
            \begin{pmatrix}
                X_i' & \ast \\ 0 & X_i''
            \end{pmatrix}
        \]
        where $X_i'$ are all of size $b\times b$ for some integer $0<b<a$.
        \end{theorem}
        \begin{proof}
            The proof is very similar to the preceding theorem. We first show that the map in \eqref{map of classification of isogeny log Gamms-modules of higher ranks-2} is well-defined. We fix a $\mathbb{Z}$-basis of $\mathbb{Z}[\zeta_r]$ as in Definition \ref{T_{r,k}} (2) and a dual $\Z$-basis for $\Z[\zeta_r]$.
            Suppose that $(M,N,\Phi)=(\mathbb{Z}[\zeta_{r}]^a,(\mathbb{Z}[\zeta_{r}]^a)^\vee,\Phi)$ is polarizable in $\Mor_{\Gamma'}^P$ for some open subgroup $\Gamma'$ of $\Gamma$ with a polarization $\lambda\colon M\to N$. We can assume that $\Gamma'$ acts trivially on $M$ and $N$. Then $\lambda$ is represented by a matrix $\Lambda \in\mathrm{GL}_{a\phi(r)}(\mathbb{Q})$. 
            So all of $X_1^{\mathrm{t}}\Lambda ,\cdots,X_k^{\mathrm{t}}\Lambda $ are symmetric and positive semi-definite and $\sum_{i=1}^kX_i^{\mathrm{t}}\Lambda $ is symmetric and positive definite (see Corollary \ref{polarization of higher rank vs rank-1-in terms of matrices}). In particular, $X=\sum_iX_i$ is non-singular and thus $[X^{-1}X_1,\cdots,X^{-1}X_k]$ lies in $T_{r,k}^{(a)}/\sim$. Moreover, if we have an isogeny in $\MorGammaPpPol$
            \[
            \psi=(\psi_1,\psi_2)\colon
            (\mathbb{Z}[\zeta_{r}]^a,(\mathbb{Z}[\zeta_{r}]^a)^\vee,\Phi)
            \to
            (\mathbb{Z}[\zeta_{r}]^a,(\mathbb{Z}[\zeta_{r}]^a)^\vee,\widetilde{\Phi})
            \]
            with $\psi_i$ ($i=1,2$) given by
            \[
            \Psi_{Y_i}
            \colon
            \mathbb{Z}[\zeta_r]^a
            \to
            \mathbb{Z}[\zeta_r]^a,
            \quad
            m\mapsto Y_im
            \]
            where $Y_i\in\mathrm{Mat}_a(\mathbb{Z}[\zeta_r])\bigcap
            \mathrm{GL}_a(\mathbb{Q}(\zeta_r))$, then as in the preceding proof, we have
            \[
            X_iY_2=Y_1\widetilde{X}_i.
            \]
            It follows that
            $\widetilde{X}^{-1}\widetilde{X}_i=(Y_1^{-1}XY_2)^{-1}(Y_1^{-1}X_iY_2)=Y_2^{-1}X^{-1}X_iY_2$ and thus $[X^{-1}X_1,\cdots,X^{-1}X_k]$ and $[\widetilde{X}^{-1}\widetilde{X}_1,\cdots,\widetilde{X}^{-1}\widetilde{X}_k]$ are the same in $T_{r,k}^{(a)}/\sim$. So the map (\ref{map of classification of isogeny log Gamms-modules of higher ranks-2}) is well-defined.

            For the surjectivity, given $[X_1,\cdots,X_k]\in T_{r,k}^{(a)}/\sim$, by definition, there is $\Lambda \in\mathrm{GL}_{a\phi(r)}(\mathbb{Q})$ such that $X_1^{\mathrm{t}}\Lambda ,\cdots,X_k^{\mathrm{t}}\Lambda $ are all symmetric and positive semi-definite. Up to multiplying $X_1,\cdots,X_k,\Lambda $ by some positive integer, we can assume that $X_1,\cdots,X_k\in\mathrm{Mat}_a(\mathbb{Z}[\zeta_r])$ (so that $X=\sum_iX_i=n\in\mathbb{N}$) and $\Lambda \in\mathrm{Mat}_{a\phi(r)}(\mathbb{Z})\bigcap\mathrm{GL}_{a\phi(r)}(\mathbb{Q})$.
            Then we define the morphisms $\Phi_i$ in $\Mod_\Gamma$ by
            \[
            \Phi_i
            \colon
            \mathbb{Z}[\zeta_r]^a
            \to
            \mathbb{Z}[\zeta_r]^a,
            \quad
            m
            \mapsto
            X_im.
            \]
            Put $\Phi=(\Phi_1,\cdots,\Phi_k)$, then $(M,N,\Phi)$ is indeed an object in $\MorGammaP$. Moreover, take $\Gamma'$ to be an open subgroup of $\Gamma$ acting trivially on $M,N$, then $(M,N,\Phi)$ is polarizable in $\Mor_{\Gamma'}^P$ with a polarization
            \[
            \lambda
            \colon
            \mathbb{Z}[\zeta_r]^a
            \to
            (\mathbb{Z}[\zeta_r]^a)^\vee,
            \quad
            m
            \mapsto
            \Lambda m.
            \]

            Next we prove the injectivity of (\ref{map of classification of isogeny log Gamms-modules of higher ranks-2}): suppose we have two objects
            $(\mathbb{Z}[\zeta_r]^a,
            (\mathbb{Z}[\zeta_r]^a)^\vee,\Phi)$ and $(\mathbb{Z}[\zeta_r]^a,
            (\mathbb{Z}[\zeta_r]^a)^\vee,\widetilde{\Phi})$ in $\MorGammaPpPol$ such that there is a matrix $Y_2\in\mathrm{GL}_{a}(\mathbb{Q}(\zeta_r))$ with
            \[
            (X^{-1}X_1,\cdots,X^{-1}X_k)
            =            Y_2(\widetilde{X}^{-1}\widetilde{X}_1,\cdots,\widetilde{X}^{-1}\widetilde{X}_k)\in T_{r,k}^{(a)}.
            \]
            Up to multiplying $Y_2$ by a positive integer, we can assume
            \[
            Y_1:=XY_2\widetilde{X}^{-1}
            \in\mathrm{Mat}_a(\mathbb{Z}[\zeta_r])
            \bigcap\mathrm{GL}_a(\mathbb{Q}(\zeta_r)).
            \]
            Then it follows that for any $i=1,\cdots,k$
            \[
            X_iY_2=Y_1\widetilde{X}_i.
            \]
            We define morphisms $\Psi_{Y_i}\colon\mathbb{Z}[\zeta_r]^a\to\mathbb{Z}[\zeta_r]^a$ sending $m$ to $Y_im$ and it follows that
            \[
            \psi=(\psi_1=\Psi_{Y_1},\psi_2=\Psi_{Y_2}^\vee)
            \colon
            (\mathbb{Z}[\zeta_r]^a,(\mathbb{Z}[\zeta_r]^a)^\vee,\Phi)
            \to
            (\mathbb{Z}[\zeta_r]^a,(\mathbb{Z}[\zeta_r]^a)^\vee,\widetilde{\Phi})
            \]
            is an isogeny in $\MorGammaPpPol$. This gives the injectivity.

            For the last point, if $(\mathbb{Z}[\zeta_r]^a,(\mathbb{Z}[\zeta_r]^a)^\vee,\Phi)$ is not simple, that is, there is a morphism of the form
            \[
            \psi=(\psi_1,\psi_2)\colon
            (M',N',\Phi'):=(\mathbb{Z}[\zeta_r]^b,(\mathbb{Z}[\zeta_r]^b)^\vee,\Phi')
            \to
            (\mathbb{Z}[\zeta_r]^a,(\mathbb{Z}[\zeta_r]^a)^\vee,\Phi)
            \]
            for some $0<b<a$ such that $\psi_1,\psi_2^\vee$ are injective and $\psi$ is not an isogeny. We can view $M'$ (resp. $(N')^\vee$ ) as a submodule of $\mathbb{Z}[\zeta_r]^a$ via $\psi_1$ (resp. $\psi_2^\vee$). In particular,
            \[
            \Phi_i(M')\subset (N')^\vee,
            \quad
            \forall\,
            i=1,\cdots,k.
            \]
            In other words, we can find $Z\in\mathrm{Mat}_a(\mathbb{Z}[\zeta_r])\bigcap\mathrm{GL}_a(\mathbb{Q}(\zeta_r))$ such that
            \[
            ZX_iZ^{-1}
            =
            \begin{pmatrix}
                X_i' & \ast \\ 0 & X_i''
            \end{pmatrix}
            \]
            where $X_i'$ is a square matrix of size $b\times b$, $X_i''$ of size $(a-b)\times(a-b)$.

            Conversely, if there is $Z\in\mathrm{GL}_a(\mathbb{Q}(\zeta_r))$ such that $ZX_iZ^{-1}$ are as in the statement of the theorem, then up to multiplying $Z$ by a positive integer, we can assume $Z\in\mathrm{Mat}_a(\mathbb{Z}[\zeta_r])\bigcap\mathrm{GL}_a(\mathbb{Q}(\zeta_r))$. 
            We define a morphism of $\Gamma$-modules to be the following composition
            \[
            \psi_1\colon
            M':=\mathbb{Z}[\zeta_r]^b\hookrightarrow\mathbb{Z}[\zeta_r]^a\xrightarrow{Z}\Z[\zeta_r]^a,
            \]
            where the first map is the inclusion via the first $b$ factors and the second map is given by the matrix $Z$ with respect to the standard $\Z[\zeta_r]$-basis of $\Z[\zeta_r]^a$. We define $\psi_2$ by specifying $\psi_2^\vee\colon(N')^\vee:=\Z[\zeta_r]^b\to\Z[\zeta_r]^a$ by the same formula.       
            Then we define $\Phi_i'\colon M'\to(N')^\vee$ represented by the matrix $X_i'$. This gives rise to an object $(M',N',\Phi')$ in $\MorGammaP$ together with a morphism $(M',N',\Phi')\to(\mathbb{Z}[\zeta_r]^a,(\mathbb{Z}[\zeta_r]^a)^\vee,\Phi)$ in $\MorGammaP$. Suppose $(\mathbb{Z}[\zeta_r]^a,(\mathbb{Z}[\zeta_r]^a)^\vee,\Phi)$ is polarizable in $\Mor_{\Gamma'}^P$ with polarization $\lambda$ (again we can assume $\Gamma'$ acts trivially on $\mathbb{Z}[\zeta_r]^a$). So under the fixed basis of $\mathbb{Z}[\zeta_r]^a$ and the dual basis of $(\mathbb{Z}[\zeta_r]^a)^\vee$, up to replacing
            $\lambda$ by a positive integer multiple of $\lambda$, we can assume that
            $\lambda$ is represented by a matrix of the form
            \[
            \Lambda =
            Z^\mathrm{t}
            \begin{pmatrix}
                A & B \\ C & D
            \end{pmatrix}
            Z
            \]
            with $A\in\mathrm{Mat}_{\mathrm{rk}_{\mathbb{Z}}(M')}(\mathbb{Z})$. Note that
            \[
            X_i^\mathrm{t}\Lambda 
            =
            Z^{\mathrm{t}}
            \begin{pmatrix}
                X_i' & \ast \\ 0 & X_i''
            \end{pmatrix}^\mathrm{t}
            Z^{-\mathrm{t}}
            Z^\mathrm{t}
            \begin{pmatrix}
                A & B \\ C & D
            \end{pmatrix}
            Z
            =
            Z^{\mathrm{t}}
            \begin{pmatrix}
                (X_i')^\mathrm{t}A & \ast \\ \ast & \ast
            \end{pmatrix}
            Z.
            \]
            Since $X_1^\mathrm{t}\Lambda ,\cdots,X_k^\mathrm{t}\Lambda $ are all symmetric and positive semi-definite and $X^\mathrm{t}X$ is symmetric and positive definite, it follows that $(X_1')^\mathrm{t}A,\cdots,(X_k')^\mathrm{t}A$ are all symmetric and positive semi-definite and $\sum_{i=1}^k(X_i')^\mathrm{t}A$ is symmetric and positive definite. Then it follows that the map $\lambda'\colon\mathbb{Z}[\zeta_r]^b\to(\mathbb{Z}[\zeta_r]^b)^\vee$ represented by the matrix $A$ is a polarization for the object $(\mathbb{Z}[\zeta_r]^b,(\mathbb{Z}[\zeta_r]^b)^\vee,\Phi')$ in $\Mor_{\Gamma'}^P$ and thus $(\mathbb{Z}[\zeta_r]^b,(\mathbb{Z}[\zeta_r]^b)^\vee,\Phi')$ is pointwise polarizable in $\MorGammaP$. We conclude that $(\mathbb{Z}[\zeta_r]^a,(\mathbb{Z}[\zeta_r]^a)^\vee,\Phi)$ is \emph{not} simple.
            
        \end{proof}
        \begin{remark}
            It is not hard to see that $(\mathbb{Z}[\zeta_r]^a,(\mathbb{Z}[\zeta_r]^a)^\vee,\Phi)$ is polarizable in $\MorGammaPpPol$ if and only if for the $k$-tuple
            $(X_1X^{-1},\cdots,X_kX^{-1})\in T_{r,k}^{(a)}$, we can choose $\Lambda\in\mathrm{GL}_a(\mathbb{Q}(\zeta_r))$ such that $X_1^\mathrm{t}\Lambda,\cdots,X_k^\mathrm{t}\Lambda$ are all symmetric (as elements in $\mathrm{Mat}_{a\phi(r)}(\mathbb{Q})$), positive semi-definite.
        \end{remark}
        
        Combining Theorems \ref{classificaton of isogenous classes of log Gamma-module of higher ranks} and \ref{classificaton of isogenous classes of log Gamma-module of higher ranks-2}, we get
        \begin{corollary}
            The following map is a bijection
            \[
            T_{r,k}\simeq T_{r,k}^{(1)}\simeq T_{r,k}^{(1)}/\sim,
            \quad
            (t_1,\cdots,t_k)
            \mapsto
            [t_1,\cdots,t_k].
            \]
        \end{corollary}

\appendix
\section{Some lemmas on homomorphisms of (commutative) group schemes}
	Let $S$ be a scheme. 
 
 \begin{lemma}\label{hom of lattices}
     Let $Y$ and $Y'$ be group schemes over $S$ which are \'etale locally isomorphic to a free abelian group of finite rank, and let $f:Y\to Y'$ be a homomorphism. Assume that for any $s\in S$, the fiber $f_{\bar{s}}$ of $f$ at a geometric point $\bar{s}$ above $s$ is injective with finite cokernel. Then $f$ is injective with cokernel a finite locally constant group scheme.
 \end{lemma}
 \begin{proof}
     This follows from \cite[\href{https://stacks.math.columbia.edu/tag/093S}{Tag 093S} (3)]{Stacks-project}.
 \end{proof}
 
 Let $B$ and $B'$ (resp. $T$ and $T'$) be abelian schemes (resp. quasi-isotrivial tori) over $S$, and let $G$ (resp. $G'$) be a group scheme over $S$ which is an extension of $B$ (resp. $B'$) by $T$ (resp. $T'$).

 \begin{lemma}\label{hom of tori}
     Let $h:T\to T'$ be a homomorphism of tori over $S$. Assume that for any $s\in S$, $h_{\bar{s}}:T_{\bar{s}}\to T'_{\bar{s}}$ is an isogeny. Then $h$ is faithfully flat (hence surjective) with kernel finite locally free.
 \end{lemma}
 \begin{proof}
     Let $X$ (resp. $X'$) be the character group of $T$ (resp. $T'$). Then $X$ and $X'$ are \'etale locally isomorphic to a free abelian group of finite rank, in particular they are quasi-isotrivial ``groupe constant tordu'' in the sense of \cite[\'Expose X, Definition 5.1]{DemazureGrothendieck1970}. Let $h^*:X'\to X$ be the homomorphism induced by $h$. It suffices to show that $h^*$ is injective with cokernel a finite locally constant group scheme by \cite[\href{https://stacks.math.columbia.edu/tag/05B2}{Tag 05B2}]{Stacks-project}. By the assumption in the statement, $(h^*)_{\bar{s}}$ is injective with finite cokernel for any $s\in S$, then we are done by Lemma \ref{hom of lattices}.
 \end{proof}

 \begin{lemma}\label{hom of semiabelian schemes}
     Let $g:G\to G'$ be a homomorphism of group schemes over $S$. 
     \begin{enumerate}
         \item The composition $T\hookrightarrow G\xrightarrow{g} G'\to B'$ is trivial, and thus $g$ induces $g_{\mathrm{t}}:T\to T'$ and $g_{\mathrm{ab}}:B\to B'$.
         \item Assume that for any $s\in S$, $g_{\bar{s}}:G_{\bar{s}}\to G'_{\bar{s}}$ is an isogeny. Then $g$, $g_{\mathrm{t}}$ and $g_{\mathrm{ab}}$ are all faithfully flat (hence surjective) with kernel finite locally free. 
     \end{enumerate}
 \end{lemma}
 \begin{proof}
     (1) By \cite[Lemma 1.2.1]{Bertolin2009} the composition $T\hookrightarrow G\xrightarrow{g} G'\to B'$ is trivial, and thus $g$ induces $g_{\mathrm{t}}:T\to T'$ and $g_{\mathrm{ab}}:B\to B'$ as depicted in the commutative diagram
     \[\xymatrix{
     0\ar[r] &T\ar[r]\ar[d]_{g_{\mathrm{t}}} &G\ar[r]\ar[d]^{g} &B\ar[r]\ar[d]^{g_{\mathrm{ab}}} &0 \\
     0\ar[r] &T'\ar[r] &G'\ar[r] &B'\ar[r] &0
     }.\]
     
     (2) Let $s$ be any point of $S$. Applying the snake lemma to the base change of the above diagram to $\bar{s}$, we get an exact sequence
     \[0\to\ker((g_{\mathrm{t}})_{\bar{s}})\to \ker(g_{\bar{s}})\to \ker((g_{\mathrm{ab}})_{\bar{s}})\to \mathrm{coker}((g_{\mathrm{t}})_{\bar{s}})\to 0\]
     and $\mathrm{coker}((g_{\mathrm{ab}})_{\bar{s}})=0$. Since the reduced connected component of $\ker((g_{\mathrm{ab}})_{\bar{s}})$ is an abelian variety which is proper, and $\mathrm{coker}((g_{\mathrm{t}})_{\bar{s}})$ is a torus which is affine, we must have $\mathrm{coker}((g_{\mathrm{t}})_{\bar{s}})=0$. Therefore both $(g_{\mathrm{t}})_{\bar{s}}$ and $(g_{\mathrm{ab}})_{\bar{s}}$ are isogenies.

     By Lemma \ref{hom of tori}, $g_{\mathrm{t}}$ is faithfully flat with kernel finite locally free. Thus we have a short exact sequence $0\to\ker(g_{\mathrm{t}})\to \ker(g)\to \ker(g_{\mathrm{ab}})\to0$ with $\ker(g_{\mathrm{t}})$ a finite locally free group scheme over $S$.

     Since $(g)_{\bar{s}}$ (resp. $(g_{\mathrm{ab}})_{\bar{s}}$) is flat for any $s\in S$, so is $g$ (resp. $g_{\mathrm{ab}}$) by \cite[\href{https://stacks.math.columbia.edu/tag/039E}{Tag 039E}]{Stacks-project}. Therefore $g$ and $g_{\mathrm{ab}}$ are faithfully flat. 
     
     Since $\ker(g_{\mathrm{ab}})$ is both proper and quasi-finite, it is finite by \cite[\href{https://stacks.math.columbia.edu/tag/02LS}{Tag 02LS}]{Stacks-project}. Since $B\to S$ and $B'\to S$ are locally of finite presentation, so is $g_{\mathrm{ab}}$ by \cite[\href{https://stacks.math.columbia.edu/tag/02FV}{Tag 02FV}]{Stacks-project} (1). Then $\ker(g_{\mathrm{ab}})$ is locally of finite presentation over $S$ by \cite[\href{https://stacks.math.columbia.edu/tag/01TS}{Tag 01TS}]{Stacks-project}. Therefore $\ker(g_{\mathrm{ab}})$ is finite locally free by \cite[\href{https://stacks.math.columbia.edu/tag/02KB}{Tag 02KB}]{Stacks-project}. Since $\ker(g)$ is a $\ker(g_{\mathrm{t}})$-torsor over $\ker(g_{\mathrm{ab}})$, it is finite locally free over $\ker(g_{\mathrm{ab}})$. It follows that $\ker(g)\to S$ which is the composition $\ker(g)\to \ker(g_{\mathrm{ab}})\to S$ is finite locally free.
     \end{proof}

\end{document}